\documentclass[12pt]{amsart}
\usepackage[margin=1.0in]{geometry}
\usepackage{amsthm}
\usepackage{enumerate}

\usepackage{amsmath,amssymb,color,mathtools}

\usepackage{faktor}
\usepackage{mathrsfs}
\usepackage{hyperref}

\newcommand{\RR}{\mathbb{R}}
\newcommand{\NN}{\mathbb{N}}

\newcommand{\ZZ}{\mathbb{Z}}
\newcommand{\CC}{\mathbb{C}}

\newcommand{\II}{\mathbf{I}}

\newcommand{\mM}{\mathcal{M}}

\DeclareMathOperator*{\res}{Res}
\DeclareMathOperator*{\re}{Re}
\DeclareMathOperator*{\im}{Im}

\DeclareMathOperator{\sign}{sign}

\newcommand{\brackets}[1]{\left(#1\right)}

\numberwithin{equation}{section}

\newtheorem{thm}{Theorem}[section] % reset theorem numbering for each chapter

\newtheorem{lem}[thm]{Lemma}
\newtheorem{prop}[thm]{Proposition}
\newtheorem{conj}[thm]{Conjecture}
\newtheorem{cor}[thm]{Corollary} % same for example numbers
\newtheorem{rmrk}[thm]{Remark}

\newcommand{\Addresses}{{
		\bigskip
		\footnotesize

  \textsc{D{\'e}partment  de Math{\'e}matiques et Statistique,   Universit{\'e} de Montr{\'e}al, Montr{\'e}al, QC  H3C 3J7, Canada}
  \par\nopagebreak
		\textit{E-mail address:}
\texttt{valkovaleva42@gmail.com}
}}

\mathtoolsset{showonlyrefs}

\begin{document}
	
\title{Correlations of the squares of the Riemann zeta on the critical line}
\author{Valeriya Kovaleva}
\maketitle

\begin{abstract}
We compute the average of a product of two shifted squares of the Riemann zeta on the critical line with shifts up to size $T^{3/2-\varepsilon}$. We give an explicit expression for such an average and derive an approximate spectral expansion for the error term similar to Motohashi's. As a consequence, we also compute the $(2,2)$-moment of moment of the Riemann zeta, for which we partially verify (and partially refute) a conjecture of Bailey and Keating.
\end{abstract}

\section{Introduction}
{\noindent The behaviour of the Riemann zeta function $\zeta(s)$ on the critical line $\{1/2+it: t \in \RR\}$ is one of the central topics in analytic number theory. Possible questions vary from the Lindel\"{o}f hypothesis to computing moments, and more recently, probabilistic problems such as studying the distribution of $\zeta(1/2 + i\tau + ih)$, where $\tau$ is random and $h$ is fixed. }

Let $k \ge 0$, then the $2k$-moment of the Riemann zeta on the critical line is the average
\begin{equation*}
    \mM_{2k} (T) = \int_0^T |\zeta(1/2+it)|^{2k} dt\,.
\end{equation*}
According to a folklore conjecture $   \mM_{2k} (T) \sim c_k T \log^{k^2} T$, though it is only known to hold for $k=1,2$ due to Hardy and Littlewood \cite{hardy1916contributions} and Ingham \cite{ingham1928mean} respectively, both results now century-old. It is also conjectured that a stronger structural statement might hold that
\begin{equation*}
    \mM_{2k} (T) =T P_{k^2}(\log T/2\pi) + O_k(T^{1/2+\varepsilon})\,,
\end{equation*}
where $P_{k^2}$ is a polynomial of degree $k^2$. 

Various authors over the years worked on refining this conjecture. Using random matrix theory, Keating and Snaith \cite{keating2000random} conjectured that the leading coefficient of $P_{k^2}$ equals $a_k g_k$, where $a_k$ is the arithmetic constant coming from primes, and $g_k$ is the geometric constant. Conrey, Farmer, Keating, Rubinstein and Snaith \cite{conrey2008lower} extended this conjecture to lower order terms. Conrey and Ghosh \cite{conrey1998conjecture}, and Conrey and Gonek \cite{conrey2001high} obtained the  same conjectures for $k=3,4$ using a number-theoretic heuristic. Ng \cite{ng2016sixth} and Ng, Shen and Wong \cite{ng2022eighth} gave a rigorous justification for this heuristic conditionally on certain error size for divisor function correlations. Conrey and Keating devoted a series of papers \cite{conrey2015moments1,conrey2015moments2,conrey2015moments3,conrey2016moments4,conrey2019moments5} to developing the heuristic based on divisor function correlations for all $k\in \NN$, and Baluyot and Conrey \cite{baluyot2022moments} reexamined this approach using Mellin transforms. Unconditional lower bounds for all rational $k \ge 0$ are due to Heath-Brown \cite{heath1981fractional}, for all real $k > 0$ are due to Radziwi\l{}\l{} and Soundararajan \cite{radziwill2013continuous}, and Heap and Soundararajan \cite{heap2020lower}. Harper \cite{harper2013sharp} obtained a matching upper bound conditionally on the Riemann Hypothesis improving on the work of Soundararajan \cite{soundararajan2009moments} for all real $k \ge 0$.

\subsection{Shifted averages}
As a generalisation of moments, one may also study shifted averages
\begin{equation*}\label{eq:shifted_prod}
    \mM_{2k}(T;\overline{\alpha},\overline{\beta}) = \int_0^T \prod_{j=1}^k \zeta(1/2 + it + i\alpha_j) \zeta(1/2 - it - i\beta_j) dt\,
\end{equation*}
with shifts $\overline{\alpha}= (\alpha_1,\ldots,\alpha_k)\in \RR^k$ and $\overline{\beta} = (\beta_1,\ldots,\beta_k)\in \RR^k$. Conrey, Farmer, Keating, Rubinstein and Snaith \cite{conrey2005integral} conjectured that for all $k\in \NN$
\begin{equation}\label{eq:m2kab}
\begin{aligned}
 \mM_{2k}(T;\overline{\alpha},\overline{\beta}) =  \int_0^T Q_k(t; \overline{\alpha},\overline{\beta}) dt + O_k(T^{a+\varepsilon})\,,
\end{aligned}
\end{equation}
where $a < 1$ and $Q_k(t;\overline{\alpha},\overline{\beta})$ is a certain multiple contour integral in $z_1,\ldots,z_{2k}$ involving $\zeta(1 +z_i-z_j)$ and the Vandermonde determinant $\prod_{i<j} (z_i-z_j)$. In particular, $Q_k(t;\mathbf{0},\mathbf{0})$ is a polynomial of degree $k^2$ in $\log t/2\pi$, but in general the explicit expression for $Q_k$ is more involved. Ingham \cite{ingham1928mean} established \eqref{eq:m2kab} for $k=1$ when $\alpha_j,\beta_j \ll 1$, and Bettin \cite{bettin2010second} extended this result to $\alpha_j,\beta_j \ll T^{2-\varepsilon}$. The case of $k=2$ with $\alpha_j,\beta_j \ll 1/\log T$ is due to Motohashi \cite{motohashi1993explicit}, and the range restriction essentially comes from the correlations of the generalised divisor functions. This result predated any of the generalised conjectures, and in fact inspired them. Remarkably, Motohashi also gave an explicit expansion for the error term using spectral theory of automorphic forms.

The upper and lower bounds in the general case are more difficult to understand, and are currently only known for
\begin{equation*}
\mM_{2\mathbf{k}}(T;\overline{\alpha}) = \int_0^T \prod_{j=1}^m |\zeta(1/2 + it + i\alpha_j)|^{2k_j} dt, \quad k_j \in \ZZ^{\ge 0}\,.
\end{equation*}
In this case, Chandee \cite{chandee2011correlation} proved that for $\alpha_j \ll \log\log T$ and $\alpha_i-\alpha_j \ll 1$
\begin{equation}\label{eq:chandeebounds}
    T C_{\alpha, T} (\log T)^{k_1^2+\ldots k_m^2} 
 \ll_{\mathbf{k}} \mM_{2\mathbf{k}}(T;\overline{\alpha}) \ll_{\mathbf{k},\varepsilon} T C_{\alpha,T} (\log T)^{k_1^2+\ldots k_m^2 + \varepsilon}\,,
\end{equation}
where $C_{\alpha,T} = \prod_{i<j} \min\{|\alpha_i-\alpha_j|^{-1},\log T\}^{2k_ik_j}$ and the upper bound is conditional on RH. Further, Chandee conjectured that for $\mathbf{k} = (k,k)$ and $\overline{\alpha} = (\alpha_1,\alpha_2)$ such that $\alpha_j \ll \log T$ and $\alpha_1-\alpha_2 \ll 1$
\begin{equation}\label{eq:chandee}
\mM_{2\mathbf{k}}(T;\overline{\alpha})  \begin{cases}
\asymp_k  T \log^{k^2} T , & \text{if $\lim_{T\to \infty} |\alpha_1-\alpha_2| \log T = 0$}\,;\\
\asymp_{k,b}  T \log^{k^2} T, & \text{if $\lim_{T\to \infty} |\alpha_1-\alpha_2| \log T = b \neq 0$}\,;\\
\asymp_k T |\alpha_1-\alpha_2|^{-k^2/2} \log^{k^2/2} T, & \text{if $\lim_{T\to \infty} |\alpha_1-\alpha_2| \log T = \infty$}\,.
\end{cases}
\end{equation}
Ng, Shen and Wong \cite{ng2022shifted} proved the upper bound in \eqref{eq:chandee} under RH for $|\alpha_1-\alpha_2|\le T^{0.6}$ with some adjustment to the multiplicative constant. More recently, Curran \cite{curran2023correlations} improved the upper bound in \eqref{eq:chandeebounds} (under RH) both in terms of the range and the order of magnitude with the new bound
\[
\mM_{2\mathbf{k}}(T;\overline{\alpha}) \ll_{\mathbf{k}} T(\log T)^{k_1^2+\ldots k_m^2} \prod_{i<j} \zeta(1 + i(\alpha_i-\alpha_j) + 1/\log T)^{2k_ik_j}\,, \quad |\alpha_j| \le T/2\,,
\]
likely to be sharp.

In this paper, with some abuse of notation, we consider 
\begin{equation*}
     \mM_4(T; \alpha,\beta) =  \int_0^{T} |\zeta(1/2+it+i\alpha)|^2 |\zeta(1/2+it+i\beta)|^2 dt\,
\end{equation*}
for $\alpha,\beta \in \RR$, and verify the conjectures of Motohashi and Chandee for such an average. 

\begin{thm}\label{thm:A} Let $T \ge 10$ be large, let $\alpha,\beta \ge 0$ be such that $\delta := \beta-\alpha \ge 0$. Let $\gamma_0$ be the Euler-Mascheroni constant, and set
\begin{equation*}
    h(z,s) := \frac{e^{2\gamma_0 s}}{\zeta(2+2s)} \brackets{\zeta(1+z + s)\zeta(1-z+s) - \frac{2s\zeta(1+2s)}{s^2-z^2}}\,.
\end{equation*}
Then $\mM_4(T; \alpha,\beta) =  \int_1^T (D(t;\alpha,\beta) + OD(t;\alpha,\beta))dt + E(T;\alpha,\beta)$, where
\begin{equation*}
    \begin{aligned}
   & D(t;\alpha,\beta) = 2\re \sum_{s \in \{0,i\delta\}} \res_s \frac{\zeta^4(1+s)}{\zeta(2+2s)}  \frac{(t/2\pi)^s}{s-i\delta}\,;\\
   & OD(t;\alpha,\beta) = \frac{\partial^2}{\partial s_1 \partial s_2}\, \brackets{h(i\delta,s_1+s_2) \brackets{ \frac{t+\alpha}{2\pi} }^{s_1}\brackets{\frac{t+\beta}{2\pi} }^{s_2}} \bigg\vert_{s_1=s_2=0}  \,.
    \end{aligned}
\end{equation*}
and $E(T;\alpha,\beta) \ll O((T+\alpha)^{2/3+\varepsilon} + (T+\alpha)^{1/2+\varepsilon}\delta^{1/3} + \delta^{1/2})$.
\end{thm}

\begin{rmrk}
    When the moment is taken with a smooth weight, we obtain an approximate spectral expansion for the error term, see Theorem \ref{thm:explicit_small} and Theorem \ref{thm:explicit_large}.
\end{rmrk}

The main term in our result gives an expression identical to Motohashi's integral and can also be seen as a generalisation of the expression given by Conrey \cite{conrey1996note}. It agrees with Chandee's conjecture for $k=2$, and illustrates the sharpness of Curran's version of the upper bound when $|\alpha-\beta| \to \infty$ as $T \to \infty$. Additionally, when the moment is taken with a smooth weight, we can represent our error term in terms of a certain spectral series, similar to Motohashi's explicit formula in \cite{motohashi1993explicit}. Thus, our bound on the error term relies on shifted moments of Hecke-Maass $L$-functions
\[
\sum_{|\kappa_j| \le X} w(\kappa_j) H_j^2(1/2) H_j(1/2+i\delta)\,,
\]
where $H_j(s)$ is the $L$-function attached to the $j$-th Maass form with respective eigenvalue $\lambda_j = 1/4 + \kappa_j^2$. We note that the error term of size $O(T^{2/3+\varepsilon})$ should be compared to the error term for the fourth moment of the Riemann zeta \cite{ivic1995fourth}. Improving this term, at least naively, would involve utilising oscillation of the type $t^{i\kappa_j}$ in the sum above, which we are currently unable to do. The error term of size $T^{1/2+\varepsilon}\delta^{1/3}$ should be compared to the subconvexity bounds for Hecke-Maass $L$-functions \cite{jutila2005uniform}.

\begin{cor} Without loss of generality let $\alpha =0$ and $\beta = \delta \ll T^{3/2-\varepsilon}$, then
\begin{equation*}
    \mM_4(T; 0,\delta) \sim \begin{cases}
        \frac{1}{2\pi^2} \log^4 \frac{T}{2\pi}, & \text{if $\lim_{T\to \infty} (\delta \log T) = 0$};\\
        \frac{1}{\zeta(2)b^2}\brackets{1 - \frac{\sin^2(b/2)}{(b/2)^2}} \log^4 \frac{T}{2\pi}, & \text{if $\lim_{T\to \infty} (\delta \log T) = b\in \RR$};\\
        % \frac{1}{\zeta(2)\delta^2} \log^2 \frac{T}{2\pi}, & \text{if $\delta = o(1)$};\\
        \frac{|\zeta(1+i\delta)|^2}{\zeta(2)}\log \frac{T}{2\pi}\log \frac{T+\delta}{2\pi}, & \text{if $\lim_{T\to \infty} (\delta \log T) = \infty$}\,.
    \end{cases}
\end{equation*}    
\end{cor}
Let us also discuss a related problem known as moments of moments of the Riemann zeta.

\subsection{Moments of moments}
Let $k, \gamma \ge 0$ and $c > 0$, then the $(k,2\gamma)$-moment of moment is an average of a short average of length $2c$ around a point in $[0,T]$
\begin{equation*}
    \mM_{k,2\gamma} (T) := \int_{0}^{T} \brackets{\frac{1}{2c}\int_{-c}^{+c} |\zeta(1/2+it+ih)|^{2\gamma}dh}^kdt\,.
\end{equation*}
There has recently been a lot of interest in studying the connection between the Riemann zeta and the Gaussian multiplicative Chaos (GMC). In this context, 
\[
\frac{1}{2c}\int_{-c}^{+c} |\zeta(1/2+it+ih)|^{2\gamma}dh
\]
is an analogue of GMC with inverse temperature $2\gamma$, and $\mM_{k,2\gamma} (T)$ occur naturally as the analogues of the moments of GMC. The overarching theme of comparing the Riemann zeta to the characteristic polynomial of a random unitary matrix and to the integrand in the GMC model is the correlation structure of their logarithms. In particular, for $\log |\zeta(1/2+i\tau+ih)|$ this means that if $\tau \sim U([T,2T])$, the covariance at $h_1$ and $h_2$ is logarithmic in $d = |h_1-h_2|$ in the appropriate regime $1/\log T \ll d \ll 1$ \cite{bourgade2010mesoscopic,harper2013note}. Using the framework of log-correlated fields Fyodorov, Hiary, and Keating \cite{fyodorov2012freezing} made a very precise conjecture on the distribution of $\max_{|h|\le 1} \log |\zeta(1/2+i\tau+ih)|$ (partially verified in \cite{najnudel2018extreme, harper2019partition, arguin2019maximum, arguin2020fyodorov}) as well as the leading order of $\mM_{k,2\gamma} (T)$. Following \cite{fyodorov2012freezing, fyodorov2014freezing}, Bailey and Keating \cite{bailey2021moments} derived the following conjecture from the conjecture for shifted moments \eqref{eq:m2kab}.
\begin{conj}[Bailey and Keating]\label{conj:BK} Let $k,\gamma \in \NN$ and $c > 0$ be fixed, then
\begin{equation}\label{eq:baileykeating}
        \mM_{k,2\gamma} (T) = a_{k,\gamma} g_{k,\gamma} T \brackets{\log T}^{k^2\gamma^2 -k + 1} \brackets{1 + O_{k,\gamma}\brackets{(\log T)^{-1}}}\,,
\end{equation}
where $a_{k,\gamma}$ is the arithmetic constant, and $g_{k,\gamma}$ is the geometric constant. Further, when suitably smoothed, we expect 
\begin{equation}\label{eq:baileykeatingstrong}
T P_{k^2\gamma^2 -k + 1}(\log T/2\pi) + O_{k,\gamma}(T^{a})\,,
\end{equation}
where $P_{k^2\gamma^2 -k + 1}(x)$ is a polynomial of degree $k^2\gamma^2 -k + 1$ and $a < 1$.
\end{conj}
This conjecture does not extend to all $k,\gamma > 0$ as one expects a freezing transition at $k \gamma^2=~1$ similarly to what occurs in random matrix theory and the theory of GMC \cite{fyodorov2012freezing,fyodorov2014freezing}. 

When $k \in \NN$, $\mM_{2k,\gamma}(T)$ is a $k$-fold average of a shifted average, so integrating the expression in Theorem \ref{thm:A}, we compute
\begin{equation*}
\mM_{2,2}(T;g) := \int_0^T \brackets{\int_{-\infty}^{+\infty}  g(h) |\zeta(1/2+it+ih)|^{2}dh}^2 dt\,.  
\end{equation*}

\begin{cor}\label{cor:mompiind} Let $T \ge 10$ and let $N = \log T/2\pi$. Then for $g(x) = \II\{|x| \le \pi\}/2\pi$
\[
\mM_{2,2}(T;g) = P_3(N) +  P_2(N) \log N + a N^{-1} + O(T N^{-2} + T^{2/3+\varepsilon}) \sim \frac{1}{\pi^2} \log ^3 N\,,
\]
where $P_j(N)$ are polynomials of degree $j$ with non-zero leading coefficients, and $a \approx 0.46 \neq 0$.
\end{cor}

Thus, we show that Conjecture \ref{conj:BK} does not quite hold in this case as $\mM_{2,2}(T;g)$ has a lower order term of size $\log^2T \log\log T$, and the error term cannot be of polynomial size. The source of both issues is the diagonal term, which can essentially be written as
\[
\begin{aligned}
       \sum_{n \le T/2\pi} d^2(n)\int_{1}^{\frac{T}{2\pi n}} \hat{g}^2\brackets{\frac{1}{2\pi} \log w} dw = \frac{1}{2\pi i} \int_{(1)} \frac{\zeta^4(1+s)}{(s+1)\zeta(2+2s)}\brackets{\frac{T}{2\pi}}^s \ \mathcal{L}\hat{g}^2(s) ds\,,
\end{aligned}
\]
where $\mathcal{L}\hat{g}^2(s)$ is the Laplace transform of $\hat{g}^2$.  For a general function $g$, if $\mathcal{L}\hat{g}^2(s)$ cannot be extended analytically beyond $\re s \ge 0$, one may not be able to obtain a power saving via the standard trick of moving the line of integration, and integrating along branch cuts (if there are any) generally does not yield the same result. This turns out to be an unavoidable problem for $g(x) = \II\{|x| \le \pi\}/2\pi$ as
\[
    \mathcal{L}\hat{g}^2(s) = -\frac{1}{4\pi^2}\brackets{(s+2\pi i) \log \brackets{s+ 2\pi i}+ (s-2\pi i) \log \brackets{s-  2\pi i} - 2s \log s}\,.
\]

We can, however, obtain the second case in Conjecture \ref{conj:BK} by choosing an appropriate function $g(x)$.

\begin{cor}\label{cor:mompian}
Let $g(x) = \pi^{-1}g_0(x/\pi)$, where $g_0$ is a normalised smooth even function such that $\hat{g}_0(y/2\pi) \ll e^{-6Ay}$ for some $A > 0$. Then
\[
\begin{aligned}
\mM_{2,2}(T;g) = P_3(N;g_0) + O(T^{1-A + \varepsilon} + T^{2/3+\varepsilon}) \sim \frac{1}{\pi^2} \log ^3 N\,,
\end{aligned}
\]
where $P_3(N;g_0)$ is a polynomial of degree $3$ depending on $g_0$.
\end{cor}

\begin{rmrk}
    We state the result for $\mM_{2,2}(T;g)$ when $c= \pi$ to showcase its most unusual behaviour. We do, however, compute $\mM_{2,2}(T;g)$ in a wider range of $c = c(T)$. For the general statement and a more detailed discussion of its structure see Section \ref{sec:moms}.
\end{rmrk}

\subsection{Outline of the paper} 
We start with presenting some basic tools and background for spectral theory of Hecke-Maass forms and divisor sum correlations in Section \ref{section:conj}. The purpose of this section is purely technical and contains no new results. In Section \ref{section:ground} we lay the groundwork for the moment computation by transforming the initial problem into a tangible form. Namely, in Lemma \ref{lem:first_approx} we represent $\zeta^2(1/2+it+i\alpha)^2 \zeta^2(1/2-it-i\beta)$ as a weighted double Dirichlet series of the divisor function $d(n)$ essentially truncated at $mn \ll (T+\alpha)(T+\beta)$, and use it to construct an approximate functional equation for
\[
\int_0^\infty W(t) |\zeta(1/2+it+i\alpha)|^2 |\zeta(1/2+it+i\beta)|^2 dt\,
\]
in Proposition \ref{lem:sweights}. The function $W(t)$ is chosen to be either a bump function on a dyadic interval, or a function with sufficiently fast decay outside of the interval. The approximate functional equation depends on two parameters: $Q = Q(T)$ regulating the Dirichlet series, and $\Delta = \Delta(T)$ regulating the weight function.

We prove the main result in Sections \ref{sec:diag}, \ref{sec:small}, and \ref{sec:large}. We split the Dirichlet series into the diagonal part $D(T;\alpha,\beta)$ and the off-diagonal part $OD(T;\alpha,\beta)$. Computing
\[
 D(T;\alpha,\beta) \approx \sum_{n \le T/2\pi} \frac{d^2(n)}{n} \int_{2\pi n}^T W(t)\brackets{\frac{t}{2\pi n}}^{i(\alpha-\beta)} dt
\]
is trivial; we do this in Section \ref{sec:diag} resulting in Proposition \ref{thm:diag}. The bulk of the computation falls upon Sections \ref{sec:small} and \ref{sec:large}, where we evaluate the off-diagonal contribution
\[
OD(T;\alpha,\beta) \approx \sum_{n \neq m} \frac{d(n)d(m)}{n^{1/2+i\alpha}m^{1/2-i\beta}} \int_0^\infty W(t)\brackets{\frac{m}{n}}^{it} \brackets{\frac{t+\alpha}{2\pi e}}^{i(t+\alpha)}\brackets{\frac{t+\beta}{2\pi e}}^{-i(t+\beta)} dt
\]
in two regimes depending on the size of $\delta = \beta - \alpha$. 

We avoid dealing with generalised divisor functions by working with $\zeta^2(1/2+it + i\alpha)$ and $\zeta^2(1/2-it - i\beta)$ rather than four copies of zeta at the cost of an extra factor of
\[
\chi\brackets{1/2-it} \chi\brackets{1/2 + i(t +\delta)} \approx \brackets{\frac{t+\alpha}{2\pi e}}^{i(t+\alpha)}\brackets{\frac{t+\beta}{2\pi e}}^{-i(t+\beta)}\,.
\] 
When $\delta = \beta - \alpha$ is small,
\[
\int_0^\infty W(t)\brackets{\frac{m}{n}}^{it} \brackets{\frac{t+\alpha}{2\pi e}}^{i(t+\alpha)}\brackets{\frac{t+\beta}{2\pi e}}^{-i(t+\beta)} dt
\]
is similar to a Fourier transform, and when $\delta = \beta - \alpha$ is large, it is essentially a stationary phase integral. In both cases, we reduce the off-diagonal term to
\[
\sum_{n,r \ge 1} d(n) d(n+r) U(n/r)\,,
\]
where $U(x)$ is smooth function (approximately) supported on a dyadic interval. After that we apply Motohashi's explicit formula for the binary additive divisor problem to obtain an approximate spectral expansion of the error term, and bound it in Propositions \ref{lem:smalldelta_correr} and \ref{lem:bigdelta_correr}. 

We gather all the parts and state the (approximate) explicit expression for $\mathcal{M}_{4}(T;\alpha,\beta)$ in Theorem \ref{thm:explicit_small} for $|\alpha-\beta| \ll T^{2/3-\varepsilon}$, and Theorem \ref{thm:explicit_large} for $|\alpha-\beta| \gg T^{2/3-\varepsilon}$. In each of the sections we choose the parameters $Q$ and $\Delta$ separately to optimise the error term resulting in Theorem \ref{thm:A}. Finally, in Section \ref{sec:moms} we integrate the expression given by Theorem \ref{thm:A} to compute $\mathcal{M}_{2,2}(T)$, and provide a more detailed discussion of the structure of $\mathcal{M}_{2,2}(T)$.

\subsection*{Notation} Throughout the paper we use standard Vinogradov notation $\ll,\gg,\sim,\asymp$. Constant $\gamma_0$ is always the Euler-Mascheroni constant. $\Gamma(z)$ is the Gamma function, $B(z_1,z_2)$ is the Beta function, ${}_2F_1(a,b,c;x)$ is the hypergeometric function. The divisor function is denoted $d(n)$, and $\sigma_z(n) = \sum_{k|n} k^z$. For the Fourier transform we use normalisation  $\mathcal{F}f(\xi) =  \hat{f}(\xi) = \int_{-\infty}^{\infty} f(x) e^{-2\pi i x \xi} dx$. 

\subsection*{Acknowledgements}
The author was supported by ERC LogCorRM (grant no 740900), and by CRM-ISM postdoctoral fellowship. Part of the work was carried out at University of Oxford and submitted as a part of the doctoral thesis. The author is grateful to Jon Keating for bringing up the problem and for his guidance during the project.

\section{General moment conjectures and Motohashi's integral} \label{section:conj}
\subsection{General conjecture} Let $\overline{\alpha} = (\alpha_1, \ldots, \alpha_k)\in \RR^k$, and $\overline{\beta} = (\beta_1, \ldots, \beta_k)\in \RR^k$. Let $\overline{\gamma} = (\alpha_1, \ldots, \alpha_k, \beta_1, \ldots, \beta_k) \in \RR^{2k}$ be their concatenation. We will omit the bars and simply write $\alpha,\beta,\gamma$ throughout this particular section. Let $l_1, \ldots, l_k \in \CC$, then the generalised divisor function is defined as
\[
\tau_{(l_1,\ldots,l_k)}(n) := \sum_{n_1\cdots n_k = n} n_1^{l_1} \cdots n_k^{l_k}\,.
\]
Let $\mathbf{p}^{-i\alpha} = (p^{-i\alpha_1}, \ldots, p^{-i\alpha_k})$ and $\mathbf{p}^{i\beta} = (p^{i\beta_1},\ldots,p^{i\beta_k})$, and define
\begin{equation*}
    \zeta_\gamma (s) := 
    \sum_{n=1}^{\infty} \frac{\tau_{-i\alpha}(n) \tau_{i\beta}(n)}{n^{2s}} = \prod_{\text{$p$ prime}} \sum_{m=0}^\infty \frac{h_m(\mathbf{p}^{-i\alpha}) h_m(\mathbf{p}^{i\beta})}{p^{2ms}} = F(s;\alpha,\beta) \prod_{j,l=1}^{k} \zeta(2s+ i(\alpha_j-\beta_l))\,,
\end{equation*}
where $h_m (\mathbf{x})$ is the $m$-th complete symmetric polynomial at $\mathbf{x} = (x_1,\ldots, x_k)$ and $F(s;\alpha, \beta)$ is an absolutely convergent product for $\re s >1/4$. The product of zetas has an analytic continuation to the whole complex plane with single poles at $s = 1/2 -i(\alpha_j-\beta_l)/2$. 

Recall the functional equation of the Riemann zeta $\zeta (s) = \chi(s)\zeta(1-s)$, where
\begin{equation}\label{eq:chiexplicit}
    \chi(s) = \pi^{s-1/2} \frac{\Gamma((1-s)/2)}{\Gamma(s/2)} = 2^s\pi^{s-1} \sin(\pi s/2)\Gamma(1-s)\,.
\end{equation}
Let $\pi \in S_{2k}$ be a permutation on the coordinates of $\gamma$ with $\pi(\gamma) = (\tilde{\alpha},\tilde{\beta})$. Let
\begin{equation*}
    G_{\gamma}(\pi) : = \prod_{\substack{1\le j \le k\\\tilde{\alpha}_j \notin \alpha, \tilde{\beta}_j \notin \beta}} \chi(1/2-it-i\tilde{\alpha}_j)\chi(1/2+it+i\tilde{\beta}_j) = \prod_{j=1}^k \frac{\zeta(1/2+it+i\alpha_j) \zeta(1/2-it-i\beta_j)}{\zeta(1/2+it+i\tilde{\alpha_j}) \zeta(1/2-it-i\tilde{\beta}_j)}\,.
\end{equation*}
and, summing over all permutations in $S_{2k}/(S_k \times S_k)$, define
\begin{equation}\label{defn:Qktab}
    Q_k(t;\alpha,\beta) : = \sum_{\pi \in S_{2k}/(S_k \times S_k)}  G_{\gamma}(\pi) \zeta_{\pi(\gamma)}(1/2)\,.
\end{equation}
One can check that $Q_k$ has removable singularities at $\alpha_j=\beta_l$, so is essentially entire. 

\begin{conj}[CFKRS \cite{conrey2005integral}]\label{conj:CFKRS}
Let $k \in \NN$ and let $\alpha, \beta \in \RR^k$, then
\begin{equation*}
    \begin{aligned}
% \mathcal{M}_{2k}(T;\alpha,\beta) = 
\int_0^T \prod_{j=1}^k \zeta(1/2 + it + i\alpha_j) \zeta(1/2 - it - i\beta_j) dt =  \int_0^T Q_k(t;\alpha,\beta) dt + O(T^{a+\varepsilon})\,
    \end{aligned}
\end{equation*}
for some $a < 1$.
\end{conj}

\subsection{Motohashi's integral}
Motohashi's work \cite{motohashi1993explicit} implies Conjecture \ref{conj:CFKRS} for $k=2$ provided that $\alpha_j,\beta_j \ll 1/\log T$. In this case, for distinct $\alpha_j$ and $\beta_j$
\begin{equation*}
 \zeta_\gamma(1/2)= \frac{\zeta(1+i\alpha_1 - i\beta_1) \zeta(1+i\alpha_1-i\beta_2)\zeta(1+i\alpha_2-i\beta_1)\zeta(1+i\alpha_2-i\beta_2)}{\zeta(2+i\alpha_1-i\beta_1 + i\alpha_2 -i\beta_2)}\,.
\end{equation*}
Setting $Z(\alpha_1,\alpha_2,\beta_1,\beta_2) : = \zeta_\gamma(1/2)$, and $\kappa_{x,y}(t): = \chi(1/2+it+ix)\chi(1/2-it-iy)$, 
\begin{equation}\label{eq:conjQ4}
    \begin{aligned}
    Q_2(t;\alpha,\beta) = Z(\alpha_1,\alpha_2,\beta_1,\beta_2) +\kappa_{\alpha_1,\beta_1}(t)\kappa_{\alpha_2,\beta_2}(t)Z(\beta_1,\beta_2,\alpha_1,\alpha_2)+\\ \kappa_{\alpha_1,\beta_1}(t)Z(\beta_1,\alpha_2,\alpha_1,\beta_2) +\kappa_{\alpha_2,\beta_2}(t)Z(\alpha_1,\beta_2,\beta_1,\alpha_2)  +\\ \kappa_{\alpha_2,\beta_1}(t)Z(\alpha_1,\beta_1,\alpha_2,\beta_2) +  \kappa_{\alpha_1,\beta_2}(t)Z(\beta_2,\alpha_2,\beta_1,\alpha_1)\,.
    \end{aligned}
\end{equation}

Let us now consider a particular case when $\alpha_j=\beta_j$ for $j=1,2$. Let $\alpha := \alpha_1$, $\beta := \alpha_2$ and $\delta := \beta-\alpha$. In fact, without loss of generality we may assume that $\alpha = 0$ and $\delta = \beta > 0$.

The last two terms in \eqref{eq:conjQ4} immediately give us
\begin{equation*}
    D(i\delta, t) : = \kappa_{\delta,0}(t)L(-i\delta) + \kappa_{0,\delta}(t)L(i\delta)\,,
\end{equation*}
where
\begin{equation*}
    L(z) := \frac{\zeta^4(1+z)}{\zeta(2+2z)}\,.
\end{equation*}
The expression for the first four terms is a bit more involved. Taking the limit in \eqref{eq:conjQ4} as $\alpha_1-\beta_1$ and $\alpha_2-\beta_2$ tend to zero via a residue computation or otherwise, we see that the first four terms in \eqref{eq:conjQ4} give us
\begin{equation}\label{eq:H}
    \begin{aligned}
    H(i\delta,t) := \frac{\partial^2}{\partial u_1 \partial u_2}\, \brackets{h_0(i\delta,u_1+u_2) \kappa_{0, u_1}(t)\kappa_{0,u_2}(t+\delta)}\bigg\vert_{u_j=0}\,,
    \end{aligned}
\end{equation} 
where
\begin{equation}\label{eq:moto_offdiag}
    h_0(z,u)  := \frac{\zeta(1+z + iu)\zeta(1-z+iu)}{\zeta(2+2iu)} e^{2\gamma_0 u}\,.
\end{equation}

For any $|v| \ge 1$
 \begin{equation}\label{eq:chiapprox}
     \chi(1/2+iv) = \exp (-iv\log \frac{|v|}{2\pi e} + \frac{\pi i}{4} \sign v + O(1/|v|))\,,
 \end{equation}
and thus $\kappa_{x,y}(t) = e^{i\phi_{x,y}(t)}$, where
\begin{equation*}
    \phi_{x,y}(t) = -(t+x)\log \frac{t+x}{2\pi e} + (t+y)\log \frac{t+y}{2\pi e} + O(1/|v|)\,,
\end{equation*}
and, if $\mu :=\max\{|\alpha_j|,|\beta_j|\}\ll t^{1/2-\varepsilon}$, $\kappa_{x,y}(t) = (t/2\pi e)^{i(y-x)} (1+ O(\mu|y-x|/|t+x|))$. Further, for $m \in \NN$ and $t > 1$
\begin{equation}\label{eq:chiprime}
        \frac{d^m}{du^m}\kappa_{0, u}(t)\bigg\vert_{u=0} =  \chi(1/2+it)\chi^{(m)}(1/2-it) =
         \log^m \frac{t}{2\pi} + O (1/t) \,.
    \end{equation}

Using \eqref{eq:chiprime} then reduces \eqref{eq:H} to
\begin{equation}\label{eq:conjH}
    \begin{aligned}
    H(i\delta ,t) = 
    \frac{\partial^2}{\partial u_1 \partial u_2}\, \brackets{h_0(i\delta,u_1+u_2) \brackets{\frac{t}{2\pi}}^{u_1} \brackets{\frac{t+\delta}{2\pi}}^{u_2}}\bigg\vert_{u_j=0}  \brackets{1 + O(1/|t|)} \,.
    \end{aligned}
\end{equation}
When $\delta = o(t)$, \eqref{eq:conjH} further reduces to
\[
\frac{d^2}{du^2}\, \brackets{h_0(i\delta,u) \brackets{\frac{t}{2\pi}}^u}\bigg\vert_{u=0} \brackets{1 + O(\delta/t)}= \sum_{j=0}^2 a_j(i\delta) \brackets{ \log t/2\pi + 2\gamma_0}^j \brackets{1 + O(\delta/t)}\,.
\]

While $D(z,t)$ and $H(z,t)$ are singular at $z=0$, their sum $F(z,t) := D(z,t) + H(z,t)$ is entire, and 
\[
F(i\delta,t) = Q_2(t;0,\delta,0,\delta)\,
\]
is identical to the expression in Theorem \ref{thm:A}. Moreover,
\begin{equation*}
 Q_2(t;\mathbf{0}, \mathbf{0}) = \lim_{z \to 0} F(z,t) = \res_{z=0} (z^{-1} F(z,t))= P_4\brackets{\log t/2\pi}\,,   
\end{equation*}
where $P_4(x)$ is a polynomial of degree $4$ giving the fourth moment of the Riemann zeta
\begin{equation}\label{eq:4thmomentwitherror}
    \int_0^T |\zeta(1/2+it)|^4 dt = \int_0^T P_4\brackets{\log t/2\pi} dt + O(T^{2/3+\varepsilon})\,.
\end{equation}
Comparing to Heath-Brown \cite{heath1979fourth} and Conrey \cite{conrey1996note} we observe that the diagonal contribution (of the truncated Dirichlet series) in the fourth moment computation can be written as
\[
2T \res_{z=0} \frac{\zeta^4(1+z)}{z(z+1)\zeta(2+2z)} \brackets{\frac{T}{2\pi}}^z =  \int_0^T \res_{z=0} \brackets{z^{-1}D(z,t)} dt + o(T)\,.
\]
The off-diagonal contribution (of the truncated Dirichlet series) as stated in \cite{conrey1996note} is given by
\[
T \frac{d^2}{ds^2}\, \brackets{\frac{\zeta^2(1+ s) - \frac{2}{s} \zeta(1+2s)}{(s+1)\zeta(2+2s)}\brackets{\frac{Te^{2\gamma_0}}{2\pi}}^{s}}\bigg\vert_{s=0} =  \int_0^T \res_{z=0} \brackets{z^{-1}H(z,t)}dt + o(T)\,.
\]
Thus we can refer to $D(i\delta,t)$ and $H(i\delta,t)$ as the diagonal and the off-diagonal terms, though the correspondence of these terms to specific terms in \eqref{eq:conjQ4} is irrelevant as it is affected by the choice of lines in $\CC^4$ along which we compute the limits of separate groups of terms.

When $\delta \ll 1/\log T$, Motohashi's result \cite{motohashi1993explicit} not only gives the asymptotic, but also represents the error term in \eqref{eq:4thmomentwitherror} as a spectral expansion. This expansion is intimately related to the spectral expansion of the error term for correlations of the divisor function \cite{motohashi1994binary}
\[
\sum_{n \ge 1} d(n) d(n+r) U(n/r)\,.
\]
In our case, it will be convenient to use the latter rather than the former.

\subsection{Basic spectral theory}\label{subsec:spectheory}
Let us introduce some basic spectral theory of modular forms and Maass forms. We will only need notation and some existing results in the area.
% so this does not include any general theory of automorphic forms or Hecke operators. 
For a detailed account of spectral theory of automorphic forms see \cite{iwaniec2021spectral}.

Let $\{\lambda_j = 1/4 + \kappa_j^2\}_{j\ge 1, \kappa_j > 0}$ be the discrete spectrum of the non-Euclidean Laplacian acting on $\text{SL}_2(\ZZ)$-automorphic forms on the upper half plane. The space of all cusp forms admits an orthogonal basis $\{\phi_j\}_{j \ge 1}$ such that $\phi_j$ is an eigenfunction attached to the eigenvalue $\lambda_j$, and simultaneously an eigenfunction of all Hecke operators $\{T(n)\}_{n\ge 1,n=-1}$ with Hecke eigenvalues $t(n)$. This basis is unique up to normalisation, and functions $\phi_j$ are Maass cusp forms attached to respective eigenvalues $\lambda_j$. At $n = -1$ we have $T(-1)\phi_j(z) = \varepsilon_j\phi_j(z)$ with $\varepsilon_j \in \{\pm 1\}$ defining the parity of $\phi_j$. Similarly, consider the space of holomorphic cusp forms of weight $2k$, where $k\ge 3$. Let $\vartheta(k)$ be its dimension and $\{\phi_{j,k}\}_{1\le j\le \vartheta(k)}$ be its orthogonal basis. Simultaneously, $\{\phi_{j,k}\}_{1\le j\le \vartheta(k)}$ are eigenfunctions of the Hecke operators $\{T_j(n)\}_{n \in \NN}$ with corresponding Hecke eigenvalues $\{t_{j,k}(n)\}_{n \in \NN}$. We will be using the normalisation
\[
\begin{aligned}
  \phi_j(x+iy) = \rho_j(1) \sqrt{y} \sum_{n \neq 0} t_j(n) K_{i\kappa_j}(2\pi|n|y) e^{2\pi i nx},\quad \phi_{j,k}(z) = \rho_{j,k}(1) \sum_{n=1}^\infty t_{j,k}(n) n^{k-1} e^{2\pi i n z}\,,
\end{aligned}
\]
where $\rho_j(1), \rho_{j,k}(1)$ are normalisation constants and $K_s(z)$ is the Bessel function of order $s$.

Hecke eigenvalues $t_j(n)$ and $t_{j,k}(n)$ are real multiplicative functions. By the Ramanujan-Petersson conjecture for Maass cusp forms we expect $t_j(n) \ll n^{\varepsilon}$, and the current best bound is $t_j(n) \ll n^{7/64}$ \cite{kim2003functoriality}. In contrast, the bound $|t_{j,k}(n)| \le d(n) \ll n^\varepsilon$ is known to hold \cite{deligne1971formes,deligne1974conjecture}.

The Hecke $L$-functions of $\phi_j$ and $\phi_{j,k}$ are the respective Dirichlet series 
\[
H_j(s) = \sum_{n=1}^{\infty} \frac{t_j(n)}{n^s} \quad (\re s > 2), \quad H_{j,k}(s) = \sum_{n=1}^{\infty} \frac{t_{j,k}(n)}{n^s}\quad (\re s > k+1)\,.
\]
Both $L$-functions extend to $\CC$, are entire, and satisfy functional equations
\begin{equation}
    \begin{aligned}
&H_j(s) = \chi_j(s)H_j(1-s) =\varepsilon_j  \chi(s-i\kappa_j)\chi(s+i\kappa_j)H_j(1-s)\,,\\
&H_{j,k}(s) = \chi_{j,k}(s)H_{j,k}(1-s) = \chi(s+k-1/2) \chi(s-k+1/2)H_{j,k}(1-s)\,.
    \end{aligned}
\end{equation}
% In particular, it immediately follows $H_j(s)$ and $H_{j,k}(s)$ grow at most polynomially in $\kappa_j + |s|$ or $k + |s|$. 
Extending the Lindel\"{o}f hypothesis to this case we expect $H_j(1/2+it) \ll (|t|+ \kappa_j)^{\varepsilon}$ and $H_{j,k}(1/2+it) \ll (|t| + k)^{\varepsilon}$. In the regime when $t$ and $\kappa_j$ or $k$ have different sizes the current best bound  is $H_j(1/2+it) \ll (|t|+\kappa_j)^{1/3+\varepsilon}, H_{j,k}(1/2+it) \ll (|t|+k)^{1/3+\varepsilon}$ \cite{jutila2005uniform}.

Set
\[
\alpha_j = \frac{|\rho_j(1)|^2}{\cosh \pi \kappa_j}, \quad \alpha_{j,k} = \frac{16 \Gamma(2k)|\rho_{j,k}(1)|^2}{(4\pi)^{2k+1}}\,
\]
so that $\sum_{\kappa_j\le K} \alpha_{j} \ll K^2$ and $\sum_{j=1}^{\vartheta(k)} \alpha_{j,k} \ll k$. Then by \cite{motohashi1992spectral} and \cite[Theorem 2]{deshouillers1982kloosterman} for $m=2,4$
\begin{equation}\label{eq:Hjmoment}
\begin{aligned}
 \sum_{\kappa_j \le X} \alpha_j H_j^m(1/2) \ll X^{2+\varepsilon},\quad 
   \sum_{k \le X} \sum_{j=1}^{\vartheta(k)}\alpha_{j,k} H_{j,k}^m(1/2)\ll X^{2+\varepsilon}\,.
\end{aligned}
\end{equation}
In both cases $H_j(1/2)$ and $H_{j,k}(1/2)$ are in fact real. When $t\neq 0$ Jutila \cite{jutila2004spectral} showed that
\begin{equation}\label{eq:hecke2mom}
\begin{aligned}
   &\sum_{\kappa_j \le X} \alpha_j |H_j(1/2 + it)|^2 \ll X^{2+\varepsilon} + X^{\varepsilon}t^{2/3+\varepsilon}\,,\\
   & \sum_{k \le X} \sum_{j=1}^{\vartheta(k)}\alpha_{j,k} |H_{j,k}(1/2+it)|^2 \ll X^{2+\varepsilon} + X^{\varepsilon}t^{2/3+\varepsilon}\,.
\end{aligned}
\end{equation}
Jutila's proof relies on the corresponding bounds for linear approximations of $H_j$ and $H_{j,k}$, and so \eqref{eq:hecke2mom} also holds for 
\[
\begin{aligned}
S_j(1/2+it) = \sum_{N \le n\le 2N} \frac{t_j(n)}{n^{1/2+it}}\,, \quad 
S_{j,k}(1/2+it) = \sum_{N \le n\le 2N} \frac{t_{j,k}(n)}{n^{1/2+it}}\,,
\end{aligned}
\]
where $t \in \RR$ and $N \in \NN$ such that $N \ll |t| + |\kappa_j|$ and $N \ll |t| + k$ respectively.

\subsection{Correlations of the divisor function}
 Let $r\in \NN$, then
\[
\sum_{n \le x} d(n) d(n+r) = x(a_2(r) \log^2 x+ a_1(r) \log x+ a_0(r)) + E(x,r)\,,
\]
where coefficients $a_j(r)$ depend on $r$, and $E(x,r) = o_r(x)$ is the error term. 

\begin{thm}[Motohashi \cite{motohashi1994binary}]\label{thm:moto_simple} Let $x \ge 1$ be large, and let $r\in \NN$ such that $1 \le  r \le x^{10/7}$.  Then uniformly for all $r$ in this range
\begin{equation*}
    E(x,r) \ll (x(x+r))^{1/3+\varepsilon} + r^{9/40} (x(x+r))^{1/2} + r^{7/10}x^{\varepsilon}\,.
\end{equation*}
In particular, $E(x,r) \ll x^{2/3+\varepsilon}$ when $r \ll x^{20/27}$.
\end{thm}
% Additionally, we have the following expansions \cite{conrey1996note}:
% \begin{equation}\label{lem_conrey}
% \begin{aligned}
%   & S(x,r) =  \sum_{k=1}^{\infty} \frac{\mu(k)}{k^2} \sum_{d|r} \frac{1}{d} \brackets{\brackets{\log \frac{xe^{2\gamma_0 -1}}{d^2k^2}}^2 + 1}\,;\\
%   & (xS(x,r))^{\prime} = \sum_{k=1}^{\infty} \frac{\mu(k)}{k^2} \sum_{d|r} \frac{1}{d} \brackets{\log \frac{xe^{2\gamma_0}}{d^2k^2}}^2\,.
% \end{aligned}
% \end{equation}

To state the explicit version of this result we will need the following additional notation. Let $U(x)$ be an infinitely differentiable function on $(0;\infty)$ of compact support (this condition can be relaxed to $U(x)$ Schwartz with an additional condition on decay as $x\to 0$ via a classic trick of approximating a Schwartz function by a combination of bump functions with arbitrary precision). Set
\begin{equation}\label{eq:XiTheta}
\begin{aligned}
        &\Xi (z;U) :=  \frac{\Gamma^2(1/2+z)}{\Gamma(1+2z)} \int_0^\infty U(x) x^{-1/2-z} {}_2F_1(1/2+z,1/2+z,1+2z;-1/x)dx\,,\\
           & \Theta(y;U) := \frac{1}{2} \re \brackets{ \brackets{1 + \frac{i}{\sinh{\pi y}}} \Xi(iy;U)}\,,
\end{aligned}
\end{equation}
and $m(x,r) := \lim_{h\to \infty} m(x,r,h)$, where
\begin{equation}\label{eq:mxrh}
\begin{aligned}
   m(x,r,h) := \sigma_{1+2h}(r) \frac{\zeta^2(1+h)}{\zeta(2+2h)} x^h (1 + x)^h + r^{2h} \sigma_{1-2h}(r) \frac{\zeta^2(1-h)}{\zeta(2-2h)} + \\r^h \sigma_1(r) \frac{\zeta(1+h)\zeta(1-h)}{\zeta(2)}(x^h + (1+x)^h)\,.
     \end{aligned}
\end{equation}
In particular, let $\frac{d^l}{dy^l} \sigma_y(r) = \sum_{d|r} d^y \log^l d$, then
\begin{equation}\label{eq:mxr}
\begin{aligned}
    m(x,r) = \frac{\sigma_1(r)}{\zeta(2)}(\log x \log (1+x) +\log x(1+x)(2\gamma_0-\log r)+ (2\gamma_0-\log r)^2) + \\ (\log x(1+x) +
    4\gamma_0- 2 \log r)\frac{d}{dh}{\frac{\sigma_{1+2h}(r)}{\zeta(2+2h)}}\bigg\vert_{h=0}
     + \frac{d^2}{dh^2}{\frac{\sigma_{1+2h}(r)}{\zeta(2+2h)}}\bigg\vert_{h=0}= \\
    \frac{d^2}{dh^2}\brackets{(xr^{-1}e^{2\gamma_0})^h \frac{\sigma_{1+2h}(r)}{\zeta(2+2h)}}\bigg\vert_{h=0} \brackets{1 + O(1/|x|)}\,.
\end{aligned}
\end{equation}

\begin{thm}[Motohashi \cite{motohashi1994binary}]\label{thm:moto_exp} Let $U(x)$ be an infinitely differentiable function of compact support defined on $(0;+\infty)$. Then for any $r \ge 1$
\begin{equation*}
    \sum_{n\ge 1} d(n) d(n+r) U(n/r) = \int_0^\infty m(x,r) U(x) dx + E(r;U)\,,
\end{equation*}
where $E(r;U) = e_1(r;U) + e_2(r;U) + e_3(r;U)$ has the following spectral expansion in terms of the continuous, discrete and residual spectra:
\begin{equation*}
\begin{aligned}
      & e_1(r;U) = \frac{r^{1/2}}{\pi} \int_{-\infty}^{\infty} \frac{r^{-iy}\sigma_{2iy}(r) |\zeta(1/2+iy)|^4}{|\zeta(1+2iy)|^2} \Theta(y;U) dy\,; \\
     &   e_2(r;U) = r^{1/2}\sum_{j=1}^{\infty} \alpha_j t_j(r) H_j^2(1/2) \Theta(\kappa_j;U)\,; \\
      &  e_3(r;U) = \frac{1}{4}r^{1/2} \sum_{k=6}^{\infty} \sum_{j=1}^{\vartheta(k)} (-1)^k \alpha_{j,k} t_{j,k}(r)H_{j,k}^2(1/2) \Xi(k-1/2;U)\,.
    \end{aligned}
\end{equation*}
\end{thm}
\subsection{Technical lemmas}

Our computation, unsurprisingly, heavily relies on oscillatory integrals, so let us state several classical results for future reference. The following two lemmas known as oscillation lemmas or Van der Corput lemmas, see \cite[Lemmas 4.3 \& 4.5]{titchmarsh1986theory}.

\begin{lem}[Oscillation lemma I]\label{lem:oscI} Let $f(x)$ be a real differentiable function, and $g(x)$ be continuous. If $f^\prime(x)/g(x)$ is monotone on the interval $[a,b]$ with $|f^\prime(x)/g(x)| \ge Y > 0$, then
\[
\int_a^b g(x) e^{if(x)} dx \ll \frac{1}{Y}\,.
\]
Implied constant does not depend on $f$ and $g$.
\end{lem}
\begin{lem}[Oscillation lemma II]\label{lem:vdc} Let $f(x)$ be a real twice differentiable function, and $g(x)$ be a real continuous function such that $f^{\prime}(x)/g(x)$ is monotone on the interval $[a,b]$. If $|f^{\prime\prime}(x)| \ge Y > 0$ and $0 < |g(x)| \le N$ on $[a,b]$, then
\[
\int_a^b g(x) e^{if(x)} dx \ll \frac{N}{Y^{1/2}}\,
\]
with constant independent of $f$ and $g$.
\end{lem}

When $g(x)$ is a sufficiently nice function, we also have the following asymptotic expansion, see for example \cite[Lemmas 2.7 \& 2.8]{tao2007lecture}.

\begin{lem}[Fresnel phase]\label{lem:fresnel} Let $g(x)$ be an infinitely differentiable function of compact support, and $Y \in \RR$. Then for any $N \in \NN$
\[
\int g(x) e^{iYx^2} dx = e^{\pi i/4}\sqrt{\frac{\pi}{Y}} \sum_{j=1}^N \frac{i^j g^{(2j)}(0)}{j!(4Y)^{j}} + O_{g, N}(Y^{-N-3/2})\,.
\]
\end{lem}

We will also need the asymptotic expansion for $\Gamma(s)$, see for example \cite[(20.2)]{rademacher2012topics}.

\begin{lem}[Stirling's formula]\label{lem:stirling} Let $s =\sigma +it$ be such that $s \notin \ZZ^{\le 0}$, then for any $N \in \NN$
\begin{equation*}
 \log \Gamma(s) = \brackets{s-\frac{1}{2}}\log s - s + \frac12 \log 2\pi + \sum_{j = 0}^{N-1} a_j s^{-(2j+1)} + O_N(|s|^{-2N-1})\,
\end{equation*}
for some real coefficients $a_j$.
\end{lem}

\section{Groundwork}\label{section:ground}
Let $ \alpha, \beta \in \RR$ such that $0 \le \alpha\le \beta \ll T^{2-\varepsilon}$ with $\delta := \beta - \alpha$. In fact, without loss of generality we will take $\alpha = 0$ and $\delta=\beta$ via the change of variables $t \mapsto t - \alpha$. Let $W(t)$ be an appropriate weight function on the interval $[T_1,T_2] \subseteq [T,2T]$, and consider
\begin{equation}\label{eq:MWdelta}
        M_4(W;\delta) = \int_0^\infty W(t) |\zeta(1/2+it)|^2 |\zeta(1/2+it+i\delta)|^2dt\,.
\end{equation}
In the classic formulation $W(t) = \II_{[T_1,T_2]}(t)$ is the indicator function of the interval $[T_1,T_2]$, but this choice of function is difficult for analysis, so instead of the integral on an interval we will be evaluating a weighted integral. One possible choice is
\begin{equation}\label{eq:dgauss}
    W(t)    = F_{\Delta} \ast \II_{[T_1,T_2]}(t)\,,
\end{equation}
where $F_{\Delta}(t) = \Delta^{-1} F(t/\Delta)$ is a normalised Schwartz function with parameter $\Delta$ satisfying
\[
T^{1/2+\varepsilon} \ll \Delta \ll T^{1-\varepsilon}\,.
\]
A common and rather convenient choice is $F(t) =  \pi^{-1/2} e^{-t^2}$. Alternatively, one may choose $W(t)$ to be a smooth function compactly supported on $[T_1,T_2]$. For example,
\begin{equation}\label{eq:dcomp}
    W(t) = \begin{cases}
    p(|t-T_j|/\Delta), & {|t-T_j| \le \Delta}\,,\\
    % p((2T-t)/\Delta), & {2T - \Delta \le t \le 2T}
    1, & {T_1 + \Delta \le t \le T_2 - \Delta}\,,\\
    0, & \text{otherwise}\,,
    \end{cases}
\end{equation}
where $p(x) = \exp\brackets{-x^{-1} \exp \brackets{-(1-x)^{-1}}}$.

In both cases parameter $\Delta$ regulates the precision of the approximation as $W$ only deviates from $\II_{[T_1,T_2]}$ noticeably when $|t-T_j| \ll T^{\varepsilon}\Delta$. For both \eqref{eq:dgauss} and \eqref{eq:dcomp} we have
\begin{equation}\label{eq:WtFder_genbound}
\begin{aligned}
 &   W^{(k)}(t) \ll_{k,l} (1+ \Delta|t-T_1|)^{-l} + (1+ \Delta|t-T_2|)^{-l} \quad \forall k,l \in \ZZ^{\ge 1}\,,\\
   &  \hat{W}^{(k)}(x) \ll_{k,l} (1+ \Delta|x|)^{-l} (1+ |x|/T)^{-1} \quad \forall k,l \in \ZZ^{\ge 0}\,
\end{aligned}
\end{equation}
and
\[
|M_4(W; \delta) - M_4(\II_{[T_1,T_2]}; \delta)| \ll T^{\varepsilon}\Delta\,.
\]

All the results in this paper apply for both choices of smoothing \eqref{eq:dgauss} and \eqref{eq:dcomp}. The distinction boils down to working with Schwartz functions, or Schwarz functions with compact support (``bump functions''). In most cases, we do not have to choose a specific $W$ since it is enough to know that there exists an infinitely differentiable function approximating the indicator function of an interval with arbitrary precision, with derivatives as above, but it might be convenient to use a specific function to simplify the computations.

\subsection{Approximate functional equation}
We start with constructing a weighted Dirichlet series for $|\zeta(1/2 +it) \zeta(1/2-it-i\delta)|^2$.

Recall the functional equation of zeta 
\begin{equation}\label{eq:feqzeta}
\pi^{-s/2} \Gamma(s/2) \zeta(s) = \pi^{-(1-s)/2} \Gamma((1-s)/2) \zeta(1-s)\,,
\end{equation}
or $\zeta(s) = \chi(s)\zeta(1-s)$, and define
\begin{equation}\label{eq:kappaphi}
     \kappa (t) = \kappa_{\delta}(t) := \chi\brackets{1/2-it} \chi\brackets{1/2 + i(t +\delta)} = e^{i\phi(t)}(1+ O(|t|^{-1}))\,,
\end{equation}
where 
\begin{equation}\label{eq:phi_explicit}
    \phi(t) = \phi_{\delta}(t) := t\log \frac{|t|}{2\pi e} - (t+\delta)\log \frac{|t+\delta|}{2\pi e}\,.
\end{equation}
Set
\begin{equation}\label{eq:GPZdelta}
\begin{aligned}
    &\Gamma_{\delta}(z;t) :=\Gamma^2((z+it)/2)\Gamma^2((z-it-i\delta)/2)\,,\\
    & P_{\delta}(z;t) := \brackets{(z+it)(z-1+it)(z-it-i\delta)(z-1-it-i\delta)}^2\,,\\
    & Y_\delta(z;t) := \Gamma_{\delta}(z;t)P_{\delta}(z;t)\,.
\end{aligned}
\end{equation}
Let $Q = Q(T)$ be a parameter, and set
\begin{equation}\label{def:GQ}
    G(z) = e^{z^2/Q}, \quad 1 \ll Q \ll T\,.
\end{equation}
For such $G(z)$ and $x,t > 0$ we define
\begin{equation}\label{def:Vxt}
V_{\delta}(x,t) := \frac{1}{2\pi i} \int_{(1)} (\pi^2 x)^{-z} \frac{Y_\delta(1/2+z;t)}{Y_\delta(1/2;t)} G(z)\frac{dz}{z}\,.
\end{equation}
Now we can rewrite $|\zeta(1/2 +it) \zeta(1/2-it-i\delta)|^2$ as a weighted double Dirichlet series. 
\begin{lem}\label{lem:first_approx} Let $t \ge 1$ and $\delta \in \RR$. Let $Q = Q(T) \gg 1$ and $G(z) = e^{z^2/Q}$. Then
\begin{equation*}
\begin{aligned}
        |\zeta(1/2 +it) \zeta(1/2-it-i\delta)|^2 =   2\re \sum_{n,m=1}^{\infty} \frac{d(n)d(m)}{n^{1/2+it}m^{1/2-it-i\delta}} \kappa(t) V_{\delta}(nm,t)\,.
\end{aligned}
\end{equation*}
\end{lem}
\begin{proof} 
Let $\Lambda_{\delta}(z;t) :=\pi^{-2z +i\delta}Y_{\delta}(z;t)  \zeta^2(z+it)\zeta^2(z-it-i\delta)$. It follows from functional equation \eqref{eq:feqzeta} that $\Lambda_{\delta}(z;t)$ is entire and satisfies its own functional equation
\begin{equation}\label{eq:lambda_feq}
    \Lambda_{\delta}(z;t) = \Lambda_{-\delta}(1-z;-t)\,.
\end{equation}
Consider the contour integral
\begin{equation*}
    \frac{1}{2\pi i} \int_{(1)} \Lambda_{\delta}(1/2+z;t) G(z)\frac{dz}{z}\,.
\end{equation*}
Move the line of integration to $\Re z= -1$, then by Cauchy's residue theorem,
\[
\begin{aligned}
\frac{1}{2\pi i} \int_{(1)} \Lambda_{\delta}(1/2+z;t) G(z)\frac{dz}{z} = \res_{z=0} \frac{\Lambda_{\delta}(1/2+z;t)G(z)}{z} & + \frac{1}{2\pi i} \int_{(-1)} \Lambda_{\delta}(1/2+z;t) G(z)\frac{dz}{z} = \\
\Lambda_{\delta}(1/2;t) & + \frac{1}{2\pi i} \int_{(-1)} \Lambda_{\delta}(1/2+z;t) G(z)\frac{dz}{z} \,.
\end{aligned}
\]
Changing variables to $w=-z$ in the integral on the right-hand side, applying functional equation \eqref{eq:lambda_feq} and dividing by $Y_{\delta}(1/2;t)$, we get
\begin{eqnarray*}
 \frac{\Lambda_{\delta}(1/2;t)}{Y_{\delta}(1/2;t)} =  \frac{1}{2\pi i } \int_{(1)} \frac{\Lambda_{\delta}(1/2+z;t) + \Lambda_{-\delta}(1/2+z;-t)}{Y_{\delta}(1/2;t)}  G(z)\frac{dz}{z}\,.
\end{eqnarray*}
By definition, 
\[
Y_\delta(1/2;t) = P_{\delta}(1/2;t)\Gamma_{\delta}(1/2;t) = \kappa^2(t) P_{-\delta}(1/2;-t)  \Gamma_{-\delta}(1/2;-t) = \kappa^2(t)Y_{-\delta}(1/2;-t)\,.
\]
Then, expanding $\zeta^2(1/2 + z +it)\zeta^2(1/2 + z -it-i\delta)$ into a double series, we obtain
\begin{equation*}
    \begin{aligned}
        |\zeta^2(1/2 +it) & \zeta^2(1/2 -it-i\delta)| = \kappa(t)  \zeta^2(1/2 +it) \zeta^2(1/2 -it-i\delta) = \\
 & \sum_{n,m=1}^{\infty} \frac{d(n)d(m)}{\sqrt{nm}} \brackets{ \frac{m^{i(t+\delta)}}{n^{it}} \kappa(t) V_{\delta}(nm,t)   + 
\kappa^{-1}(t) \frac{n^{it}}{m^{i(t+\delta)}} V_{-\delta}(nm,-t) }\,.
    \end{aligned}
\end{equation*}
The statement of the lemma follows immediately. 
\end{proof}

Our next step is to see that weights $V_\delta(nm,t)$ in the representation given by Lemma \ref{lem:first_approx} essentially truncate the corresponding double Dirichlet series at $mn \ll T(T+\delta)$. To do this, we will show that we can replace $V_{\delta}(nm,t)$ with a smooth cutoff function.

% Recall that the cutoff function $\mathbb{I}\{w \le 1\}$ can be represented by a conditionally convergent contour integral
% \[
% \frac{1}{2\pi i} \int_{(1)} w^z \frac{dz}{z} = \begin{cases}
% 1, & {w > 1}\,,\\
% 1/2, & {w=1}\,,\\
% 0, & {w < 0}\,.
% \end{cases}
% \]
Let $G(z) = e^{z^2/Q}$ as before, and for $w \ge 0$ define
\begin{equation}\label{def:Iw}
I(w) :=\frac{1}{2\pi i} \int_{(1)} w^z G(z/2)\frac{dz}{z} = \frac{1}{2\pi i} \int_{(\sigma)} w^z G(z/2)\frac{dz}{z} \quad (\sigma > 0)\,.
\end{equation}
We have
\[
I(w)  = \sqrt{\frac{Q}{4\pi}} \int_{-\infty}^{\log w} e^{-\frac{Q}{4}y^2} dy = \begin{cases} 1+ O(e^{-\frac{Q}{4}\log^2 w}),& \text{if $w > 1$,}\\
O(e^{-\frac{Q}{4}\log^2 w}),& \text{if $w \le 1$}\,,\end{cases}  
\]
and
\begin{equation}\label{eq:Iwderivatives}
      \frac{d I(w)}{dw} = \frac{\sqrt{Q} }{w} e^{-\frac{Q}{4}\log^2 w}\,, \quad \frac{d^l }{dw^l} I(w) \ll_l \brackets{\frac{\sqrt{Q} }{w}}^l e^{-\frac{Q}{8}\log^2 w}
\end{equation}
so $I(w)$ is a smooth absolutely convergent approximation of the cutoff function. Further, let $x,t,u  > 0$, let $\xi = \log \frac{t(t+u)}{4\pi^2 x}$, and define
\begin{equation}\label{def:Jxt}
    J (x,t) = J_u (x,t) := \frac{1}{2\pi i} \int_{(1)} \brackets{\frac{t(t+u)}{4\pi^2 x}}^z G(z)\frac{dz}{z} = I(e^\xi)\,.
\end{equation}

Our goal now is to show that we can replace $V_\delta(nm,t)$ with $J_\delta(nm,t)$. In particular, we will show that the contribution of
\begin{equation}\label{eq:Bxrdef}
    \begin{aligned}
 B_{\delta} (x,t) := & V_{\delta}(x,t) - J_{\delta} (x,t) = \\
&  \frac{1}{2\pi i} \int_{(1)} (\pi^2 x)^{-z} \frac{Y_{\delta}(1/2+z;t)}{Y_{\delta}(1/2;t)} G(z)\frac{dz}{z} - \frac{1}{2\pi i} \int_{(1)} \brackets{\frac{t(t+u)}{4\pi^2 x}}^z G(z)\frac{dz}{z}
    \end{aligned}
\end{equation}
at $x = nm$ to the total sum in $n,m$ is small. 

\begin{lem}\label{lem:Bbound} Let $t \in [T,2T]$ and $\delta \ge 0$. Let $Q = Q(T)$ such that $Q \ll T$. Let

\begin{equation*}
    A:= A(T,\delta,Q) = \frac{\delta Q}{T(T+\delta)} + \frac{Q^{3/2}}{T^2}\,.
\end{equation*}
Then for $\xi = \log \frac{t(t+\delta)}{4\pi^2 x^2}$ and any $a > 0$
\begin{equation*}
 B_\delta(x,t) \ll_{a} 
 Ae^{-\frac{Q}{8}\xi^2} + x^{-1}T^{-a},\quad  \frac{d}{dt}B_\delta(x,t) \ll_{a} 
 \frac{AQ^{1/2}}{T}e^{-\frac{Q}{8}\xi^2} + x^{-1}T^{-a}\,.
\end{equation*}
\end{lem}
\begin{proof}
Let $y = \Im z > M := T^\varepsilon \sqrt{Q}$. For such $y$, applying Stirling's formula (Lemma \ref{lem:stirling}), we trivially have
\[
\frac{P_\delta(1/2+z;t)\Gamma_{\delta}(1/2+z;t)}{P_\delta (1/2;t)\Gamma_{\delta}(1/2;t)} \ll (T + \delta + |y|)^C\,
\]
for some constant $C$ depending on $\re z$. Then for any $a > 0$
\[
\int_{(1), \re z > M} u^z \frac{Y_{\delta}(1/2+z;t)}{Y_{\delta}(1/2;t)} G(z)\frac{dz}{z} \ll u \int_{y > M} (T+\delta + |y|)^{C} |G(1+iy)| dy  \ll \frac{u}{T^a}\,.
\]
Now suppose $y = \Im z \le M$. Using Stirling's formula in the form of the asymptotic expansion from Lemma \ref{lem:stirling}, we have
\begin{equation}\label{eq:sqrtQbound}
\begin{aligned}
            \frac{\Gamma_{\delta}(1/2+z;t)}{\Gamma_{\delta}(1/2;t)} = \frac{\Gamma^2((1/2+z+it)/2)\Gamma^2((1/2+ z-it-i\delta)/2)}{\Gamma^2((1/2+it)/2)\Gamma^2((1/2-it-i\delta)/2)} = \\
            \exp\brackets{z\log \frac{t(t+\delta)}{4}} (1 + R_{\delta} (z;t))\,,
\end{aligned}
\end{equation}
where $R_{\delta} (z;t)$ is an absolutely convergent series with complex coefficients. In particular,
\begin{equation}\label{eq:sqrtQerror}
    R_{\delta} (z;t) = \sum_{j,k} b_{j,k} \brackets{\frac{z^j}{t^{k}} +  \frac{(-1)^k z^j}{(t+\delta)^{k}}} = \frac{b_{2,1}z^2\delta}{t(t+\delta)} + \frac{b_{3,2}z^3}{t^2} + \frac{b_{3,2}z^3}{(t+\delta)^2}+ \ldots \,
\end{equation}
where $k \ge j-1$ for $j \ge 3$ and $k \ge j$ for $j\le 2$.
% where $N$ can be chosen as large as necessary and  and a finite range of $m$ and $n$ depending on $N$. 
Further, we can also expand
\begin{equation}\label{eq:Pdeltaexp}
    \frac{P_\delta(1/2+z;t)}{P_\delta (1/2;t)} = 
    % \brackets{1 + \frac{z}{it}}^4 \brackets{1 - \frac{z}{i(t+\delta)}}^{4}(1 + O(1/t))= 
    1 + \frac{4z \delta}{it(t+\delta)} + \frac{4z^2}{t(t+\delta)} + \ldots
\end{equation}
into another absolutely convergent seies. Multiplying \eqref{eq:sqrtQbound} and \eqref{eq:Pdeltaexp} we obtain an absolutely convergent series with some complex coefficients $a_j$
\begin{equation}\label{eq:Ydeltaseries}
    \frac{Y_{\delta}(1/2+z;t)}{Y_{\delta}(1/2;t)} =\exp\brackets{z\log \frac{t(t+\delta)}{4}}\brackets{1 + \frac{a_1 z^2}{t(t+\delta)} + \frac{a_2 z^2 \delta}{t(t+\delta)} + \frac{a_3z^3}{t^2} + \ldots}\,.
\end{equation}

For any integer $\nu \ge 0$, any real number $\mu \in \RR$ and any $a > 0$
\[
\int_{y \le M} e^{2\pi z\mu} z^{\nu} G(z) dz = \int_{(1)} e^{2\pi z\mu} z^{\nu} G(z) dz + O_\nu(T^{-a})\,,
\]
where
  \begin{equation}  \label{eq:zG}
    \int_{(1)} e^{2\pi z\mu} z^{\nu} G(z) dz =  i^{\nu + 1} \int_{-\infty}^\infty e^{2\pi i\mu y} y^{\nu} G(iy) dy \ll
    (2\pi)^\nu \frac{d^\nu}{d \mu^\nu} e^{-\frac{Q}{4}\mu^2}
    \ll_{\nu} Q^{\frac{\nu + 1}{2}} e^{-\frac{Q}{8}\mu^2}\,.
\end{equation}
Taking $\mu = \log \frac{t(t+\delta)}{4\pi^2 x^2}$ and applying \eqref{eq:zG} to each term in \eqref{eq:Ydeltaseries} yields the statement of the lemma.
\end{proof}

\begin{prop}[Approximate functional equation]\label{lem:sweights}
Let $0 \le \delta \ll T^{2-\varepsilon}$, and let $J_\delta(nm,t)$ be as defined in \eqref{def:Jxt} with $G(z) = e^{z^2/Q}$ and $Q \ll T$, then
\[
\begin{aligned}
  \int_0^{\infty} W(t) |\zeta(1/2+it)|^2 & |\zeta(1/2+it+i\delta)|^2dt = \\ & \sum_{n,m\ge 1} \frac{d(n)d(m)}{n^{1/2}m^{1/2-i\delta}} \int_0^{\infty} W(t) \brackets{\frac{m}{n}}^{it} \kappa(t)J_{\delta}(nm,t) dt + E(T,\delta,Q)\,,  
\end{aligned}
\]
where
% \[
% E(T,\delta,Q) \ll T^{\varepsilon} \begin{cases}
%   \brackets{\frac{Q}{T}+ \frac{1}{\Delta}}(\delta+ Q^{1/2}), & \text{if $\delta \ll Q$}\,;\\
% \delta^{1/2} Q , & \text{if $\delta \gg 
%  Q$}\,.
%  \end{cases}
% \]
\[
E(T,\delta,Q) \ll T^{\varepsilon} \begin{cases}
T^{1/2} + \delta, & \text{if $\delta \ll Q$}\,;\\
\delta^{1/2} Q , & \text{if $\delta \gg 
 Q$}\,.
 \end{cases}
\]
\end{prop}

\begin{proof}
Let $B_{\delta} (x,t) = V_{\delta}(x,t) - J_{\delta} (x,t)$ be as defined in \eqref{eq:Bxrdef}, then our goal is to show
\begin{equation}\label{eq:Bdeltabound1}
\begin{aligned}
 \int_{0}^{\infty} W(t)\kappa(t)\sum_{n,m \ge 1} \frac{d(n)d(m)B_{\delta}(nm,t)}{n^{1/2+it}m^{1/2-it-i\delta}}   dt \ll 
 T^{\varepsilon} \begin{cases}
 \frac{Q}{T}(\delta+ Q^{1/2}), & \text{if $\delta \ll Q$}\,;\\
\delta^{1/2} Q , & \text{if $\delta \gg 
 Q$}\,.
 \end{cases}
\end{aligned}
\end{equation}
Let
\begin{equation}\label{eq:Adefn}
\xi :=\xi(t,\delta,n,m)= \log \frac{t(t+\delta)}{4\pi^2 mn}, \quad A := A(T,\delta,Q) = \frac{\delta Q}{T(T+\delta)} + \frac{Q^{3/2}}{T^2}\,,
\end{equation}
then by Lemma \ref{lem:Bbound} for any $a > 0$
\begin{equation}\label{eq:Bdeltanmtbound}
    B_{\delta}(nm,t) \ll_a  A e^{-\frac{Q}{8}\xi^2} + (nm)^{-1}T^{-a}\,.
\end{equation}
This bound allows us to restrict to $\xi \ll T^\varepsilon Q^{-1/2}$. 

Let $\tau := 2\pi \sqrt{nm}$, let $q$ satisfy $q(q+\delta) = \tau^2$. Let $\tau^* = \tau^2/(2q+\delta) = q(q+\delta)/(2q+\delta)$. Setting $w = t-q$ we have
\[
\frac{t(t+\delta)}{4\pi^2mn} = \frac{(q+w)(q+w+\delta)}{\tau^2} = 1 + \frac{w(2q+\delta)}{\tau^2} + \frac{w^2}{\tau^2}\,
\]
and
\begin{equation}\label{eq:xiexp}
    \xi = \log \frac{t(t+\delta)}{4\pi^2mn} = \frac{w}{\tau^*} + O\brackets{\frac{w^2}{(\tau^*)^2} + \frac{w^2}{\tau^2}}\,. 
    % \ll T^\varepsilon Q^{-1/2} \,.
\end{equation}
Then $\xi \ll T^\varepsilon Q^{-1/2}$ implies that $w = t -q \ll T^\varepsilon Q^{-1/2}\tau^* \ll T^\varepsilon Q^{-1/2} q$. In particular, since $t \asymp T$, our double sum in $n,m$ is essentially restricted to $q = q(n,m) \asymp T$, or, equivalently, $\tau = \tau(n,m) \asymp T(T+\delta)$. Thus, due to our choice of smoothing \eqref{eq:Bdeltabound1} is absolutely convergent so by Fubini's theorem we can change the order of summation and integration.

Recall that $\kappa(t) = e^{i\phi(t)}$, where $\phi(t)$ is as defined in \eqref{eq:phi_explicit}, and let $\eta(x) = \log m/n + \phi^\prime(x)$.

Suppose $\delta \ll Q \ll T$.  If $\eta(q) \ll 1/T$, then by \eqref{eq:Bdeltanmtbound} for $q \asymp T$
\[
\int_{0}^{\infty}\big(\frac{m}{n}\big)^{it} W(t)\kappa(t) B_{\delta}(nm,t) dt \le \int_{0}^{\infty}|W(t)B_{\delta}(nm,t)|dt \ll  AT^{1+\varepsilon} Q^{-1/2}
\]
and the contribution of such $n,m$ is
\[
\begin{aligned}
\int_{0}^{\infty} W(t)\kappa(t)\sum_{\substack{n,m \ge 1\\ \eta(q) \ll 1/T}} \frac{d(n)d(m) B_{\delta}(nm,t)}{n^{1/2+it}m^{1/2-it-i\delta}} dt \ll 
\frac{AT^{1+\varepsilon}}{Q^{1/2}} \sum_{\substack{\eta(q) \ll 1/T\\ q \asymp T}} \frac{1}{\sqrt{mn}} + O(T^{-a}) \ll \frac{{AT^{1+\varepsilon}}}{Q^{1/2}}\,.
\end{aligned}
\]
If $\eta(q) \gg 1/T$, then integrating by parts
\begin{equation*}
    \begin{aligned}
\int_{0}^{\infty}\big(\frac{m}{n}\big)^{it} W(t)\kappa(t) & B_{\delta}(nm,t) dt = 
 O((\tau+  T)^{-a}) - \\
&  \frac{1}{i\eta(q)}\int_{0}^{\infty}e^{it\eta(q)}\brackets{W(t)e^{i(\phi(t)-t\phi^\prime(q))} B_{\delta}(nm,t)}^\prime dt\,.
    \end{aligned}
\end{equation*}
When $q \asymp T$, using the general bound \eqref{eq:WtFder_genbound} for $W^\prime(t)$ and \eqref{eq:Bdeltanmtbound} for $B_{\delta}(nm,t)$
\begin{equation}\label{eq:Blem2}
    \begin{aligned}
        \int_{0}^{\infty}\big(\frac{m}{n}\big)^{it} \kappa(t) W^\prime(t) B_{\delta}(nm,t) dt \ll \begin{cases}
            \frac{A}{\Delta} \min\{\frac{T}{\sqrt{Q}},\Delta\}, & \text{if $|q-T_j| \ll T^\varepsilon\brackets{\frac{T}{\sqrt{Q}} + \Delta}$}\,,\\
            T^{-a}, & \text{otherwise}\,,
        \end{cases}
    \end{aligned}
\end{equation}
and using the bound for $B_{\delta}^\prime(nm,t)$ from Lemma \ref{lem:Bbound}
\begin{equation}\label{eq:Blem1}
    \int_{0}^{\infty}\big(\frac{m}{n}\big)^{it} \kappa(t) W(t) B_{\delta}^\prime(nm,t) dt\ll \int_{0}^{\infty} |W(t) B_{\delta}^\prime(nm,t)| dt \ll A\,.
\end{equation}
Further, given
\[
\brackets{e^{i(\phi(t)-t\phi^\prime(q))}}^\prime \ll |t-q| \max_{|y-q| \ll T^{1+\varepsilon} Q^{-1/2}} |\phi^{\prime\prime}(y)| \ll  \frac{\delta |t-q|}{q(q+\delta)} \ll \frac{\delta |t-q|}{T^2}
\]
we similarly have
\begin{equation}\label{eq:Blem3}
\begin{aligned}
\int_{0}^{\infty} e^{it\eta(q)}W(t)B_{\delta}(nm,t)\brackets{\kappa(t)e^{-it\phi^\prime(q)} }^\prime dt \ll \frac{A\delta}{T^2} \frac{T^2}{Q} \ll \frac{A\delta}{Q} \ll A\,.
    \end{aligned}
\end{equation}
Putting together \eqref{eq:Blem2}, \eqref{eq:Blem1}, and \eqref{eq:Blem3}, for $\eta(q) \gg 1/T$
\[
\int_{0}^{\infty}\big(\frac{m}{n}\big)^{it} W(t)\kappa(t) B_{\delta}(nm,t) dt \ll \begin{cases}
A, & \text{if $q \asymp T$}\,;\\
(\tau+ T)^{-a}, & \text{otherwise}\,.
\end{cases}
\]
Summing over $n,m$ we obtain
\begin{equation*}
    \begin{aligned}
        \int_{0}^{\infty} W(t)\kappa(t)\sum_{\substack{n,m \ge 1\\ \eta(q) \gg 1/T}} \frac{d(n)d(m)B_{\delta}(nm,t)}{n^{1/2+it}m^{1/2-it-i\delta}}   dt \ll A \sum_{\substack{q \asymp T,\\ \eta(q) \gg 1/T}} \frac{d(n)d(m)}{(nm)^{1/2}|\eta(q)|}  
         + O(T^{-a}) \ll \\
         AT^{1+\varepsilon} \ll  T^\varepsilon\brackets{\frac{\delta Q}{T} + \frac{Q^{3/2}}{T}} \ll T^\varepsilon\brackets{\delta + T^{1/2}} \,.
    \end{aligned}
\end{equation*}

When $\delta \gg Q$, we may instead use the oscillation lemma for the second derivative. Let $\tilde{T} = \sqrt{T(T+\delta)}$, then $\phi^{\prime\prime}(t) \asymp \delta/\tilde{T}^2$, so by Lemma \ref{lem:vdc}
\begin{equation}
       \int_0^\infty  \big(\frac{m}{n}\big)^{it }W(t)  \kappa(t)  B_{\delta}(nm,t) dt \ll 
\begin{cases}A\tilde{T}\delta^{-1/2}, & \text{if $q \asymp T$}\,,\\
       (\tau+T)^{-a}, & \text{otherwise}\,.
       \end{cases}
\end{equation}
This bound is not quite optimal, but it suffices for our purposes. Summing over $m$ and $n$,
\begin{eqnarray*}
\begin{aligned}
 \sum_{n,m \ge 1}\frac{d(n)d(m)}{n^{1/2}m^{1/2-i\delta}} \int_0^\infty  \big(\frac{m}{n}\big)^{it }W(t)  \kappa(t)  B_{\delta}(nm,t) dt \ll 
          \frac{A\tilde{T}}{\delta^{1/2}} \sum_{nm \ll \tilde{T}^2}\frac{d(n)d(m)}{\sqrt{nm}} \ll \\
          T^\varepsilon \delta^{1/2} Q + T^\varepsilon \frac{Q^{3/2}(T+\delta) }{\delta^{1/2}T} \ll T^\varepsilon \delta^{1/2} Q\,.
        %   + \frac{Q^{3/2}(T+\delta) }{\delta^{1/2}T})\,.
\end{aligned}
\end{eqnarray*}
This concludes the proof of the proposition.
\end{proof}

Our next step is to evaluate the double Dirichlet series obtained in Proposition \ref{lem:sweights},
\[
\sum_{n,m\ge 1} \frac{d(n)d(m)}{n^{1/2}m^{1/2-i\delta}} \int_0^{\infty} W(t) \brackets{\frac{m}{n}}^{it} \kappa(t)J_{\delta}(nm,t) dt\,.
\]
We treat the diagonal $m=n$ and the off-diagonal $m\neq n$ contributions separately. The former is quite simple, while the latter constitutes the bulk of the computation. When evaluating the off-diagonal term, there are two distinct regimes. When $\delta$ is small, the integral
\begin{equation*}\label{eq:wJkappa}
\int_0^{\infty} W(t) (\frac{m}{n})^{it} \kappa(t)J_{\delta}(nm,t) dt
\end{equation*}
behaves similarly to the Fourier transform of $WJ_{\delta}$ at $\eta(t) = \log m/n + \phi^\prime(t)$. When $\delta$ is large, the main contribution comes from the stationary point instead. In both cases the off-diagonal term computation relies on bounds for the error term in the correlations of the divisor function (see Theorem \ref{thm:moto_exp}).

\section{Diagonal term}\label{sec:diag} 

Evaluating the diagonal contribution in our problem is fairly straightforward. We will see that it turns out to have size
\[
\sum_{n\ge 1} \frac{d^2(n)}{n^{1-i\delta}} \int_0^{+\infty} W(t) \kappa(t) J_{\delta}(n^2,t) dt \ll \frac{T^{1+\varepsilon}}{1+ |\delta|}\,,
\]
so it is only meaningful when $\delta \ll T^{1/2}$.

\begin{lem}\label{lem:diagJI} Suppose $0 \le \delta \ll \min\{TQ^{-1/2}, T^{1/2}\}$, then
    \[
    \sum_{n\ge 1} \frac{d^2(n)}{n^{1-i\delta}} \int_0^{+\infty} W(t) \kappa(t) J_{\delta}(n^2,t) dt = \sum_{n\ge 1} \frac{d^2(n)}{n^{1-i\delta}} \int_0^{+\infty} W(t) \kappa(t) I(t/2\pi n) dt + O(T^\varepsilon \delta)\,.
    \]
\end{lem}
\begin{proof}
Recall the integral definitions of $I(w)$ and $J_\delta(x,t)$ (see \eqref{def:Iw} and \eqref{def:Jxt}), then
\[
J_\delta (n^2,t) - I(t/2 \pi n) = \frac{1}{2\pi i} \int_{(1)} \brackets{\frac{t}{2\pi n}}^{2z}\brackets{e^{z\log (1+ \delta/t)}-1} G(z)\frac{dz}{z}\,.
\]
Provided that $t \asymp T$, $\delta \ll TQ^{-1/2}$ and $z \ll Q^{1/2}T^\varepsilon$ the series
\[
e^{z\log (1+ \delta/t)}-1 = \sum_{j=1}^{\infty} \frac{z^j}{j!}\log^j (1+ \delta/t)
\]
is absolutely convergent with $j$-th term is of size $O_j(T^{\varepsilon} (\delta Q^{1/2}T^{-1})^j)$. Computing the Fourier transform term by term with
\[
\int_{(1)} \brackets{\frac{t}{2\pi n}}^{2z} \frac{z^j}{j!} \log^j (1+ \delta/t) G(z)\frac{dz}{z} \ll_j \delta^jQ^{j/2}T^{-j} e^{-\frac{Q}{8} \log^2 \frac{t}{2\pi n}}
\]
we obtain $J_\delta (n^2,t) - I(t/2 \pi n) \ll \delta Q^{1/2}T^{-1} e^{-\frac{Q}{8} \log^2 \frac{t}{2\pi n}}$. Therefore
\[
\begin{aligned}
\sum_{n\ge 1} \frac{d^2(n)}{n^{1-i\delta}} & \int_0^{\infty} W(t) \kappa(t) (J_\delta(n^2,t)-I(t/2\pi n)) dt \ll \\
& \frac{\delta Q^{1/2}}{T}\sum_{n\ge 1} \frac{d^2(n)}{n} \int_0^{\infty} |W(t)|e^{-\frac{Q}{8} \log^2 \frac{t}{2\pi n}} dt \ll 
\delta \sum_{n\ll T^2} \frac{d^2(n)}{n} \ll T^\varepsilon \delta \,,
\end{aligned}
\]
which then implies the statement of the lemma.
\end{proof}

\begin{prop}[Diagonal term]\label{thm:diag} Let $0 \le \delta \ll T^{2-\varepsilon}$. Let $J(x,t)$ be as defined in \eqref{def:Jxt} with $G(z) = e^{z^2/Q}$ and $Q \ll T$. Let $W(t)$ be as defined in \eqref{eq:dgauss} or \eqref{eq:dcomp} with parameter $\Delta$. Let
\begin{equation*}
    D_\delta(t) = 2 \re \sum_{s \in \{0,i\delta\}} \res_{s} \frac{\zeta^4(1+s)}{\zeta(2+2s)}  \frac{(t/2\pi)^s}{s-i\delta}\,,
\end{equation*}
then
\begin{equation*}
    \begin{aligned}
       2 \re \sum_{n\ge 1} \frac{d^2(n)}{n^{1-i\delta}} \int_0^{\infty} W(t) & \kappa(t) J_{\delta}(n^2,t) dt = 
       \int_0^{\infty} W(t) D_\delta(t) dt + E(T,\delta,Q)\,,
    \end{aligned}
\end{equation*}
where $E(T,\delta,Q) \ll T^{\varepsilon}(T^{1/2} +T (\delta^2+Q)^{-1})$.
\end{prop}
\begin{proof} Recall that $\kappa (t) = \chi\brackets{1/2-it} \chi\brackets{1/2 + i(t +\delta)} = e^{i\phi(t)}(1+ O(1/|t|))$, where
\begin{equation*}
    \phi(t) = t\log \frac{t}{2\pi e} - (t+\delta)\log \frac{t+\delta}{2\pi e}, \quad \phi^\prime(t) = \log t - \log (t+\delta)\,.
\end{equation*}
Observe that both $W(t)$ and $J_\delta (nm,t)$ are bounded real functions. Further, if $n \gg T^2$ 
\[
J_\delta(nm,t) \ll e^{-\frac{Q}{4} \log^2\frac{t(t+\delta)}{4\pi^2n^2}} \ll e^{-\frac{Q}{8} \log^2 n}
\]
making the contribution of such $n$ negligible. 

For $\delta \gg T^{1/2-\varepsilon}$ and $ n \ll T^2$ the oscillation lemma (Lemma \ref{lem:oscI}) gives
\[
\begin{aligned}
\int_0^{\infty} \kappa(t) W(t) J_{\delta}(n^2,t) dt= \int_1^{\infty} e^{i\phi(t)} W(t)  J_{\delta}(n^2,t) dt + O(1) \ll 
 % \max_{t \asymp T} \frac{|W(t)J_{\delta}(n^2,t)|}{|\phi^\prime(t)|}  + 1 \ll \brackets{\log \frac{T+\delta}{T}}^{-1} + 1 \ll 
 T^{1/2+\varepsilon}\,,
\end{aligned}
\] 
hence the contribution of the diagonal term is bounded by
\begin{equation*}
   \int \kappa(t) W(t) \sum_{n\ge 1} \frac{d^2(n)}{n^{1-i\delta}} J_{\delta}(n^2,t) dt \ll  T^{1/2+\varepsilon} \sum_{n\ll T^2} \frac{d^2(n)}{n}  \ll T^{1/2+\varepsilon} \,. 
\end{equation*}

For $\delta \ll T^{-\varepsilon} \min\{TQ^{-1/2}, T^{1/2}\}$, apply Lemma \ref{lem:diagJI} so that
\[
\begin{aligned}
    \int_0^{\infty} W(t) \kappa(t) \sum_{n\ge 1} \frac{d^2(n)}{n^{1-i\delta}}  J_{\delta}(n^2,t) dt =  \int_0^{\infty} W(t) \kappa(t) \sum_{n\ge 1} \frac{d^2(n)}{n^{1-i\delta}} 
 I(t/2\pi n) dt + O(T^{\varepsilon}\delta) = \\
 \int_0^{\infty} W(t) \sum_{n\ge 1} \frac{d^2(n)}{n} \brackets{\frac{t}{2\pi n}}^{-i\delta} I(t/2\pi n) dt + O( T^\varepsilon \delta )\,.
\end{aligned}
    \]
Set $L(s) = \zeta^4(s)/\zeta(2s)$, then
\begin{equation}\label{eq:divdeltasum}
    \begin{aligned}
   \sum_{n\ge 1} \frac{d^2(n)}{n} \brackets{\frac{t}{2\pi n}}^{-i\delta} I(t/2\pi n) =  \frac{1}{2\pi i}\int_{(2)} L(1+z-i\delta)  \brackets{\frac{t}{2\pi}}^{z-i\delta} \frac{G(z/2)dz}{z} = \\
\frac{1}{2\pi i}\int_{(2)} L(1+s)  \brackets{\frac{t}{2\pi}}^{s} \frac{G((s+i\delta)/2)}{(s+i\delta)}ds\,.
\end{aligned}
\end{equation}
Proceeding similarly to the classic result for $\sum_{n \ge 1} \frac{d^2(n)}{n} g(X/n)$, shift the contour to $\re s = -1/2$ and apply Cauchy's residue theorem. When $\re s = -1/2$ we apply standard bounds $\zeta(1/2+it) \ll (1+|t|)^{1/6+\varepsilon}$ and $1/\zeta(1+it) \ll \log(2+|t|)$. Then \eqref{eq:divdeltasum} equals
\begin{equation*}
    \sum_{s \in \{0,-i\delta\}} \res_{s} \frac{L(1+s) G((s+i\delta)/2) }{(s+i\delta)} \brackets{\frac{t}{2\pi}}^s + O(|t|^{-1/2+\varepsilon})\,.
\end{equation*}
If $\delta \neq 0$, then $s=-i\delta$ is a simple pole with residue
\[
\begin{aligned}
    \res_{s =-i\delta} \frac{L(1+s) G((s+i\delta)/2) }{(s+i\delta)} \brackets{\frac{t}{2\pi}}^s =  L(1-i\delta) G(0) \brackets{\frac{t}{2\pi}}^{-i\delta} =  L(1-i\delta) \brackets{\frac{t}{2\pi}}^{-i\delta}\,.
\end{aligned}
\]
If $\delta \gg Q^{1/2-\varepsilon}$, the residue at $s=0$ is small with
\[
\re \res_{s = 0 } \frac{L(1+s) G((s+i\delta)/2) }{s+i\delta} \brackets{\frac{t}{2\pi}}^s \ll t^\varepsilon\delta^{-2}e^{-\delta^2/Q} \ll T^\varepsilon Q^{-1}\,.
\]
Otherwise for $\delta \ll Q^{1/2-\varepsilon}$ and $l \in \NN$ we have 
\[
\re \frac{d^l}{ds^l}\brackets{\frac{G((s+i\delta)/2) }{(s+i\delta)} } \bigg\vert_{s=0} = \frac{d^l}{ds^l}\brackets{\frac{s}{s^2+ \delta^2} } \bigg\vert_{s=0} + O(Q^{-1})
\]
and thus for all $\delta > 0$
\[
\re \res_{s = 0 } \frac{L(1+s) G((s+i\delta)/2) }{(s+i\delta)} \brackets{\frac{t}{2\pi}}^s =  \re \res_{s = 0} \frac{L(1+s)}{s+i\delta} \brackets{\frac{t}{2\pi}}^s+ O(T^{\varepsilon} (\delta^2 + Q)^{-1})\,.
\]
Taking the minimum between the two regimes yields the statement of the proposition.
\end{proof}
Note that it immediately follows from Proposition \ref{thm:diag} that
\begin{equation}
    \begin{aligned} 
        \int_0^{\infty} W(t) D_\delta(t) dt = \int_{T_1}^{T_2} D_\delta(t) dt +O(T^{\varepsilon}\Delta)= F(T_2) - F(T_1) +O(T^{\varepsilon}\Delta)\,,
    \end{aligned}
\end{equation}
where
\[
F(X) =2 X\re \sum_{s \in \{0,i\delta\}}  \res_{s} \frac{\zeta^4(1+s) (X/2\pi)^s}{\zeta(2+2s)(s-i\delta)(s+1)}\,.
\]

We now move on to the off-diagonal contribution starting with the case of small shifts.

\section{Off-diagonal term for small shifts} \label{sec:small}

In this section we consider $\delta$ such that 
\begin{equation}\label{eq:smalldelta_assum0}
    0 \le \delta \ll T^{-\varepsilon} \min\{\Delta, T^2/\Delta^2\} \ll T^{2/3-\varepsilon}\,,
\end{equation}
and assume that parameters $Q$ and $\Delta$ satisfy
\begin{equation}\label{eq:smalldelta_assum}
    T^{\varepsilon} \delta \ll Q \ll T, \quad \Delta Q^{1/2} \ll T\,.
\end{equation}
In fact, we can set $\Delta Q^{1/2} = T^{1-\varepsilon}$, though this is not necessary. The current best error term for the fourth moment of the Riemann zeta is of size $T^{2/3+\varepsilon}$ \cite{ivic1995fourth}, so we do not expect to be able to capture any terms of smaller size making this a natural cutoff for $\delta$. 

To simplify our computations, in this subsection we set
\begin{equation}
    W(t)
    = F_{\Delta} \ast \II_{[T_1,T_2]}(t)\,, \quad F_{\Delta}(t) =  (\sqrt{\pi} \Delta)^{-1} e^{-t^2/\Delta^2}\,.
\end{equation}
Additionally, for $x > 0$ let us define
\begin{equation}\label{defn:Udelta}
U_{\delta}(x) :=  \int_0^\infty W((x+1/2) w-\delta/2) \cos(w) \brackets{\frac{w}{2\pi r}}^{-i\delta} I(w/2\pi r) dw \,.
\end{equation}

\subsection{Transforming the integral weight}
In order to apply Theorem \ref{thm:moto_exp} we will first need to reduce the integral weight 
\[
\int_0^\infty \big(\frac{m}{n}\big)^{it} W(t) \kappa(t) J_{\delta}(nm,t) dt\,
\]
to a form compatible with the theorem, and check that the resulting function satisfies appropriate decay conditions. Our eventual goal is to obtain an approximate spectral expansion of the error term up to $O(T^{1/2+\varepsilon})$ rather than an exact one, so we want the weight function to be as simple as possible.

We start with the following technical lemma transforming the integral weight.

\begin{lem}\label{lem:smallprep} 
Let $0 \le \delta \ll T^{-\varepsilon} \min\{\Delta, T^2/\Delta^2\}$. Then 
\begin{equation*}
\begin{aligned}
\sum_{n\neq m}\frac{d(n)d(m)}{n^{1/2}m^{1/2-i\delta}}\int_0^\infty \big(\frac{m}{n}\big)^{it} W(t) \kappa(t) J_{\delta}(nm,t) dt
     =   
     \sum_{n,r \ge 1}
   \frac{d(n)d(n+r)}{r} U_\delta(n/r) + O(T^\varepsilon \Delta)\,..
\end{aligned}
\end{equation*}
\end{lem}
\begin{proof}
    We start with observing that weights $\int_0^\infty \big(\frac{m}{n}\big)^{it} W(t) \kappa(t) J_{\delta}(nm,t) dt$ in the double sum in $n,m$ restrict us to the range where
    \[
    \log \frac{m}{n} \ll T^\varepsilon\brackets{\frac{\delta}{T} + \frac{1}{\Delta} + \frac{\sqrt{Q}}{T}} \ll T^\varepsilon\brackets{\frac{\delta}{T} + \frac{1}{\Delta}} =: C\,.
    \]
Let $\tau = 2\pi \sqrt{mn}$. Let $f(t) := t\log m/n + \phi(t)$, then 
    \[
    f^\prime(t) = \log \frac{m}{n} + \phi^\prime(t) = \log \frac{mt}{n(t+\delta)}
    \]
and $f^\prime(t) = 0$ at $t_0 = n\delta/(m-n)$. Thus, if $r:= m-n \ge 100\delta n/T$, $f^\prime(t) >0$ on $[T/2,4T]$. Further, when $\tau \gg T$ with sufficiently large constant it follows from the definition of $J(x,t)$ (in \eqref{def:Jxt}) that for any $a >0$
\[
\int_0^\infty \big(\frac{m}{n}\big)^{it} W(t) \kappa(t) J_{\delta}(nm,t) dt \ll_a \tau^{-a}\,,
\]
so we are also restricted to $\tau \ll T$ or $mn \ll T^2$.

Suppose $\log m/n \gg C$. Let 
\[
\begin{aligned}
    & \lambda := \min_{t \in [T_1,T_2]} |f^\prime(t)| \asymp |\log m/n| + |\phi^\prime(T)| \gg C, \\
    & \lambda^\prime := \max_{t \in [T_1,T_2]} |f^{\prime\prime}(t)| \asymp \phi^{\prime\prime}(T) \asymp \frac{\delta}{T^2}\,.
\end{aligned}
\]
Integrating by parts $l$ times and applying the oscillation lemma (Lemma \ref{lem:oscI}) together with the rate of decay of $W^{(j)}$ and $I^{(j)}$ (in \eqref{eq:WtFder_genbound} and \eqref{eq:Iwderivatives}) we see that
\begin{equation*}
    \begin{aligned}
        \int_0^\infty \big(\frac{m}{n}\big)^{it} & \kappa(t) W(t) J_{\delta}(nm,t) dt \ll_l 
        & \lambda^{-1} M^l \max_{t \asymp T} |W(t) J_{\delta}(nm,t)| + O((T+\tau)^{-a})\,,
    \end{aligned}
\end{equation*}
where
\[
M := \frac{1}{\Delta\lambda} + \frac{\sqrt{Q}}{T\lambda} + \frac{\lambda^\prime}{\lambda^2} \ll \frac{1}{C}\brackets{\frac{1}{\Delta} + \frac{\sqrt{Q}}{T} + \frac{\delta}{T^2C}} \ll T^{-\varepsilon}\,.
\]
That is, for large enough $l = l(a)$ the contribution of such $m,n$ is $O(T^{-a})$ for any $a > 0$.

Now, if $\log m/n \ll C$ and thus $r/n \ll C$, let us change variables to
\[
w = \sqrt{\frac{t(t+\delta)}{4\pi^2 mn}} = \frac{\sqrt{t(t+\delta)}}{\tau}\,,
\]
then setting $w_1 = w(T/2)$ and $w_2 = w(4T)$
\[
\frac{m^{i\delta}}{n^{1/2}m^{1/2}}\int_{T/2}^{4T} \big(\frac{m}{n}\big)^{it} W(t) \kappa(t) J_{\delta}(nm,t) dt = \int_{w_1}^{w_2} e^{i\tilde{f}(w)} w^{-i\delta} \tilde{W}(w) I(w) \frac{2\pi e^{i\delta} dw}{\sqrt{1 + \delta^2/4w^2\tau^2}}\,,
\]
where $\tilde{W}(w) =W((\sqrt{\delta^2 + 4\tau^2w^2} - \delta)/2)$ and 
\[
\tilde{f}(w) = \sqrt{\delta^2 + 4\tau^2w^2}\log \frac{\sqrt{\delta^2 + 4\tau^2w^2} - \delta}{4\pi nw}\,.
\]

Let $R := R(w) = r-\delta/2\pi w$ and $ N := N(w) = 4\pi wn + \delta$. Using the Taylor expansion, we directly obtain
\begin{equation}\label{exp:1}
    \begin{aligned}
     \sqrt{\delta^2 + 4\tau^2w^2} = \sqrt{N^2 + 16\pi^2w^2 n R} =  
     % N \sqrt{1 +  \frac{16\pi^2w^2 n R}{N^2}} = \\
      N\brackets{1 + \frac{8\pi^2w^2 n R}{N^2}  - \frac{32\pi^4w^4n^2R^2}{N^4} + O\brackets{\frac{R^3}{ n^{3}}}}\,.
\end{aligned}
\end{equation}
The error term here is in fact a series in $w$ bounded uniformly in $w$. Further,
\[
\frac{\sqrt{\delta^2 + 4\tau^2w^2} - \delta}{2} = 2\pi n w + \pi w R + O\brackets{\frac{R\delta}{n} + \frac{TR^2}{n^2}} 
% = v -\delta/2 + O\brackets{\frac{R\delta}{n} + \frac{TR^2}{n^2}}\,,
\]
and
\[
\begin{aligned}
\log \frac{\sqrt{\delta^2 + 4\tau^2w^2} - \delta}{4\pi nw} = 
% \log \brackets{\frac{N}{4\pi n w}  \sqrt{1 +  \frac{16\pi^2w^2 n R}{N^2}} - \frac{\delta}{4\pi n w}} = \\ 
&\log \brackets{1 + \frac{2\pi Rw}{N} - \frac{8\pi^3w^3R^2n}{N^3} + O\brackets{\frac{R^3}{ n^{3}}}} = \\
&\frac{2\pi Rw}{N}\brackets{1 - \frac{\pi wR(2N- \delta)}{N^2} + O\brackets{\frac{R^2}{ n^{2}}}}\,.
\end{aligned}
\]
Then
\begin{equation}\label{exp:ftilde}
    \begin{aligned}
   \tilde{f}(w) =  N \sqrt{1 +  \frac{16\pi^2w^2 n R}{N^2}} \log \brackets{\frac{N}{4\pi n w}  \sqrt{1 +  \frac{16\pi^2w^2 n R}{N^2}} - \frac{\delta}{4\pi n w}} = \\
   2\pi w R \brackets{1 - \frac{\pi w \delta R}{N^2} + O\brackets{\frac{R^2}{ n^{2}}}}\,
\end{aligned}
\end{equation}
and
\begin{equation}\label{exp:ftildeprime}
\begin{aligned}
    \tilde{f}^\prime(w) = \frac{4\tau^2 w \tilde{f}(w)}{\delta^2 + 4\tau^2w^2} + \frac{\delta}{w} = 2\pi R\brackets{1 -\frac{\pi w \delta R}{N^2} - \frac{\delta^2}{N^2} + O\brackets{\frac{R^2}{n^2} + \frac{R\delta^2}{nT^2}}} + \frac{\delta}{w}\,.
\end{aligned}
\end{equation}
In particular, when $\log m/n \ll C$,
\begin{equation}\label{eq:ftildeprime1}
   \tilde{f}^\prime(w) = 2\pi R (1 + O(C^2)) + \delta/w\,. 
\end{equation}

Now, let us simplify
\[
2 \pi \int_{w_1}^{w_2} e^{i\tilde{f}(w)} w^{-i\delta} \tilde{W}(w) I(w) \frac{dw}{\sqrt{1 + \delta^2/4w^2\tau^2}}\,.
\]
First, applying the oscillation lemma for the first or second derivative (Lemma \ref{lem:oscI} and \ref{lem:vdc}), and taking the minimum between the two,
\begin{equation}\label{eq:E1nmbound}
    \begin{aligned}
   E_1(n,m) :=  2 \pi \int_{w_1}^{w_2} e^{i\tilde{f}(w)} w^{-i\delta} \tilde{W}(w) I(w) \brackets{\frac{1}{\sqrt{1 + \delta^2/4w^2\tau^2}} -1}dw \ll \\
   \frac{\delta^2}{T^2}\frac{\delta^{1/2}}{n(|\log m/n| + \delta/T)}\,,
\end{aligned}
\end{equation}
and thus
\begin{equation}\label{eq:sqrterror}
\begin{aligned}
    \sum_{nm \ll T^2, \log m/n \ll C} d(n)d(m) E_1(n,m) \ll  \frac{T^\varepsilon \delta^{5/2}}{T^2}\sum_{n \ll T, r \ll nC}  \frac{1}{|r| + \delta n/T} \ll  \frac{T^{\varepsilon}\delta^{5/2}}{T}\,.
    \end{aligned}
\end{equation}
Next, expand $W(t(w)) = \tilde{W}(w)$ near $v = v(w) = 2\pi nw + \pi r w-\delta/2$, and consider
\[
E_2(n,m) := 2 \pi \int_{w_1}^{w_2} e^{i\tilde{f}(w)} w^{-i\delta} (\tilde{W}(w) - W(2\pi n w + \pi w r))I(w) dw\,.
\]
Near $v = 2\pi n w + \pi w R = 2\pi n w + \pi w r-\delta/2$ we have
\begin{equation}
\begin{aligned}\label{eq:tildeWtaylor}
    \tilde{W}(w) = 
    W(v) + \brackets{\frac{\sqrt{\delta^2 + 4\tau^2w^2} - \delta}{2} - v}  W^\prime(v(1 + u))\,,
\end{aligned}
\end{equation}
where by \eqref{exp:1}
\[
u \le \frac{1}{v}\bigg\vert\frac{\sqrt{\delta^2 + 4\tau^2w^2} - \delta}{2} - v\bigg\vert \ll \frac{RC}{n} \ll C^2 = o(1)\,.
\]
Given the rate of decay of $W^\prime$ (see \eqref{eq:WtFder_genbound}), the second term in \eqref{eq:tildeWtaylor} satisfies
\begin{equation*}
    \brackets{\frac{\sqrt{\delta^2 + 4\tau^2w^2} - \delta}{2} - v}  W^\prime(v(1 + u)) \ll \begin{cases}
        (\tau+ T)^{-1000}, & \text{if $|v(1 + u)- T_j| \gg T^\varepsilon \Delta$}\,;\\
        \frac{R}{\Delta n}\brackets{\delta + \frac{TR}{n}}, & \text{if $|v(1 + u)- T_j| \ll T^\varepsilon \Delta$}\,.
    \end{cases}
\end{equation*}
The contribution of the first case is negligible. In the second case when $|v(1 + u)- T_j| \ll T^\varepsilon \Delta$
\begin{equation}\label{eq:rangew}
    w = \frac{T_j+ \delta/2 + O(T^\varepsilon \Delta + T^{1+\varepsilon} C^2)}{2\pi (n+r/2)}\,.
\end{equation}
Set $v_j = (T_j+\delta/2)/ (2\pi (n+r/2))$. Then for $w$ satisfying \eqref{eq:rangew} given $\delta \ll T^{-\varepsilon} \min\{\Delta, T^2/\Delta^2\}$
\[
\tilde{f}^\prime (w) - \delta/w = 
\tilde{f}^\prime (v_j) -\delta/v_j + O \brackets{\frac{\delta \Delta n}{T^2}} = \tilde{f}^\prime (v_j) -\delta/v_j + O \brackets{\frac{T^{-\varepsilon}n}{\Delta}}\,.
\]
In particular, integrating $E_2(n,m)$ by parts $l$ times allows us to restrict the double sum $\sum d(n)d(m) E_2(n,m)$ to $n,m$ satisfying $(\tilde{f}^\prime (v_j) - \delta/v_j)/ n\ll T^\varepsilon \Delta^{-1}$, or more conveniently,
\begin{equation}\label{eq:ftildder2restr}
    f^\prime(T_j)\ll T^\varepsilon \Delta^{-1}.
\end{equation}
Using \eqref{eq:ftildeprime1} and \eqref{eq:tildeWtaylor},
\[
\frac{\tilde{W}(w) - W(2\pi n w + \pi w r - \delta/2)}{\tilde{f}^\prime (w) - \delta/w } \ll \frac{1}{\Delta}\brackets{\frac{\delta}{n} + \frac{TC}{n}} \ll \frac{TC}{n\Delta}\,,
\]
which combined with Lemma \ref{lem:oscI} gives us
\begin{equation*}
\begin{aligned}
    E_2(n,m) \ll \frac{T C}{n\Delta} + O((\tau + T)^{-a})\,.
\end{aligned}
\end{equation*}
Invoking the restriction \eqref{eq:ftildder2restr}, we therefore have
\begin{equation}\label{eq:tildeWerror}
\begin{aligned}
    \sum_{nm \ll T^2, \log m/n \ll C} d(n)d(m) E_2(n,m) \ll 
   \frac{T^{1+\varepsilon}C}{\Delta} \sum_{n \ll T, f^\prime (v_j) \ll T^\varepsilon / \Delta} \frac{1}{n} \ll \frac{T^{2+\varepsilon}C}{\Delta^2} \ll T^{1/2+\varepsilon}\,.
\end{aligned} 
\end{equation}

Finally, let us simplify the phase $\tilde{f}(w) - \delta \log w$ to $2\pi w R - \delta \log w$. Integrating by parts,
\begin{equation*}
    \begin{aligned}
E_3(n,m) :=        \int_{w_1}^{w_2} (e^{i\tilde{f}(w)-2\pi i wR}-1) e^{2\pi iwR} w^{-i\delta} W(v(w)) I(w) dw = O(T^{-a}) - \\
        - \int_{w_1}^{w_2}  \frac{e^{i\tilde{f}(w)-2\pi i wR}-1}{2\pi i R} e^{2\pi iwR} w^{-i\delta} \brackets{W(v(w))I(w)}^\prime dw - \\
        - \int_{w_1}^{w_2}  \brackets{\frac{e^{i\tilde{f}(w)-2\pi i wR}-1}{2\pi i R}}^\prime e^{2\pi i w R} w^{-i\delta} W(v(w)) I(w) dw
        \,.
    \end{aligned}
\end{equation*}
Using expansions \eqref{exp:ftilde} and \eqref{exp:ftildeprime}, when $\delta \ll T^{2/3-\varepsilon}$, we have
\[
\tilde{f}(w)-2\pi wR \ll \frac{\delta R^2}{n^2} + \frac{TR^3}{n^3}\ll TC^3= o(1)
\]
and thus
\[
 \frac{e^{i\tilde{f}(w)-2\pi i wR}-1}{2\pi i R} = -\frac{\pi w^2 \delta R}{N^2} +O\brackets{\frac{TR^2}{n^3}}\,.
\]
Applying Lemma \ref{lem:oscI} together with bounds on $W^\prime$ and $I^\prime$
\begin{equation}\label{eq:er1}
    \begin{aligned}
        \int_{w_1}^{w_2} \frac{e^{i\tilde{f}(w)-2\pi i wR}-1}{2\pi i R} e^{2\pi iwR} w^{-i\delta} \brackets{W(v(w)) I(w)}^\prime dw \ll \frac{TC}{n^2 \Delta} \,.
    \end{aligned}
\end{equation}
Expanding
\[
\brackets{\frac{e^{i\tilde{f}(w)-2\pi i wR}-1}{2\pi i R}}^\prime =  -\delta \frac{e^{i\tilde{f}(w)-2\pi i wR}-1}{i(2\pi Rw)^2} +  e^{i\tilde{f}(w)-2\pi i wR} \frac{\tilde{f}^\prime(w) - 2\pi r}{2\pi R} \ll \brackets{\frac{\delta}{T} + \frac{|R|}{n}}^2\,
\]
and applying the oscillation lemma (either Lemma \ref{lem:oscI} or \ref{lem:vdc} depending on the term, similarly to \eqref{eq:E1nmbound}), we can bound
\begin{equation}\label{eq:er2}
    \int_{w_1}^{w_2} \brackets{\frac{e^{i\tilde{f}(w)-2\pi i wR}-1}{2\pi i R}}^\prime e^{2\pi i w R} w^{-i\delta} W(v(w)) I(w) dw \ll \frac{C}{n} +\frac{\delta^{5/2}}{Tn(|\log m/n| + \delta/T)}\,.
\end{equation}
Combining \eqref{eq:er1} and \eqref{eq:er2},
\begin{equation}
\label{eq:tildeferror}
\begin{aligned}
        \sum_{nm \ll T^2, \log m/n \ll C} d(n)d(m) E_3(n,m) \ll \frac{T^\varepsilon\delta^{5/2}}{T} + T^\varepsilon \sum_{n \ll T, r \ll nC} \brackets{\frac{C}{n} + \frac{TC}{n^2 \Delta}} \ll \\
        T^\varepsilon\brackets{\frac{\delta^{5/2}}{T} + TC^2 } \ll T^{\varepsilon} + \frac{T^\varepsilon\delta^{5/2}}{T}\,.
\end{aligned}
\end{equation}

Collecting \eqref{eq:sqrterror}, \eqref{eq:tildeWerror} and \eqref{eq:tildeferror} and observing that $\delta \ll T^{2/3-\varepsilon}$ implies $\delta^{5/2}/T \ll \delta$ yields the statement of the lemma.
\end{proof}
 
Observe that $U_\delta(x)$ is smooth, and for any $k \in \NN$
\[
U_\delta^{(k)}(x) = \int_0^{\infty} w^k W^{(k)}((x+1/2) w-\delta/2) \cos(w) \brackets{\frac{w}{2\pi r}}^{-i\delta} I(w/2\pi r) dw\,. 
\]
While $W^{(k)}((x+1/2) w-\delta/2)$ is (approximately) supported on $w \asymp T/x$, function $I(w/2\pi r)$ restricts the integral to $w \gg r$,  as for $0 < u < 1$
\[
I(u) = \frac{1}{2\pi i} \int_{(1)}u^{2z} G(z) \frac{dz}{z} \ll e^{-\frac{Q}{4} \log^2 u}\,.
\]
Then for small $x$, say $x \ll 1$ integrating by parts $l$ times we have
\[
U_\delta(x) \ll_l (x+1/2)^{l+1}\brackets{\frac{1}{\Delta} + \frac{\sqrt{Q}}{T}+ \frac{\delta}{T} }^l + O(e^{-c_1 (T+r)^2/\Delta^2}I(c_2 T/r))
\]
for some $c_1,c_2 > 0$. On the other hand, when $x \gg T^{1+\varepsilon}$, the intervals $[T_1/x, T_2/x]$ and $[1;+\infty)$ do not overlap. In particular, we immediately see that for any $k \in \ZZ^{\ge 0}$ for some $c_3 > 0$
\begin{equation}\label{eq:Udqinfty}
    U^{(k)}(x) \ll_k e^{-c_3 Q \log^2 (x/T)}\,.
\end{equation}

In the following lemma we prove a more precise version of the observation above.

\begin{lem}\label{lem:Udelta_kder_bound}
Let $x,\delta > 0$ and let $\eta(t) = 1/x - \delta/(t+\delta/2)$, then for any $k \in \NN$
\[
U_\delta^{(k)}(x-1/2) \ll \frac{T}{x^2}\brackets{\frac{T}{x\Delta}}^{k-1} \max_{j = 1,2} \brackets{I(T_j/2\pi xr) e^{-\frac{\Delta^2}{8} \eta^2(T_j)}}\,. 
\]
\end{lem}
\begin{proof} Since $U_\delta(x)$ is absolutely convergent as an integral, we have 
\[
U_\delta^{(k)}(x) =  \int_0^{+\infty} w^k W^{(k)}((x+1/2) w-\delta/2) \cos(w) \brackets{\frac{w}{2\pi r}}^{-i\delta} I(w/2\pi r) dw\,.
\]
For $W(t) = (F_{\Delta} \ast \II_{[T_1,T_2]})(t)$ and $k \in \ZZ^{\ge 0}$
\[
W^{(k+1)}(t) = (F_{\Delta}(t-T_2) - F_{\Delta}(t-T_1))^{(k)} = \brackets{\frac{1}{\Delta\pi^{1/2}} e^{-\frac{(t-T_2)^2}{\Delta^2}}}^{(k)} - \brackets{\frac{1}{\Delta\pi^{1/2}} e^{-\frac{(t-T_1)^2}{\Delta^2}}}^{(k)}\,.
\]
Further, for $X \asymp T$ and $y >0$ change variables to $v = yw - X$, then
\begin{equation}\label{eq:Udelta_kder}
\begin{aligned}
\int_0^{\infty} &w^k F_{\Delta}^{(k)}(y w -X) e^{iw} \brackets{\frac{w}{2\pi}}^{-i\delta} I(w/2\pi r) dw = \\
&\frac{(2\pi)^{i\delta}e^{iX/y}}{X^{i\delta}y^{k+2-i\delta}} \int_0^{\infty} (X+v)^k F_{\Delta}^{(k)}(v) e^{iv/y} (1+ v/X)^{-i\delta} I((X+v)/2\pi r y) dv\,.
\end{aligned}
\end{equation}
Here $F_{\Delta}^{(k)}(v)$ restricts us to $v \ll T^\varepsilon \Delta$. For such $v$ we have $(1+ v/X)^{-i\delta} = e^{-iv\delta/X}(1+ E_1(v))$, where $E_1(v) \ll v^2\delta/T^2$ is an absolutely convergent series. Moreover,
    \[
I((X+v)/2\pi ry) - I(X/2\pi ry)= \sum_{l=1}^\infty \frac{1}{l!}\brackets{\frac{v}{2\pi r y}}^l I^{(l)}(X/2\pi r y) \ll \begin{cases}
1, & \text{if $\log \frac{X}{2\pi r y} \ll \frac{T^\varepsilon\Delta}{T}$}\,,\\
    (y+T)^{-a}, &     \text{otherwise}\,,
\end{cases}
\]
is also absolutely convergent for $\delta \ll T^{-\varepsilon} \min\{\Delta, T^2/\Delta^2\}$, and given $I^{(l)}(w) \ll_l \frac{Q^{l/2}}{w^l}e^{-\frac{Q}{8}\log^2 w}$
\begin{equation*}
    \begin{aligned}
         \int e^{-\frac{v^2}{\Delta^2}} e^{iv\eta(y)}\sum_{l=0}^\infty 
         \frac{v^l I^{(l)}(X/2\pi r y)}{l!(2\pi r y)^l} dv \ll 
         \sum_{l=0}^\infty \frac{|I^{(l)}(X/2\pi r y)|}{l!(2\pi r y)^l}  \int |v|^{l} e^{-\frac{v^2}{\Delta^2}}  dv \ll 
        \Delta^k \sum_{l=0}^\infty \frac{\Delta^l Q^{l/2}}{T^l} 
    \end{aligned}
\end{equation*}
is absolutely convergent as well.
Changing the order of summation and integration and computing the Fourier transform term by term at $\eta(t) = 1/y - \delta/t$, we see that \eqref{eq:Udelta_kder} is bounded by
\[
\frac{X}{y^2}\brackets{\frac{X}{y\Delta}}^{k}I(X/2\pi ry) e^{-\frac{\Delta^2}{8} \eta^2(X)}\,.
\]
Finally, choosing $X = T_j+\delta/2$ and $y = x+1/2$ yields the statement of the lemma.
\end{proof}

Since $U_\delta(x)$ and its derivatives decay sufficiently fast for both $x \to \infty$ and $x \ll 1$, we are now able to apply Theorem \ref{thm:moto_exp} so that
\[
\sum_{n\ge 1}d(n)d(n+r) U_\delta(n/r) = \int_0^\infty m(x,r) U_\delta(x) dx + E(r;U_\delta)\,.
\]
We can think of $n$ and $r$ as restricted to $n \ll T, r \ll R:= T^\varepsilon(T/\Delta + \delta)$, or
\[
x = n/r \gg T^{1+\varepsilon}/R\,.
\]

\subsection{Approximate spectral expansion}\label{subsec:smalld_expansion}
To apply Motohashi's explicit formula effectively and bound the error term $E(r; U_\delta)$ (as in the notation from Section \ref{subsec:spectheory}) and consequently $\sum_{r \ge 1} r^{-1}E(r; U_\delta)$, we will first need to understand functions
\begin{equation*}
\begin{aligned}
        &\Xi (z;U_\delta) =  \frac{\Gamma^2(1/2+z)}{\Gamma(1+2z)} \int_0^\infty U_\delta(x) x^{-1/2-z} {}_2F_1(1/2+z,1/2+z,1+2z;-1/x)dx\,,\\
           & \Theta(y;U_\delta) = \frac{1}{2} \re \brackets{ \brackets{1 + \frac{i}{\sinh{\pi y}}} \Xi(iy;U_\delta)}\,,
\end{aligned}
\end{equation*}
both implicitly depending on $r$.

 For this purpose let us define the following auxiliary functions. Let $s \in \CC$, then
\begin{equation}\label{def:tildewfr}
\tilde{W} (s) :=  \int_1^\infty W (x-\delta/2) x^{s-1}dx, \ \ f_r(s) :=  \frac{(2\pi)^{s}}{4\pi i} \int_{(1)} r^{-z+i\delta} \chi(1-z+i\delta -s) G(z/2) \frac{dz}{ z}\,.
\end{equation}
For $y \in \RR$ and $k \in \NN$ set
\begin{equation}\label{def:smalltheta01}
\begin{aligned}
       & \theta_0(y) :=  B(1/2+iy, 1/2+iy) \tilde{W} (1/2-iy)\,,\\
       & \theta_1(y,k) := 2^{-1-2iy-2k} B(1/2+iy, 1/2+k) \tilde{W} (1/2-iy-2k)\
\end{aligned}
\end{equation}
and
\begin{equation}\label{def:smallnu}
\nu_{\delta}(y) = \brackets{1 + \frac{i}{\sinh{\pi y}}} \theta_0(y) \chi(1/2-iy+i\delta)\,.
\end{equation}

We start with proving the following estimates for $\Xi(z;U_\delta)$.

\begin{lem} \label{lem:thetaxi_smalldelta} Let $X =T^{1+\varepsilon}/\Delta$ and $R = T^\varepsilon (\delta + T/\Delta)$. Then for any $y \in \RR$
\[
\Xi(iy;U_\delta) = \Xi_0(iy;U_\delta) + \Xi_1(iy;U_\delta) + O((T+r + |y|)^{-a})\,,
\]
where $\Xi_0(iy;U_\delta) = \theta_0(y) f_r(1/2+iy)$ is the main term, and $\Xi_1(iy;U_\delta)$ is the error term represented by the asymptotic series with real coefficients
   \[
   \Xi_1(iy;U_\delta) = \sum_{\substack{k \ge 1\\ 1\le l \le 2k}}  c_{l,2k} (1/2-iy)^{l}\theta_1(y,k) f_r(1/2+iy+2k)\,.
\]
In addition, 
\[
\Xi(iy;U_\delta) \ll \max_{|u-\delta/2\pi| \ll \frac{T}{\Delta}}|I(u/2\pi r)|
\begin{cases}
    \frac{\sqrt{T}}{(1+|y|)^{3/2}}, & \text{if $y \ll X$}\,;\\
    O((y+T)^{-a}), & \text{otherwise}\,.
\end{cases}
\]
\end{lem}
\begin{proof}
For $\re c > \re b > 0 $ recall the integral representation of the hypergeometric function,
\[
B(b,c-b) {}_2F_1(a,b,c;-1/x)= \int_0^1 w^{b-1}(1-w)^{c-b-1} (1+w/x)^{a-1} dw
\]
 and rewrite $\Xi(iy;U_\delta)$ as
\begin{equation*}
    \Xi(iy;U_\delta) = \int_0^\infty U_\delta(x) \int_0^1 (w(1-w))^{-1/2+iy} (x+w)^{-1/2-iy} dw\, dx\,,
\end{equation*}
or, integrating by parts $l$ times with respect to $x$,
\begin{equation}\label{eq:xiWl}
    \Xi(iy;U_\delta) = \prod_{j=0}^{l-1} (1/2-iy+j)^{-1} \int_0^\infty U_\delta^{(l)}(x) \int_0^1 (w(1-w))^{-1/2+iy} (x+w)^{l-1/2-iy} dw\, dx\,.
\end{equation}
By Lemma \ref{lem:Udelta_kder_bound}
\begin{equation*}
    U_\delta^{(l)}(x-1/2) \ll \frac{T}{x^2}\brackets{\frac{T}{x\Delta}}^{l-1}I(T_j/2\pi xr) e^{-\frac{\Delta^2}{8} \brackets{\frac{1}{x}-\frac{\delta}{T_j}}^2}\,,
\end{equation*}
so 
\begin{equation}\label{eq:xi_tail}
    \begin{aligned}
        \int_0^{T^{1-\varepsilon}/R} U_\delta^{\prime}(x) dx  \int_0^1 (w(1-w))^{-1/2+iy} (x+w)^{1/2-iy} dw \ll e^{-c_1T^\varepsilon \brackets{\frac{\Delta R}{T}}^2} e^{-c_2 Q \log^2 \frac{T^\varepsilon R}{r}}\,
    \end{aligned}
\end{equation}
for some positive constants $c_1,c_2$. Hence $x$ and $r$ are essentially restricted to $x \gg T/R \gg \min \{\Delta, T/\delta\}$ and $r\ll R$. 

To evaluate the internal integral in \eqref{eq:xiWl}
\[
\int_0^1 (w(1-w))^{-1/2+iy} (x+w)^{l-1/2-iy} dw = \int_0^1 \frac{(x+w)^{l-1/2}}{(w(1-w))^{1/2}} \brackets{\frac{w(1-w)}{x+w}}^{iy} dw
\]
one can use the oscillation lemma for the second derivative or stationary point analysis. Given $x \gg \min\{\Delta, T/\delta\}$, the stationary point of the phase is at 
\[
w_s = \sqrt{x^2+x}-x = 1/2 + O(1/x)\,,
\]
hence by Lemma \ref{lem:vdc}
\begin{equation}\label{eq:2F1bound1}
    \int_0^1 (w(1-w))^{-1/2+iy} (x+w)^{l+1/2-iy} dw \ll (1+|y|)^{-1/2} |x|^{l-1/2}\,.
\end{equation}
Combining \eqref{eq:xiWl} and \eqref{eq:2F1bound1},
\begin{equation*}
\begin{aligned}
  \Xi(iy;U_\delta) \ll_l \frac{T^l}{ (1+|y|)^{l+1/2}\Delta^{l-1}} \int_0^\infty I(T_j/2\pi xr) e^{-\frac{\Delta^2}{8} \brackets{\frac{1}{x}-\frac{\delta}{T_j}}^2} \frac{dx}{|x|^{3/2}} \ll_l \\
\frac{T^{l+1/2}}{ (1+|y|)^{l+1/2}\Delta^{l}R^{1/2}}\max_{|u-\delta/2\pi| \ll \frac{T}{\Delta}}|I(u/2\pi r)|\,.
\end{aligned}
\end{equation*}
In particular, this bound implies that we can restrict to $y \ll X = X(T,\Delta) = T^{1+\varepsilon}/\Delta$, $r \ll R$ and take $l=1$.

Now let us be a bit more careful. Set $a = 1/2+iy$, $b = 1/2-iy$, and assume $y \ll X$, then
\begin{equation*}
    \begin{aligned}
    \int_0^1 (w(1-w))^{a-1} (x+w)^{b} dw = 
    (x+1/2)^{b} \int_0^1 (w(1-w))^{a-1} \brackets{1+\frac{w-1/2}{x+1/2}}^{b} dw = \\
    % (x+1/2)^{b} \int_0^1 (w(1-w))^{a-1} (1+ R_1(w,x,y)) dw = \\
       (x+1/2)^{1/2-iy} (B(1/2+iy, 1/2+iy) + L_1(x,y))\,,
    \end{aligned}
\end{equation*}
where $L_1(x,y)$ is the series representing lower order terms. In particular, $L_1 (x,y)$ is the asymptotic series with real coefficients
\begin{equation*}
\begin{aligned}
        L_1 (x,y)= \sum_{l,k \in \mathcal{I}_0}  \frac{c_{l,k} b^l}{(x+1/2)^k} \int_0^1 (w(1-w))^{a-1}  (w-1/2)^k dw
\end{aligned}
\end{equation*}
where $\mathcal{I}_0 = \{(1,1)\} \cup \{(l,k) \in \NN^2: k\ge 2, \ l \le k\}$. For $k \in \NN$ we have
\[
\int_0^1 (w(1-w))^{a-1}  (w-1/2)^k dw = \begin{cases}
   \frac{B(a,(k+1)/2)}{2^{2a+k}} \ll \frac{1}{1+|y|} & \text{if $k$ is even}\,; \\
    0, & \text{if $k$ is odd}\,.
\end{cases}
\]
Then, given $\delta \ll T^{-\varepsilon} \Delta$ (see assumption \eqref{eq:smalldelta_assum0}),
\[
b/x \ll y/x \ll yR/T \ll T^\varepsilon R/\Delta = o(1)\,,
\]
so the series above is absolutely convergent and can be bounded as
\begin{equation*}
    L_1(x,y) \ll \frac{1+|y|}{(x+1/2)^2} \ll \frac{(1+|y|)R^2}{T^2} \ll \frac{R^2}{T\Delta}  = o(1)\,
\end{equation*}
if $x \gg T/R$. In addition, we can think of $ L_1(x,y)$ as a finite series with $l,k \le N(a)$ up to error $O((R/\Delta)^a)$ for any $a > 0$.

Let
\begin{equation*}
\begin{aligned}
    & g_0(y) =  \frac{B(1/2+iy, 1/2+iy)}{1/2-iy} \int_0^\infty U_\delta^\prime (x) (x+1/2)^{1/2-iy}dx\,;\\
   & g_1(y) =  \frac{1}{1/2-iy} \int_{0}^\infty U_\delta^\prime (x) (x+1/2)^{1/2-iy} L_1(x,y) dx \,.
    \end{aligned}
\end{equation*}
Similarly to the argument for $\Xi(iy;U_\delta)$, both functions are also $O((T+|y|)^{-a})$ for $y \gg X$. Combining this observation with \eqref{eq:xi_tail} yields
\[
\Xi(iy;U_\delta) = g_0(y) + g_1(y)  + O\brackets{ (T+r + |y|)^{-a}}\,
\]
for all $y \in \RR$ and $a > 0$.

By definition, $U_\delta(x)$ and its derivatives are multiplicative convolutions, so by the Mellin convolution theorem for $s\in \CC$ such that $\re s \in (-1,1)$ together with \eqref{eq:xi_tail}
\begin{equation*}
   \begin{aligned}
       \frac{1}{s-1} & \int_0^\infty U_\delta^\prime (x)  (x+1/2)^{s} dx = \\
       & \tilde{W}(s) \int_0^\infty I(u/2\pi r) \cos(u) \brackets{\frac{u}{2\pi r}}^{-i\delta} u^{-s} du + O\brackets{ (T+r+ |y|)^{-a}}
   \end{aligned} 
\end{equation*}
for any $ a > 0$. Note that the integral in $u$ convergent but not absolutely convergent, so we need to be careful when treating it. To avoid this issue, integrate by parts with respect to $u$, then the new double integral (in $u$ and in $z$ in the integral representation of $I$) is absolutely convergent, so changing the order of integration we can then obtain
\begin{equation*}
    \begin{aligned}
        \int_0^\infty I(u/2\pi r) \cos(u) \brackets{\frac{u}{2\pi r}}^{-i\delta} u^{-s}du = -\int_0^\infty \sin(u) \brackets{I(u/2\pi r) \brackets{\frac{u}{2\pi r}}^{-i\delta} u^{-s}}^\prime du = \\
        -\frac{1}{2\pi i} \int_{(1)} \frac{G(z/2) (z-i\delta -s)dz}{(2\pi r)^z z} \int_0^\infty u^{z-s-1} \brackets{\frac{u}{2\pi r}}^{-i\delta} \sin(u) du= \\
        \frac{(2\pi)^{1-s}}{4\pi i} \int_{(1)} r^{-z+i\delta} \chi(s-z+i\delta) G(z/2) \frac{dz}{ z}= f_r(1-s)\,,
    \end{aligned}
\end{equation*}
where $\chi(s)$ is the factor from the functional equation of the Riemann zeta \eqref{eq:chiexplicit}. 

% For lower order terms the computation goes similarly with the difference that for every extra power of $x$ we need to integrate by parts one more time with respect to $u$ inside of the integral definition of $U_\delta(x)$. 
Now, for the lower order terms write
\begin{equation*}
    \begin{aligned}
        g_1(y) = \sum_{l,k \in \mathcal{I}}  \frac{c_{l,2k} (1/2-iy)^{l-1}}{2^{1+2k+2iy}} B(1/2+iy, 
 k+1/2) \int_{0}^\infty U_\delta^\prime (x) (x+1/2)^{1/2-iy-2k} dx\,,
    \end{aligned}
\end{equation*}
where $\mathcal{I} = \{(l,k): k\ge 1, \ l \le 2k\}$. Here for every extra power of $x$ we need to integrate by parts one more time with respect to $w$ inside of the integral definition of $U^\prime_\delta(x)$:
\begin{equation*}
    \begin{aligned}
        U_\delta^\prime (x-1/2) = & \int_0^\infty \cos(w) W^\prime(x w-\delta/2)  \brackets{\frac{w}{2\pi r}}^{-i\delta} w I(w/2\pi r) dw =  \\
        & (-1)^k \int_0^\infty \cos(w) \brackets{W^\prime(x w-\delta/2)  \brackets{\frac{w}{2\pi r}}^{-i\delta} w I(w/2\pi r)}^{(2k)} dw = \\
       &  (-1)^{k+1} \int_0^\infty \sin(w) \brackets{W^\prime(x w-\delta/2)  \brackets{\frac{w}{2\pi r }}^{-i\delta} w I(w/2\pi r)}^{(2k+1)} dw\,.
    \end{aligned}
\end{equation*}
Then using the Mellin convolution theorem and \eqref{eq:xi_tail} for $\re s \in (-1,1)$
\begin{equation*}
    \begin{aligned}
        \frac{1}{s}  \int_0^\infty U_\delta^\prime (x)  (x+1/2)^{s-2k} dx = 
        \tilde{W}(s-2k) f_r(1-s+2k) + O\brackets{(T+r+|y|)^{-a}}\,.
    \end{aligned}
\end{equation*}
Thus choosing $s = 1/2-iy$ up to an error \eqref{eq:xi_tail}
\begin{equation*}
    \begin{aligned}
    & g_0(y) = B(1/2+iy, 1/2+iy) \tilde{W}(1/2-iy)f_r(1/2+iy) = \theta_0(y) f_r(1/2+iy) \,;\\
       &  g_1(y) = \sum_{k,l \in \mathcal{I}}  c_{l,2k} (1/2-iy)^l\theta_1(y,k)f_r(1/2+iy+2k) \,.
    \end{aligned}
\end{equation*}
Computing the corresponding transforms, we have
\[
\tilde{W}(1/2-iy-2k)\ll \frac{T^{1/2-2k}}{1+ |y|} e^{-c_4\frac{\Delta^2 y^2}{T^2}}, \quad f_r(1/2+iy+2k) \ll R^{2k} \max_{|u-\delta/2\pi| \ll \frac{T}{\Delta}}|I(u/2\pi r)| \,
\]
for some constant $c_4 > 0$. Together with the standard bounds
\[
B(1/2+iy, 1/2+iy) \ll (1+|y|)^{-1/2}, \quad B(1/2+iy, 
 1/2+k) \ll 2^{1+2k} (1+|y|)^{-1}\,,
\]
this yields
\begin{equation}\label{eq:thetay}
    \theta_0(y) \ll \frac{T^{1/2}}{(1+ |y|)^{3/2}} e^{-c_4\frac{\Delta^2 y^2}{T^2}}, \quad \theta_1(y,k) \ll \frac{T^{1/2-2k}}{(1+ |y|)^2} e^{-c_4\frac{\Delta^2 y^2}{T^2}}\,,
\end{equation}
and therefore
\[
\Xi(iy;U_\delta) \ll |g_0(y)| + |g_1(y)| \ll e^{-\frac{c_4\Delta^2 y^2}{T^2}} \max_{|u-\delta/2\pi| \ll \frac{T}{\Delta}} |I(u/2\pi r)| \brackets{\frac{T^{1/2}}{(1+ |y|)^{3/2}}  + \frac{R^2T^{-3/2}}{(1+ |y|)^2}}\,.
\]
The second part in the statement of the lemma then follows.
\end{proof}

Observe that the upper bound from Lemma \ref{lem:thetaxi_smalldelta} truncates the sum in the spectral parameter at $X = T^{1+\varepsilon}/\Delta$, and the sum in $r$ at $r \ll R = T^{\varepsilon}(T/\Delta+ \delta)$, so 
\[
\sum_{r \ge 1} \frac{e_j(r;U_\delta)}{r}
\]
are absolutely convergent, and we may change the order of summation and integration. 

\begin{lem}[Continuous spectrum]\label{lem:smalldelta_con} Let $X =T^{1+\varepsilon}/\Delta$ and $R = T^\varepsilon (\delta + T/\Delta)$. Let
\[
E_1 = \frac{1}{\pi} \sum_{r=1}^\infty \frac{1}{r^{1/2}}  \int_{-\infty}^{\infty} \frac{r^{-iy}\sigma_{2iy}(r) |\zeta(1/2+iy)|^4}{|\zeta(1+2iy)|^2} \Theta(y;U_\delta) dy\,.
\]
Let $\nu_{\delta}(y)$ be as defined in \eqref{def:smallnu}. Then 
\[
\begin{aligned}
  E_1 = \frac{1}{\pi} \re \int_{-\infty}^{\infty} \frac{|\zeta(1/2+iy)|^4 \zeta(1/2+i\delta+iy)\zeta(1/2+i\delta-iy)}{|\zeta(1+2iy)|^2} \overline{\nu_{\delta}(y)} dy + O(R^{13/42+\varepsilon})\,.
\end{aligned}
\]
In particular, $E_1 \ll T^{1/2} (1 + \delta^{13/42+\varepsilon})$.
\end{lem}
\begin{proof}
For $s,z \in \CC$ such that $\re s > 1 + \re z$ define
\[
Z_z(s) := \sum_{r \ge 1} \frac{\sigma_{2iy}(r)}{r^{s+iy}} = \zeta(s+z)\zeta(s-z)\,.
\]
This of course extends analytically to the whole complex plane with simple poles at $s = 1 \pm z$. 

 Let $\Xi_m(iy;U_\delta)$ be as in the notation of Lemma \ref{lem:thetaxi_smalldelta}, then
\[
\Theta(y;U_\delta) = \Theta_0(y;U_\delta) + \Theta_1(y;U_\delta) + O((r+y+T)^{-a})\,,
\]
where
\[
\Theta_m(y;U_\delta) =   \frac{1}{2} \re \brackets{ \brackets{1 + \frac{i}{\sinh{\pi y}}} \Xi_m(iy;U_\delta)}, \quad (m=0,1)\,.
\]
Note that both functions implicitly depend on $r$. Let 
\[
\begin{aligned}
    E_1^{(m)} := \frac{1}{\pi} \sum_{r=1}^\infty \frac{1}{r^{1/2}}  \int_{-\infty}^{\infty} \frac{r^{-iy}\sigma_{2iy}(r) |\zeta(1/2+iy)|^4}{|\zeta(1+2iy)|^2} \Theta_m(y;U_\delta) dy\,.
\end{aligned}
\]
Due to the rate of decay of $\Xi_m$ as shown in Lemma \ref{lem:thetaxi_smalldelta}, we can swap the order of integration over $y$ and summation over $r$. Additionally, observe that for real coefficients $a(r)$ the series
\begin{equation}\label{eq:Ziyreal}
    \sum_{r\ge 1} \frac{\sigma_{2iy}(r) a(r)}{r^{1/2+iy}}  = \sum_{e,f\ge 1} \frac{a(ef)}{e^{1/2-iy}f^{1/2+iy}} 
\end{equation}
is in fact also real. Then
\[
\begin{aligned}
    E_1^{(m)} = &\frac{1}{\pi}   \int_{-\infty}^{\infty} \frac{|\zeta(1/2+iy)|^4}{|\zeta(1+2iy)|^2} \sum_{r=1}^\infty \frac{\sigma_{2iy}(r) }{r^{1/2+iy}}\Theta_m(y;U_\delta) dy = \\
    &  \re \frac{1}{\pi} \int_{-\infty}^{\infty} \frac{|\zeta(1/2+iy)|^4}{|\zeta(1+2iy)|^2} \brackets{1 + \frac{i}{\sinh{\pi y}}} \sum_{r=1}^\infty \frac{\sigma_{2iy}(r) }{r^{1/2+iy}}\Xi_m(y;U_\delta)\,.
\end{aligned}
\]
For $k \in \ZZ^{\ge 0}$ and $y \in \RR$ let 
\[
\begin{aligned}
    L_\delta(y,k) := &\sum_{r\ge 1} \frac{\sigma_{2iy}(r)}{r^{1/2+iy}}f_r(1/2+iy-2k) = \\
    & \frac{1}{2\pi i} \int_{(1)} Z_{iy}(1/2 + z-i\delta) \chi(1/2-z+i\delta -iy-2k) G(z/2) \frac{dz}{ z}\,.
\end{aligned}
\]
Then
\begin{equation*}
    \begin{aligned}
&\sum_{r=1}^\infty \frac{\sigma_{2iy}(r) }{r^{1/2+iy}}\Xi_0(iy;U_\delta) = \theta_0(y) L_\delta(y,0)\,,\\
    & \sum_{r=1}^\infty \frac{\sigma_{2iy}(r) }{r^{1/2+iy}}\Xi_1(iy;U_\delta) =  \sum_{\substack{k \ge 1\\ 1\le l \le 2k}}  c_{l,2k} (1/2-iy)^{l}\theta_1(y,k) L_\delta(y,k)\,.
    \end{aligned}
\end{equation*}
When $k = 0$, apply the functional equation of $\zeta(s)$ \eqref{eq:feqzeta}, then
\[
L_\delta(y,0) = \frac{1}{2\pi i} \int_{(1)} Z_{z-i\delta}(1/2-iy) G(z/2) \frac{dz}{ z}\,.
\]
The integrand has a pole at $z=0$ as well as $z = i\delta \pm (1/2+iy)$. Moving the line of integration to $\re z = -1$ and changing variables from $z$ to $-z$,
\begin{equation*}
    \begin{aligned}
        L_\delta(y,0) = Z_{i\delta}(1/2-iy) + \zeta(-2iy) \sum_{z=i\delta \pm (1/2+iy)} \frac{G(z/2)}{z}- 
        \frac{1}{2\pi i} \int_{(1)} Z_{z+i\delta}(1/2-iy)G(z/2) \frac{dz}{ z}\,.
        % \frac{1}{2\pi i} \int_{(1)} \zeta(1/2+z+i\delta -iy) \zeta(1/2-z-i\delta -iy) G(z/2) \frac{dz}{ z}\,.
    \end{aligned}
\end{equation*}
The latter integral is simply the conjugate in $\delta$, and the residues at $z = i\delta \pm (1/2+iy)$ cancel out with those at $z = -i\delta \pm (1/2+iy)$ of the conjugate term. Hence
\begin{equation*}
\begin{aligned}
    E_1^{(0)} = &\re \frac{1}{\pi} \int_{-\infty}^{\infty} \frac{|\zeta(1/2+iy)|^4}{|\zeta(1+2iy)|^2} \brackets{1 + \frac{i}{\sinh{\pi y}}} \sum_{r=1}^\infty \frac{\sigma_{2iy}(r) }{r^{1/2+iy}}\Xi_0(y;U_\delta) =\\
    &\re \frac{1}{\pi} \int_{-\infty}^{\infty} \frac{|\zeta(1/2+iy)|^4}{|\zeta(1+2iy)|^2} \brackets{1 + \frac{i}{\sinh{\pi y}}} L(y,0) \theta_0(y) dy = \\
   & \re \frac{1}{\pi} \int_{-\infty}^{\infty} \frac{|\zeta(1/2+iy)|^4 Z_{i\delta}(1/2-iy)}{|\zeta(1+2iy)|^2} \brackets{1 + \frac{i}{\sinh{\pi y}}} \theta_0(y) dy
    \,.
\end{aligned}
\end{equation*}
To bound this term, use the bound on $\theta_0(y)$ from \eqref{eq:thetay} together with the bounds $\zeta(1+2it) \gg 1+ 1/|t|$ and $\zeta(1/2+it) \ll t^{b+\varepsilon}$ (the Weil bound $b=1/6$ suffices, but the smallest known value is $b = 13/84$ \cite{bourgain2017decoupling}). Then for any $Y > 0$
\begin{equation*}
    \begin{aligned}
    \int_{Y}^{2Y} & \frac{|\zeta(1/2+iy)|^4 Z_{i\delta}(1/2-iy)}{|\zeta(1+2iy)|^2} \brackets{1 + \frac{i}{\sinh{\pi y}}} \theta_0(y) dy \ll \\
   & T^{1/2}\frac{(1+ Y+ \delta)^{2b+\varepsilon}}{(1+Y)^{3/2}}\int_{Y}^{2Y} |\zeta(1/2+iy)|^4 dy \ll T^{1/2}((1+Y)^{2b-1/2 +\varepsilon} + 
    \delta^{2b + \varepsilon}(1+Y)^{-1/2+\varepsilon})\,,
    \end{aligned}
\end{equation*}
and, summing over dyadic intervals, $E_1^{(0)}  \ll      T^{1/2}(1 + \delta^{2b+\varepsilon})$.

Moving on to the error term $E_1^{(1)}$, when $k \in \NN$ using $\zeta(1/2+it) \ll t^{b+\varepsilon}$ and
\[
\chi(1/2-z+it-2k) \ll (1+ |t-\im z|)^{2k+\re z}\,,
\]
we trivially have $L_\delta(y,k) \ll R^{2b+2k+ \varepsilon}$. Further, $\theta_1(y,k) \ll T^{1/2-2k} (1+|y|)^{-2} e^{-c\Delta^2y^2/T^2}$ by \eqref{eq:thetay}, so we can bound the contribution of $\Xi_1(iy;U_\delta)$ by
\begin{equation*}
    \begin{aligned}
       E_1^{(1)} \ll
\sum_{\substack{k \ge 1\\ 1\le l \le 2k}}  |c_{l,2k}|  T^{1/2-2k} 
 \int_{|y| \ll X} \frac{|\zeta(1/2+iy)|^4 |L_\delta(y,k)|}{|\zeta(1+2iy)|} (1 + |y|)^{l-2}  dy\ll \\
 T^{1/2 + \varepsilon} R^{2b}\sum_{\substack{k \ge 1\\ 1\le l \le 2k}}  |c_{l,2k}|  T^{-2k} R^{2k}
 \int_{|y| \ll X} \frac{|\zeta(1/2+iy)|^4}{|\zeta(1+2iy)|} (1 + |y|)^{l-2}dy \ll \\
  T^{1/2+ \varepsilon} R^{2b} \sum_{\substack{k \ge 1}} C_k T^{-2k} R^{2k} X^{2k-1} \ll \frac{T^\varepsilon R^{2+2b}}{\Delta T^{1/2}} \ll R^{2b+\varepsilon} \ll T^{1/2} \,.
    \end{aligned}
\end{equation*}
The statement of the lemma then follows.
\end{proof}

\begin{lem}[Discrete spectrum]\label{lem:smalldelta_disc} Let $X =T^{1+\varepsilon}/\Delta$ and $R = T^\varepsilon (\delta + T/\Delta)$. Let
\[
E_2 =  
 \sum_{r \ge 1} \frac{1}{r^{1/2}} 
\sum_{j=1}^{\infty} \alpha_j t_j(r) H_j^2(1/2) \Theta(\kappa_j;U_\delta)\,.
\]
Let $\nu_{\delta}(y)$ be as defined in \eqref{def:smallnu}. Then
\[
E_2 =  \re \sum_{j=1}^{\infty} \alpha_j H_j^2(1/2) H_j(1/2+i\delta) \overline{\nu_{\delta}(y)} + O(T^{1/2+\varepsilon} (R/\Delta)^2)\,.
\]
In particular, $E_2 \ll T^{1/2+\varepsilon}(\sqrt{X} + \delta^{1/3})$.
\end{lem}
\begin{proof}
Let $\Xi_m(iy;U_\delta)$ be as in the notation of Lemma \ref{lem:thetaxi_smalldelta} ($m = 1,2$), then
\[
\Theta(y;U_\delta) = \Theta_0(y;U_\delta) + \Theta_1(y;U_\delta) + O((r+y+T)^{-a})\,,
\]
where
\[
\Theta_m(y;U_\delta) =   \frac{1}{2} \re \brackets{ \brackets{1 + \frac{i}{\sinh{\pi y}}} \Xi_m(iy;U_\delta)} \ll (1+1/|y|) |\Xi_m(iy;U_\delta)|\,.
\]
Both functions implicitly depend on $r$. Denote the contribution of the respective parts by
\begin{equation*}
    \begin{aligned}
         &E_2^{(m)} : = \sum_{r \ge 1} \frac{1}{r^{1/2}}\sum_{j=1}^{\infty} \alpha_j t_j(r) H_j^2(1/2) \Theta_m(\kappa_j;U_\delta)\,.
    \end{aligned}
\end{equation*}
% \begin{equation*}
%     \begin{aligned}
%          &E_2^{(0)} : = \sum_{r \ge 1} \frac{1}{r^{1/2}}\sum_{j=1}^{\infty} \alpha_j t_j(r) H_j^2(1/2) \theta_0(\kappa_j) f_r(1/2+i\kappa_j)\,,\\
%          &E_2^{(1)} : = \sum_{r \ge 1} \frac{1}{r^{1/2}} \sum_{j=1}^{\infty} \alpha_j t_j(r) H_j^2(1/2)  \sum_{\substack{k \ge 1\\ 1\le l \le 2k}}  c_{l,2k} (1/2-i\kappa_j)^{l}\theta_1(\kappa_j,k) f_r(1/2+i\kappa_j+2k)\,.
%     \end{aligned}
% \end{equation*}
 Both sums are absolutely convergent, so by Fubini's theorem we can change the order of summation over $r$ and $j$ giving
 \[
 E_2^{(m)} = \re \sum_{j=1}^{\infty} \alpha_j  H_j^2(1/2) \brackets{1 + \frac{i}{\sinh{\pi \kappa_j}}} \sum_{r \ge 1} \frac{t_j(r)}{r^{1/2}} \Xi_m(\kappa_j;U_\delta)\,.
 \]

For $k \in \ZZ^{\ge 0}$ and $\sigma > 0$ define
 \begin{equation}
    \begin{aligned}
        L_j(\delta,k) : = \frac{1}{2\pi i} \int_{(\sigma)} H_j(1/2-i\delta+z) \chi(1/2-z+i\delta -i\kappa_j-2k) G(z/2) \frac{dz}{ z},\,
    \end{aligned}
\end{equation}
so that
\begin{equation*}
    \begin{aligned}
        & \sum_{r \ge 1} \frac{t_j(r)}{r^{1/2}} \Xi_0(\kappa_j;U_\delta) = \theta_0(\kappa_j) L_j(\delta,0)\\
        & \sum_{r \ge 1} \frac{t_j(r)}{r^{1/2}} \Xi_1(\kappa_j;U_\delta) = \sum_{\substack{k \ge 1\\ 1\le l \le 2k}}  c_{l,2k} (1/2-i\kappa_j)^{l}\theta_1(\kappa_j,k) L_j(\delta,k)\,.
    \end{aligned}
\end{equation*}
The Hecke $L$-function $H_j(s)$ has no poles and satisfies the functional equation
\begin{equation}\label{eq:Hjfneq}
    H_j(s) = \chi(s+i\kappa_j)\chi(s-i\kappa_j) H_j(1-s)\,.
\end{equation}

When $k=0$, set $\sigma= 1/4$. The factor $\chi(1/2-z+i\delta -i\kappa_j)$ has a single pole at $z=-1/2+i\kappa_j-i\delta$, so moving the line of integration to $\re z = -1/4$ and changing variables from $z$ to $-z$, then using functional equation \eqref{eq:Hjfneq} we see that
\begin{equation*}
   L_j(\delta,0) =  H_j(1/2-i\delta) \chi(1/2+i\delta-i\kappa_j) -\frac{1}{2\pi i} \int_{(1/4)} H_j(1/2+i\delta+z) \chi(1/2-z-i\delta -i\kappa_j) G(z/2) \frac{dz}{ z}\,,
\end{equation*}
where the latter term is in fact $L_j(-\delta,0)$, and thus is the conjugate term. Hence
\begin{equation*}
   E_2^{(0)} = \re \sum_{j=1}^{\infty} \alpha_j H_j^2(1/2) H_j(1/2-i\delta) \nu_\delta(\kappa_j)\,.
\end{equation*}
To bound $E_2^{(0)}$, use the bound on $\theta_0(y)$ from \eqref{eq:thetay} together with the moment bounds \eqref{eq:Hjmoment} and \eqref{eq:hecke2mom}. Then for $K$ such that $0 < \kappa_1 \le K \ll X$ Cauchy's inequality yields
\begin{equation*}
    \begin{aligned}
\sum_{K \le \kappa_j < 2K} \alpha_j H_j^2(1/2) H_j(1/2-i\delta) \nu_\delta(\kappa_j) \ll
\frac{T^{1/2}}{K^{3/2}} \sum_{K \le \kappa_j < 2K} H_j^2(1/2) |H_j(1/2+i\delta)| \ll \\
\frac{T^{1/2}}{K^{3/2}} \sqrt{\sum_{\kappa_j \sim K}\alpha_j |H_j(1/2+i\delta)|^2 \sum_{\kappa_j \sim K} \alpha_j H_j^4(1/2)} \ll 
\frac{T^{1/2} K^{\varepsilon}}{K^{1/2}} ( K + \delta^{1/3+\varepsilon})\,,
    \end{aligned}
\end{equation*}
and\ summing over dyadic intervals,
\begin{equation*}
    \begin{aligned}
        E_2^{(0)} \ll \sqrt{T} \sum_{\kappa_j \ll X} \frac{\alpha_j}{\kappa_j^{3/2}} H_j^2(1/2) |H_j(1/2+i\delta)| \ll T^{1/2+\varepsilon}(\sqrt{X} + \delta^{1/3}) \,.
    \end{aligned}
\end{equation*}

To bound the term corresponding to $\Xi_1(i\kappa_j;U_\delta)$, it will suffice to bound $L_j(\delta,k)$. Observe that the weight $G(z)$ restricts the integral $L_j(\delta,k)$ to $\im z \ll T^{\varepsilon} \sqrt{Q} \ll R = T^{\varepsilon}(T/\Delta + \delta)$ given assumption \eqref{eq:smalldelta_assum} at the beginning of the subsection. Further, for $k \in \NN$ we have
\[
\chi(1/2-z+i\delta -i\kappa_j-2k) \ll (1+ |\kappa_j-\delta+\im z|)^{2k+\re z} \ll R^{2k+\re z}\,.
\]
Then for $z = \sigma + it$ ($\sigma > 0$) and $\kappa_j \ll X \ll R$
\begin{equation}\label{eq:boundLjdeltak}
\begin{aligned}
        L_j(\delta,k) \ll \int_{t \ll R} |H_j(1/2+\sigma-i\delta+it) \chi(1/2-\sigma -2k+i(\delta -\kappa_j-t))| \frac{dt}{\sigma + |t|} \ll \\
       R^{2k+\sigma} \int_{t \ll R} |H_j(1/2+\sigma+it)| \frac{dt}{\sigma + |t|}\,.
        % \max_{t \ll R} \brackets{(1+ |t|)^{\sigma + 2k} |H_j(1/2+\sigma +it)|} \ll R^{2k} \max_{t \ll R} |t^{\sigma} H_j(1/2+\sigma +it)|\,.
\end{aligned}
\end{equation}
% This bound in fact extends to $\sigma = 0$ since we can move the line of integration to $\re z = 0$ with a small semicircle around $z = 0$ giving 
% \begin{equation*}
%     \begin{aligned}
%         \frac{1}{2} \res_{z=0} & \ H_j(1/2-i\delta+z) \chi(1/2-z+i\delta -i\kappa_j-2k)  \frac{G(z/2)}{ z} = \\
%         & \frac{1}{2}H_j(1/2-i\delta) \chi(1/2+i\delta -i\kappa_j-2k) \ll (1+|\kappa_j-\delta|)^{2k} |H_j(1/2-i\delta)|\,.
%     \end{aligned}
% \end{equation*}
Using the bound on $\theta_1(\kappa_j,k)$ from \eqref{eq:thetay} together with \eqref{eq:boundLjdeltak} ($\sigma = \varepsilon$), for $j$ such that $\kappa_j \ll X$
\begin{equation*}
    \begin{aligned}
\sum_{r \ge 1} \frac{t_j(r)}{r^{1/2}} \Xi_1(\kappa_j;U_\delta) \ll &\sum_{\substack{k \ge 1\\ 1\le l \le 2k}} \kappa_j^{l} |\theta_1(\kappa_j,k) L_j(\delta,k)| \ll \\
 &T^{1/2} R^{\sigma} \int_{t \ll R}  \frac{|H_j(1/2+\sigma+it)|dt}{\sigma + |t|} \sum_{\substack{k \ge 1\\ 1\le l \le 2k}} \kappa_j^{l-2} T^{-2k} R^{2k}\ll \\
 & T^{-3/2} R^{2+\sigma} \int_{t \ll R}  \frac{|H_j(1/2+\sigma+it)|dt}{\sigma + |t|}\,.
    \end{aligned}
\end{equation*}
Proceeding similarly to $E_2^{(0)}$ by splitting into dyadic intervals and applying Cauchy's inequality together with \eqref{eq:Hjmoment} and \eqref{eq:hecke2mom}, we can then bound $E_2^{(1)}$ by
\begin{equation*}
\begin{aligned}
    E_2^{(1)} \ll & T^{-3/2} R^{2+\varepsilon} \int_{t \ll R} \frac{dt}{\varepsilon + |t|}\sum_{\kappa_j \ll X} \alpha_j H_j^2(1/2)|H_j(1/2+\varepsilon+it)| \ll \\
     &T^{-3/2+\varepsilon} R^{2} X \int_{t \ll R} \frac{(X + t^{1/3})dt}{\varepsilon + |t|}  \ll 
T^{-3/2+\varepsilon} R^{2} X^{2}\ll T^{1/2+\varepsilon} \brackets{\frac{R}{\Delta}}^{2}\,.
        \end{aligned}
\end{equation*}
By \eqref{eq:smalldelta_assum0}, $R  = T^\varepsilon(T/\Delta + \delta) \ll \min\{\Delta,T^{2/3}\}$, so $E_2^{(1)} \ll T^{1/2+\varepsilon}$, and the statement of the lemma follows.
\end{proof}

\begin{lem}[Residual spectrum]\label{lem:smalldelta_res} Let
\[
E_3=\frac{1}{4} \sum_{r \ge 1} \frac{1}{r^{1/2}} \sum_{k=6}^{\infty} \sum_{j=1}^{\vartheta(k)} (-1)^k \alpha_{j,k} t_{j,k}(r)H_{j,k}^2(1/2) \Xi(k-1/2;U_\delta)\,.
\]
Then $E_3 \ll T^{-\varepsilon}$.
\end{lem}
\begin{proof}
If $z = k-1/2$ and $k \in \NN$ recall the series representation of the hypergeometric function:
\begin{equation*}
    B(k,k) F(k,k,2k;-1/x) = \frac{(k-1)!^2}{(2k-1)!} \sum_{j=0}^{\infty} \frac{(k)_j^2}{(k)_{2j}j!} (-x)^{-j} \ll \frac{(k-1)!^2}{(2k-1)!}\,.
\end{equation*}
This immediately gives us 
\[
\Xi(k-1/2;U_\delta) \ll
        \frac{(k-1)!^2}{(2k-1)!}\int_0^\infty |U_\delta(x)| x^{-k} dx \ll  \brackets{\frac{R}{T}}^{k-2} \max_{|u-\delta/2\pi| \ll \frac{T}{\Delta}}|I(u/2\pi r)|\,.
\]
Then
\begin{equation*}
        E_3=\frac{1}{4} \sum_{r \ge 1} \frac{1}{r^{1/2}} \sum_{k=6}^{\infty} \sum_{j=1}^{\vartheta(k)} (-1)^k \alpha_{j,k} t_{j,k}(r)H_{j,k}^2(1/2) \Xi(k-1/2;U_\delta) \ll R\brackets{\frac{R}{T}}^{4} \ll T^{-\varepsilon}\,,
    \end{equation*}
and we can completely ignore this term from now on.
\end{proof}

Applying Theorem \ref{thm:moto_exp} to 
\[
\sum_{n \ge 1} d(n)d(n+r) U_\delta(n/r)\,,
\]
and gathering Lemmas \ref{lem:smalldelta_con}, \ref{lem:smalldelta_disc} and \ref{lem:smalldelta_res} we obtain the approximate spectral expansion of the error term in our case.

\begin{prop}[Approximate spectral expansion I]\label{lem:smalldelta_correr} 
Let $m(x,r)$ be as in \eqref{eq:mxr} and let
\[
\nu_{\delta}(y) =  \brackets{1 + \frac{i}{\sinh{\pi y}}}  \chi(1/2-iy+i\delta) B(1/2+iy, 1/2+iy) \int_1^\infty W(u-\delta/2) u^{-1/2-iy} du\,.
\]
Then for $U_\delta(x)$ defined in \eqref{defn:Udelta}
\[
 2 \re \sum_{n,r \ge 1}
   \frac{d(n)d(n+r)}{r} U_\delta(n/r) = 2 \re \sum_{r \ge 1} \frac{1}{r} \int_0^\infty m(x,r) U_\delta(x) dx + E_c+ E_d+O(T^{1/2+\varepsilon})\,,
\]
where
\begin{equation*}
\begin{aligned}
   & E_c = \re \frac{1}{\pi}\int_{-\infty}^{\infty} \frac{|\zeta(1/2+iy)|^4 \zeta(1/2+i\delta-iy)\zeta(1/2+ i\delta+iy)}{|\zeta(1+2iy)|^2} \overline{\nu_\delta(y)} dy \,,\\
    & E_d =  \re \sum_{j=1}^{\infty} \alpha_j H_j^2(1/2) H_j(1/2+i\delta) \overline{\nu_{\delta}(\kappa_j)}\,.
    \end{aligned}
\end{equation*}
In particular, $E_c+E_d \ll T^\varepsilon\brackets{T {\Delta^{-1/2}} + T^{1/2}\delta^{1/3}}$.
\end{prop}

\subsection{Explicit formula}
It now remains to compute
\[
 \sum_{r\ge 1} \frac{1}{r}  \int_0^\infty m(x,r) U_\delta(x) dx\,.
\]
In fact, this is somewhat unnecessary since the essence of the problem is in bounding the error term, and we already know the Dirichlet series of $m(x,r)$. Nonetheless, let us finish the computation for the sake of completeness.

\begin{prop}\label{prop:smalldelta_main_offdiag} Let $R = T^{\varepsilon}(T/\Delta + \delta)$, and let
\[
H_\delta (t) = \frac{d^2}{d s^2} \brackets{\frac{1}{\zeta(2+2s)}\brackets{ \zeta(1-i\delta+s)\zeta(1+i\delta +s) -\frac{2s\zeta(1+2s)}{s^2+ \delta^2}} \brackets{\frac{te^{2\gamma_0}}{2\pi}}^s}\bigg\vert_{s=0}\,.
\]
Then
    \begin{equation*}
        \begin{aligned}
          2\re   \sum_{r\ge 1} \frac{1}{r}  \int_0^\infty m(x,r) U_\delta(x) dx =
            \int_0^\infty W(t) H_\delta(t) dt+O(T^\varepsilon R + T^{1+\varepsilon}(\delta^2+Q)^{-1})\,.
        \end{aligned}
    \end{equation*}
\end{prop}
\begin{proof}
First, integrate by parts inside of $U_\delta(x)$ so that
\begin{equation*}
    \begin{aligned}
        U_\delta(x) = \int_0^\infty W((x+1/2) w - \delta/2) \cos(w) \brackets{\frac{w}{2\pi r}}^{-i\delta} I(w/2\pi r) dw = \\
        - \int_0^\infty \sin(w) \brackets{W((x+1/2) w- \delta/2) \brackets{\frac{w}{2\pi r }}^{-i\delta} I(w/2\pi r)}^\prime dw\,.
    \end{aligned}
\end{equation*}
Change variables from $x$ to $t = xw$, then
\begin{equation}\label{eq:Udeltainternalint}
    U_\delta(x) = - \int_0^\infty W(t- \delta/2) dt \int_0^\infty \sin(w) \frac{d}{dw} \brackets{\frac{m(t/w,r)}{w} \brackets{\frac{w}{2\pi r}}^{-i\delta}I(w/2\pi r)} dw \,.
\end{equation}
Recall that
    \begin{equation}
\begin{aligned}
    m(x,r) =     \frac{d^2}{d s^2}\brackets{(xr^{-1}e^{2\gamma_0})^{s} \frac{\sigma_{1+2s}(r)}{\zeta(2+2s)}}\bigg\vert_{s=0}(1+ O(1/|x|))\,.    
\end{aligned}
\end{equation}
Expanding the internal integral in \eqref{eq:Udeltainternalint} into a double integral
\begin{equation*}
    \begin{aligned}
       \int_0^\infty \sin(w) \frac{d}{dw} \brackets{\frac{m(t/w,r)}{w} \frac{1}{2\pi i}\int_{(\sigma)}\frac{G(z/2)dz}{ z} \brackets{\frac{w}{2\pi r}}^{z-i\delta}} dw
    \end{aligned}
\end{equation*}
we see that the resulting integral is absolutely convergent for $\sigma \in (0,1)$, so we may change the order of integration. For $z,s \in \CC$ such that $\re (z-s) \in (0,1)$ we have
\begin{equation*}
    \frac{1}{(2\pi)^{z-i\delta}}\int_0^{\infty} \sin(u) (u^{z-i\delta-s-1})^\prime du = \frac{1}{2(2\pi)^s} \chi(1-z+i\delta+s)\,,
\end{equation*}
so the internal integral in \eqref{eq:Udeltainternalint} equals
\begin{equation*}
    \frac{1}{2} \frac{d^2}{ds^2} \brackets{\brackets{\frac{te^{2\gamma_0}}{2\pi r}}^{s} \frac{\sigma_{1+2s}(r)}{\zeta(2+2s)} \frac{1}{2\pi i} \int_{(\sigma)} \frac{G(z/2) dz}{r^z z} \chi(1-z+i\delta+s)}\bigg\vert_{s=0}(1+ O(R/t))\,.
\end{equation*}
Moving the line of integration from $(\sigma)$ to $(2)$, define the integral weight
\[
g_{s+i\delta}(r) := \frac{1}{2\pi i}\int_{(\sigma)} \frac{G(z/2) dz}{r^z z} \chi(1-z+i\delta+s) = \frac{1}{2\pi i}\int_{(2)} \frac{G(z/2) dz}{r^z z} \chi(1-z+i\delta+s)\,.
\]
Changing the order of summation over $r$ and integration over $t$, it remains to evaluate
\begin{equation*}
    \begin{aligned}
        \sum_{r \ge 1} \frac{\sigma_{1+2s}(r)}{r^{1+s-i\delta}} g_{s+i\delta}(r) = \sum_{r \ge 1} \frac{\sigma_{1+2s}(r)}{r^{1+s-i\delta}} \frac{1}{2\pi i} \int_{(2)} \frac{G(z/2) dz}{r^z} \chi(1-z+i\delta+s) = \\
        \frac{1}{2\pi i}\int_{(2)} \zeta(1-i\delta+s + z) \zeta(z-i\delta -s) \chi(1-z+i\delta+s) \frac{G(z/2)dz}{z} = \\
        \frac{1}{2\pi i}\int_{(2)} \zeta(1-i\delta+s + z) \zeta(1-z+i\delta +s) \frac{G(z/2)dz}{z}\,.
    \end{aligned}
\end{equation*}
Moving the line of integration to ${(-1)}$, it then follows from the functional equation of $\zeta$ that the real part of the integral above is simply the sum of residues at $z = 0$ and $z = \pm s +i\delta$
\begin{equation}
    \zeta(1+s-i\delta)\zeta(1+s+i\delta) -\frac{\zeta(1+2s)G(s+i\delta)}{s+i\delta} - \frac{\zeta(1+2s)G(s-i\delta)}{s-i\delta}\,.
\end{equation}
If $\delta \gg Q^{1/2}$, the second and third terms are negligible. Otherwise
\begin{equation*}
    \frac{G(s+i\delta)}{s+i\delta} + \frac{G(s-i\delta)}{s-i\delta} = \frac{2s}{s^2+\delta^2} +E(s)\,,
\end{equation*}
where for $\delta \ll Q^{1/2}$
\[
E(s)  = \sum_{j=1}^\infty \frac{(s+i\delta)^{2j-1} + (s+i\delta)^{2j-1}}{(2j)!}G^{(2j)}(0) \ll \frac{s}{Q}\,
\]
is an absolutely convergent series. In addition, for $l \in \ZZ^{\ge 0}$ we have 
\[
\frac{d^l}{ds^l}(s^{-1}E(s))\bigg\vert_{s=0} \ll Q^{-1}\,.
\]
Thus for all $\delta > 0$ and $l \in \ZZ^{\ge 0}$
\begin{equation*}
    \begin{aligned}
        \frac{d^l}{ds^l}&\brackets{\sum_{r \ge 1}  \frac{\sigma_{1+2s}(r)}{r^{1+s-i\delta}} g_{s+i\delta}(r) } \bigg\vert_{s=0} = \\
        & \frac{d^l}{ds^l}\brackets{\zeta(1+s-i\delta)\zeta(1+s+i\delta) - \frac{2s\zeta(1+2s)}{s^2+\delta^2}}\bigg\vert_{s=0} + O(T^\varepsilon (\delta^2+Q)^{-1})\,.
    \end{aligned}
\end{equation*}
Collecting everything together,
\begin{equation*}
    \begin{aligned}
        2 \re  & \sum_{r\ge 1} \frac{1}{r}  \brackets{\int_0^\infty m(x,r) U_\delta(x) dx} = \\
        &2 \re\int_0^\infty W(t-\delta/2) \frac{d^2}{ds^2} \brackets{\frac{(te^{2\gamma_0})^s}{\zeta(2+2s)} \sum_{r \ge 1} \frac{\sigma_{1+2s}(r)}{r^{1-i\delta+s}} g_{s+i\delta}(r)}dt  + O(T^\varepsilon R) = \\
        & 2 \re \int_0^\infty W(t) H_\delta(t) dt + O(T^\varepsilon R + T^{1+\varepsilon}(\delta^2+Q)^{-1})\,,
    \end{aligned}
\end{equation*}
and the statement of the proposition follows.
\end{proof}

Collecting Lemma \ref{lem:sweights}, Proposition \ref{thm:diag}, Lemma \ref{lem:smallprep}, Proposition \ref{lem:smalldelta_correr} and Proposition \ref{prop:smalldelta_main_offdiag} yields the following theorem.

\begin{thm}\label{thm:explicit_small} 
Let $T \ge 10$, and let $0 \le \delta  \ll T^{-\varepsilon} \min\{\Delta, T^2/\Delta^2\}$. Let $Q$ and $\Delta$ be parameters satisfying
\[
T^{1/2+\varepsilon} \ll \Delta \ll T^{1-\varepsilon}, \quad T^{\varepsilon} \ll Q \ll T, \quad \Delta Q^{1/2} \ll T^{1-\varepsilon}\,.
\]
Let $W(t)$ be as defined in \eqref{eq:dgauss}. Let
\[
\begin{aligned}
Q_2(t;\delta) = 2\re  \res_{s= 0,i\delta} \frac{\zeta^4(1+s)}{\zeta(2+2s)}  \frac{(t/2\pi)^s}{s-i\delta} + 
   \frac{\partial^2}{\partial s_1 \partial s_2}\, \brackets{h(i\delta,s_1+s_2) \brackets{ \frac{t}{2\pi} }^{s_1}\brackets{\frac{t+\delta}{2\pi} }^{s_2}} \bigg\vert_{s_j=0}\,,
   \end{aligned}
\]
where
\[
h(z,s) = \frac{e^{2\gamma_0 s}}{\zeta(2+2s)} \brackets{\zeta(1+z + s)\zeta(1-z+s) - \frac{2s\zeta(1+2s)}{s^2-z^2}}\,.
\]
Then
\begin{equation*}
    \begin{aligned}
       \int_0^\infty W(t) |\zeta(1/2+it)|^2 |\zeta(1/2+it+i\delta)|^2 = 
    \int_0^\infty W(t)Q_2(t;\delta) dt + E_c+ E_d+ E_0\,,
    \end{aligned}
\end{equation*}
where $E_c, E_d$ are as defined in Proposition \ref{lem:smalldelta_correr} and $E_0 \ll T^\varepsilon\brackets{\Delta + T(\delta^2+Q)^{-1}}$.
\end{thm}

Choosing $\Delta = T^{2/3}$ and $Q = T^{2/3-\varepsilon}$, and summing over intervals $[T_1,T_2] = [2^{-j-1}T, 2^{-j}T]$ for $0 \le j \le \log_2 T$ we obtain Theorem \ref{thm:A} for $\delta \ll T^{2/3-\varepsilon}$.

Let us now move on to the complementary case.

\section{Off-diagonal term for large shifts}\label{sec:large}

In this section we consider large $\delta$ satisfying
\begin{equation}\label{assum:largedelta}
    \delta \gg T^{\varepsilon}(T/\Delta)^2\,.
\end{equation}
For such $\delta$ we set $\tilde{T} = \sqrt{T(T+\delta)}$ and assume that parameters $\Delta$ and $Q$ satisfy
\begin{equation}\label{assum:largedeltaDQ}
   Q=T^\varepsilon, \quad  (T+\delta)^{1/2+\varepsilon} \ll \Delta \ll T^{1-\varepsilon}\,,
\end{equation}
In this range one can evaluate 
\begin{equation}\label{eq:offdiagintegral}
  \int_0^\infty \big(\frac{m}{n}\big)^{it} \kappa(t)   W(t) J_{\delta}(nm,t) dt
\end{equation}
using stationary phase analysis similar to Lemma \ref{lem:fresnel}. After obtaining a convenient expression for this integral, one can then use Motohashi's formula for weighted divisor correlations to evaluate the resulting double Dirichlet series.

In this section the computation does not rely on a specific choice of $W$ at all, but the reader can think of it as
\begin{equation*}
    p(x) = \exp\brackets{-\frac{1}{x} \exp \brackets{-\frac{1}{1-x}}}, \quad W(t) = \begin{cases}
    p((t-T_1)/\Delta), & {T_1 \le t \le T_1 + \Delta}\,;\\
    p((T_2-t)/\Delta), & {T_2-\Delta \le t \le T_2}\,;\\
    % p((2T-t)/\Delta), & {2T - \Delta \le t \le 2T}
    1, & {T_1 + \Delta \le t \le T_2 - \Delta}\,;\\
    0, & \text{otherwise}\,.
    \end{cases}
\end{equation*}

\subsection{Transforming the weights}
As with the case of small shifts, 
we start with reducing \eqref{eq:offdiagintegral} to a form compatible with Theorem \ref{thm:moto_exp}. For this purpose, define
\begin{equation}\label{def:UMd}
\begin{aligned}
    &W_{\delta}(x) :=  \sqrt{2\pi \delta} \sum_{j=0}^{\infty} \frac{i^j }{ j! 2^j} W^{(2j)}(\delta x)(\delta x(x+1))^j\,,\\
    &U_{\delta}(x) :=  U_{\delta}(x,r) = e^{\frac{\pi i}{4}-i\delta \log \frac{\delta}{2\pi e r}} W_{\delta}(x)\,.
\end{aligned}
\end{equation}
Note that given assumption \eqref{assum:largedelta} this series is absolutely convergent. Then we have the following reduction.

\begin{lem}\label{lem:bigdelta_statphase} 
Let $\delta \gg  T^\varepsilon (T/\Delta)^2$. Then
\begin{equation*}
\begin{aligned}
       \sum_{n,m \ge 1} \frac{d(n)d(m)}{n^{1/2}m^{1/2-i\delta}}\int_{0}^{\infty}& \big(\frac{m}{n}\big)^{it} \kappa(t) W(t)  J_{\delta}(nm,t) dt= \\
     & \sum_{ r \ge 1} \frac{I(\delta/2\pi r)}{r} \sum_{n \ge 1} d(n) d(n+r) U_{\delta}(n/r) + O(T^\varepsilon(\Delta + \delta^{1/2}))\,.
\end{aligned}
\end{equation*}
% where $E_M(T, \delta, \Delta) \ll T^\varepsilon\brackets{T \delta^{-1/2} + \delta^{1/2} + \tilde{T} N^{M-1/2}}$.
% where $E(T, \delta, \Delta) \ll T^\varepsilon\brackets{T \delta^{-1/2} + \delta^{1/2}}$.
\end{lem}
\begin{proof}
Let $f(t) = t \log m/n + \phi(t) + \delta \log m$. Let $r = m-n$, and let $s$ be the stationary point of $f(t)$. Then 
\begin{equation*}
     f^\prime(s) = \log m/n + \phi^\prime(s) = \log \frac{ms}{n(s+\delta)}= \log \frac{(n+r)s}{n(s+\delta)} =   0\
\end{equation*}
gives us
\[
s = \frac{n\delta}{r},\quad f(s) = -\delta \log \frac{\delta}{2\pi e r},\quad f^{\prime\prime}(s) = \phi^{\prime\prime}(s)= \frac{r^2}{mn\delta}= \frac{r^2}{n(n+r)\delta}\,.
\]
Suppose $s \notin [T_1/2,2T_2]$, then $f^\prime(t)$ does not vanish on the support of $WJ$. Let 
\[
\begin{aligned}
    & \lambda := \min_{t \in [T_1,T_2]} |f^\prime(t)| \asymp |\log m/n| + \log (1 + \delta/T), \\
    & \lambda^\prime := \max_{t \in [T_1,T_2]} |f^{\prime\prime}(t)| \asymp \phi^{\prime\prime}(T) \asymp \frac{\delta}{\tilde{T}^2}\,.
\end{aligned}
\]
Integrating by parts $l$ times and applying the oscillation lemma (Lemma \ref{lem:oscI}), we see that
\[
\int \big(\frac{m}{n}\big)^{it} \kappa(t) W(t) J_{\delta}(nm,t) dt \ll_l \lambda^{-1} N^l\max |W(t) J_{\delta}(nm,t)|\,,
\]
where, given our assumptions \eqref{assum:largedelta} and \eqref{assum:largedeltaDQ},
\[
N = \frac{1}{\Delta\lambda} + \frac{\sqrt{Q}}{T\lambda} + \frac{\lambda^\prime}{\lambda^2} \ll T^{-\varepsilon}\,,
\]
thus the contribution of $m,n$ such that $s = n\delta/(m-n) \notin [T_1/2,2T_2]$ is $O(T^{-a})$. 

From now assume that $s \asymp T$. Let $k(u) = k_s(u) = \phi(s+u)- u\phi^{\prime}(s)- \frac{u^2}{2}\phi^{\prime\prime}(s)$, then 
\[
m^{i\delta}\int_0^\infty \big(\frac{m}{n}\big)^{it} \kappa(t) W(t) J_{\delta}(nm,t) dt = e^{if(s)} \int_{-s}^\infty e^{i\frac{u^2}{2}f^{\prime\prime}(s)}  W(s+u) J_{\delta}(nm,s+u)  e^{ik(u)} du\,.
\]
First, we would like to remove the factor $e^{ik(u)}$. Integrating by parts twice,
\begin{equation}
\begin{aligned}
    \int e^{i\frac{u^2}{2}f^{\prime\prime}(s)}  W(s+u) J_{\delta}(nm,s+u)  (e^{ik(u)}-1) du = 
   -\frac{1}{f^{\prime\prime}(s)^2}\int e^{i\frac{u^2}{2}f^{\prime\prime}(s)}  g(u) du\,,
\end{aligned}
\end{equation}
where
\[
g(u) = \brackets{\frac{1}{u}\brackets{\frac{W(s+u) J_{\delta}(nm,s+u)  (e^{ik(u)}-1)}{u}}^\prime}^\prime\,.
\]
Our integral is essentially restricted to $u \ll T^\varepsilon/\sqrt{f(s)} \ll  T^\varepsilon s/\sqrt{\delta}$ (otherwise integrate by parts as many times as needed). For $u \ll  T^\varepsilon s/\sqrt{\delta}$ we have
\[
\begin{aligned}
   g(u) \ll \max_{u \ll  T^\varepsilon s/\sqrt{\delta}} \bigg\vert WJ(s+u) \frac{d^4 e^{ik(u)}}{du^4}\bigg\vert + 
\max_{u \ll  T^\varepsilon s/\sqrt{\delta}} \bigg\vert (WJ)^\prime(s+u) \frac{d^3 e^{ik(u)}}{du^3}\bigg\vert + \\
|u|\max_{u \ll  T^\varepsilon s/\sqrt{\delta}} \bigg\vert (WJ)^{\prime\prime}(s+u) \frac{d^3 e^{ik(u)}}{du^3 }\bigg\vert\,.
\end{aligned}
\]
Using
\[
k(u) = \sum_{j \ge 3}\frac{u^j}{j!}\phi^{(j)}(s),\quad \phi^{(j)}(s) \ll \frac{(j-1)! \delta }{s^{j-1}(s+\delta)}\ll \frac{(j-1)! \delta}{T^{j-2}\tilde{T}^2}\,,
\]
we then see that
\[
\begin{aligned}
g(u) \ll \frac{\delta}{T^2 \tilde{T}^2}\max_{u \ll  T^\varepsilon s/\sqrt{\delta}} \bigg\vert (WJ)(s+u) \bigg\vert + \frac{\delta}{T\tilde{T}^2}\max_{u \ll  T^\varepsilon s/\sqrt{\delta}} \bigg\vert (WJ)^\prime(s+u)  \bigg\vert + \\
\frac{|u|\delta}{T \tilde{T}^2}\max_{u \ll  T^\varepsilon s/\sqrt{\delta}} \bigg\vert (WJ)^{\prime\prime}(s+u) \bigg\vert\,.
\end{aligned}
\]
Then by oscillation lemma for the second derivative (Lemma \ref{lem:vdc}) in the first two cases, and for the first derivative (Lemma \ref{lem:oscI}) in the third case,
\[
\begin{aligned}
    \int e^{i\frac{u^2}{2}f^{\prime\prime}(s)}  g(u) du \ll \frac{\sqrt{\delta}}{T^2 \tilde{T}}\max_{u \ll  T^\varepsilon s/\sqrt{\delta}} \bigg\vert (WJ)(s+u) \bigg\vert  + 
   \frac{\sqrt{\delta}}{T\tilde{T}} \max_{u \ll  T^\varepsilon s/\sqrt{\delta}} \bigg\vert (WJ)^\prime(s+u)  \bigg\vert + \\
\frac{1}{T}\max_{u \ll  T^\varepsilon s/\sqrt{\delta}} \bigg\vert (WJ)^{\prime\prime}(s+u) \bigg\vert   \,.
\end{aligned}
\]
Plugging this bound into the sum over $m,n$,
\[
\begin{aligned}
    \sum_{\substack{nm \ll \tilde{T}^2\\ s \asymp T}} \frac{d(n)d(m)}{\sqrt{mn}f^{\prime\prime}(s)^2} \bigg\vert\int e^{i\frac{u^2}{2}f^{\prime\prime}(s)} g(u) du\bigg\vert \ll 
\frac{(T+\delta)^{3/2+\varepsilon}}{T^{1/2}\delta^{3/2}}\sum_{\substack{nm \ll \tilde{T}^2\\ s \asymp T}} \frac{1}{\sqrt{mn}} + \\
\frac{T^{1/2}(T+\delta)^{3/2+\varepsilon}}{\delta^{3/2}\Delta}\sum_{\substack{nm \ll \tilde{T}^2\\ |s-T_j| \ll T^\varepsilon\Delta}} \frac{1}{\sqrt{mn}}\ll T^\varepsilon \brackets{\frac{T + \delta}{\sqrt{\delta}}}\,.
\end{aligned}
\]

Next, use the asymptotic expansion near the stationary point from Lemma \ref{lem:fresnel},
\begin{equation}
\begin{aligned}
 \int e^{i\frac{u^2}{2}f^{\prime\prime}(s)}  W(s+u) J_{\delta}(nm,s+u)  du \sim
   \sqrt{2\pi} e^{i\pi/4} \sum_{j=0}^\infty \frac{i^j (WJ)^{(2j)}(s)}{j! 2^j(f^{\prime\prime}(s))^{1/2+j}} \,.
   % = \sqrt{2\pi} e^{\pi i/4}\frac{W(s) I(\delta/2\pi r)}{\sqrt{f^{\prime\prime}(s)}} + R_1(m,n)\,.
\end{aligned}
\end{equation}
The asymptotic series is in fact absolutely convergent as $(WJ)^{(2j)}(s) \ll \Delta^{-2j}$ and
\begin{equation}\label{eq:convcond}
  \frac{(WJ)^{(2j)}(s)}{(f^{\prime\prime}(s))^j} \ll \brackets{\frac{T(T+\delta)}{\Delta^2 \delta }}^j \ll \brackets{T^{-\varepsilon}}^j\,
\end{equation}
meaning that we can truncate the series at some finite $M\in \NN$ depending on the desired size of the error term. Let
\[
\begin{aligned}
   &  W_M (s)=  \sqrt{2\pi} \sum_{j=0}^{M-1} \frac{i^j (WJ)^{(2j)}(s)}{j! 2^j (f^{\prime\prime}(s))^{1/2+j}}\,,\\
    &  R_M(s) = \int e^{i\frac{u^2}{2}f^{\prime\prime}(s)}  W(s+u) J_{\delta}(nm,s+u)  du - W_M (s)
   \ll \frac{(WJ)^{(2M)}(s)}{f^{\prime\prime}(s)^{1/2+M}}\,,
\end{aligned}
\]
then
\begin{equation*}
    \begin{aligned}
      \sum_{n,m \ge 1} \frac{R_M(s)}{\sqrt{nm}}  \ll 
      \frac{1}{\Delta^{2M}} \sum_{\substack{nm \ll \tilde{T}^2\\ |s-T_j| \ll T^\varepsilon \Delta}} \frac{1}{\sqrt{mn}f^{\prime\prime}(s)^{1/2+M}} + \frac{Q^M}{T^{2M}} \sum_{\substack{nm \ll \tilde{T}^2\\ s \asymp T}} \frac{1}{\sqrt{mn}f^{\prime\prime}(s)^{1/2+M}}\ll  \\
      \Delta \sqrt{\delta} \brackets{\frac{\tilde{T}^{2}}{\Delta^2 \delta}}^M + Q^M \brackets{\frac{T+\delta}{T\delta}}^{M-1/2} \,.
    \end{aligned}
\end{equation*}
Given $Q = T^\varepsilon$ and $\delta \gg T^{2+\varepsilon}/\Delta^2$ the above is as small as necessary when $M$ is large enough.

Our final step is to remove all derivatives of $J$ as they have little impact on the sum. Observe that
\[
(WJ)^{(2j)}(s) - W^{(2j)}(s) J(mn,s)\ll e^{-\frac{Q}{4}\log^2 \frac{\delta}{2\pi r}}\brackets{\frac{Q^j}{T^{2j}} W(s) + \frac{\sqrt{Q}}{T\Delta^{2j-1}} \II\{|s - T_j| \ll T^\varepsilon\Delta\}}\,,
\]
and for any $j\ge 1$
\[
\begin{aligned}
\frac{\sqrt{Q}}{T\Delta^{2j-1}}\sum_{\substack{m,n\\|s-T_j| \ll T^\varepsilon \Delta}} \frac{ e^{-\frac{Q}{4}\log^2 \frac{\delta}{2\pi r}}}{\sqrt{mn}f^{\prime\prime}(s)^{1/2+j}}  + 
\frac{Q^j}{T^{2j}} \sum_{m,n} \frac{W(s) e^{-\frac{Q}{4}\log^2 \frac{\delta}{2\pi r}}}{\sqrt{mn}f^{\prime\prime}(s)^{1/2+j}}\ll \\
\frac{T}{\sqrt{\delta}} \brackets{\frac{\tilde{T}^{2}}{\Delta^2 \delta}}^{j-1} + \brackets{\frac{Q(T+\delta)}{T\delta}}^{j-1/2}\,.
\end{aligned}
\]
This allows us to remove all derivatives of $J$ at the cost
\[
\begin{aligned}
    \sum_{m,n}\frac{d(n)d(m)}{\sqrt{mn}}  \sum_{j=1}^{M-1} 2^j\frac{(WJ)^{(2j)}(s)- W^{(2j)}(s) J(nm,s)}{j! (f^{\prime\prime}(s))^{1/2+j}} \ll T^{\varepsilon}\brackets{\frac{T+\delta}{\sqrt{\delta}}}\,.
\end{aligned}
\]

Finally, observe that $J_{\delta}(nm,s) = I (\delta/2\pi r)$, and the statement of the lemma follows.
\end{proof}

\subsection{Approximate spectral expansion}
Since we chose $W(x)$ to be a bump function,
\[
W_\delta(x) =  \sqrt{2\pi \delta} \sum_{j=0}^{\infty} \frac{i^j }{j!2^j} W^{(2j)}(\delta x)(\delta x(x+1))^j \sim \int_0^\infty W(t) e^{-\frac{i(t-\delta x)^2}{2\delta x(x+1)}} \frac{dt}{\sqrt{x(x+1)}}\,
\]
is clearly one as well, so we can immediately apply Theorem \ref{thm:moto_exp} giving
\begin{equation*}
    \sum_{n\ge 1} d(n) d(n+r) U_\delta(n/r) = \int_0^\infty m(x,r) U_\delta(x) dx + E(r;U_\delta)\,,
\end{equation*}
where $E(r;U_\delta) = e_1(r;U_\delta) + e_2(r;U_\delta) + e_3(r;U_\delta)$ is as in the notation from Section \ref{subsec:spectheory}.

Let
\begin{equation*}
\begin{aligned}
    &     \Xi (z;U_\delta) =  \frac{\Gamma^2(1/2+z)}{\Gamma(1+2z)} \int_0^\infty U_\delta(x) x^{-1/2-z} {}_2F_1(1/2+z,1/2+z,1+2z;-1/x)dx\,,\\
    & \Theta(y;U_\delta) = \frac{1}{2} \re \brackets{ \brackets{1 + \frac{i}{\sinh{\pi y}}} \Xi(iy;U_\delta)}\,.
\end{aligned}
\end{equation*}
Note that both functions depend on $r$ implicitly, but by \eqref{def:UMd}
\[
\Xi (z;U_\delta) = e^{\frac{\pi i}{4}-i\delta \log \frac{\delta}{2\pi e r}} \Xi (z;W_\delta)\,,
\]
where $\Xi (z;W_\delta)$ is independent of $r$. 

Let us show that we can replace $\Xi(z;U_\delta)$ with 
\[
e^{\frac{\pi i}{4}-i\delta \log \frac{\delta}{2\pi e r}} 
e^{-z^2/\delta} \Xi(z;W_1) \sim r^{i\delta} \chi(1/2+i\delta \pm z) \Xi(z;W_1)
\]
where $W_1(x) =  \sqrt{2\pi \delta} W(\delta x)$, and bound the resulting function to see that the sums
\[
\sum_{r\ge 1} \frac{e_j(r;U_\delta)}{r}I(\delta/2\pi r)
\]
are absolutely convergent as double sums in $r$ and corresponding spectral parameters. 

\begin{lem} \label{lem:thetaxi} Let $X = (T+\delta)^{1+\varepsilon}/\Delta$ and $\lambda = \tilde{T}/T \asymp 1 + \sqrt{\delta/T}$. Let $W_1(x) =  \sqrt{2\pi \delta} W(\delta x)$. Then for any $y \in \RR$ and any $a > 0$
\begin{equation*}
    \Xi(iy;U_\delta) \ll 
    \begin{cases} \frac{\tilde{T}^{1/2}}{(\lambda+y)^{3/2}}, & \text{if $y \ll X$}\,;\\
    (|y|+T)^{-a}, & \text{otherwise}\,.
    \end{cases} \quad 
    \Theta(y;U_\delta) \ll 
     \begin{cases} \frac{\tilde{T}^{1/2}(1 + 1/|y|)}{(\lambda+y)^{3/2}}, & \text{if $y \ll X$}\,;\\
    (|y|+T)^{-a}, & \text{otherwise}\,.
    \end{cases}
\end{equation*}
Further, for $y \ll X$
\[
\Xi(iy;U_\delta) = r^{i\delta} \chi(1/2+i\delta \pm iy) \Xi(iy;W_1)  + O\brackets{\frac{\lambda\tilde{T}^{1/2}}{\delta(\lambda + |y|)^{1/2}} +(T+\delta)^{-1/2}}\,.
\]
\end{lem}
\begin{proof}
For $\re c > \re b > 0 $ recall the integral representation of the hypergeometric function
\[
B(b,c-b) {}_2F_1(a,b,c;-1/x)= \int_0^1 w^{b-1}(1-w)^{c-b-1} (1+w/x)^{a-1} dw\,.
\]
For $x > 0$ and $n \in \ZZ$ let
\begin{equation}
    F_n(x) := \int_0^1 (w(1-w))^{-1/2+iy} (x+w)^{n-1/2-iy} dw\,,
\end{equation} 
then integrating by parts $l \in \NN$ times with respect to $x$,
\begin{equation}\label{eq:xiUdlder}
    \Xi(iy;U_\delta) = (-1)^l \prod_{j=0}^{l-1} (1/2-iy+j)^{-1} \int_0^\infty U_\delta^{(l)}(x) F_l(x) dx\,.
\end{equation}
Let us first bound this integral. For $l \in \NN$, $U_\delta^{(l)}(x)$ is only non-zero when $|\delta x - T_m|\le \Delta$ ($m=1,2$), and by \eqref{eq:WtFder_genbound} for such $x$
\begin{equation}\label{eq:UMlbound}
U_\delta^{(l)}(x) \ll_l  \sum_{j\ge 0} \frac{\delta^{j+1/2}}{j!2^j} |W^{(2j+l)}(\delta x)|(x(x+1))^j \ll_l \frac{\delta^{l+1/2}}{\Delta^l}\,.
\end{equation}

To bound the internal integral over $w$, rewrite it as
\begin{equation}\label{eq:2F1l}
    F_l(x) = \int_0^1 \frac{(x+w)^{l-1/2}}{(w(1-w))^{1/2}} e^{iy f(w)} dw, \quad f(w) := \log \brackets{\frac{w(1-w)}{x+w}}\,,
\end{equation} 
and use the oscillation lemmas (Lemma \ref{lem:oscI}, \ref{lem:vdc}). This would go slightly differently depending on whether $|x| \asymp T/\delta \ll 1$ or $|x| \gg 1$. Let $\lambda := \tilde{T}/T \asymp (1 + 1/x)^{1/2}$. We have
\begin{equation}\label{eq:phasefwder}
    f^\prime(w) = \frac{1}{w} - \frac{1}{1-w} - \frac{1}{x+w}, \quad f^{\prime\prime}(w) = -\frac{1}{w^2} - \frac{1}{(1-w)^2} + \frac{1}{(x+w)^2}\,,
\end{equation}
so the stationary point $w_0$ of the phase $f(w)$ is at 
\[
w_0 = \sqrt{x^2+x} - x = \begin{cases}
    1/2 + O(1/x), & \text{if $x \gg 1$}\,,\\
    \sqrt{x} -x + O(x^{3/2}), & \text{if $x \ll 1$}\,.
\end{cases}
\]
In particular, $w_0 \asymp 1/\lambda$, and $f^{\prime\prime}(w_0) \asymp \lambda$.
% \[
% f^{\prime\prime}(w_0) = \begin{cases}
%     -8 + O(1/x), & \text{if $x \gg 1$}\,,\\
%     2/\sqrt{x} + O(1), & \text{if $x \ll 1$}\,.
% \end{cases}
% \]
Split $F_l(x)$ into intervals $|w-w_0| \le |y\lambda|^{-1/2+\varepsilon}$ and $|w-w_0| \ge \min \{1, |y\lambda|^{-1/2+\varepsilon}\}$ (or, even better, $x \le cw_0$ and $x > cw_0$). Applying Lemma \ref{lem:vdc} or the trivial bound in the first case, and Lemma \ref{lem:oscI} or the trivial bound  in the second,
\begin{equation}\label{eq:2F1bound}
\begin{aligned}
    F_l(x) \ll_l (x^2+x)^{\frac{l}{2}-\frac{1}{4}}\min\{w_0^{1/2}, (w_0 \lambda y)^{-1/2}\} + (x+1)^{l-1/2}(1+|y|)^{-1} \ll\\
   (x^2+x)^{\frac{l}{2}-\frac{1}{4}} \brackets{(\lambda+|y|)^{-1/2} + \lambda^{l-1/2} (1+|y|)^{-1}} \,.
    \end{aligned}
\end{equation}
In particular, together with \eqref{eq:UMlbound} this implies that 
\begin{equation*}
\begin{aligned}
  \Xi(iy;U_\delta) \ll_l & (1 + |y|)^{-l}  \int_0^\infty |U_\delta^{(l)}(x) F_l(x)| dx  \ll_l 
 % & \frac{\delta^{l+1/2}}{\Delta^{l}(1 + |y|)^{l+1/2}} \int_{|\delta x - T_j| \le \Delta}  (x^2+x)^{\frac{l}{2}-\frac{1}{4}} dx \ll_l 
  % & \frac{\delta^{l-1/2}}{(1 + |y|)^{l+1/2} \Delta^{l-1}} \max\{T/\delta,\sqrt{T/\delta}\}^{l-1/2} \ll_l 
  \frac{\tilde{T}^{l-1/2}}{(1 + |y|)^{l+1/2} \Delta^{l-1}} + \frac{(T+\delta)^{l-1/2}}{(1 + |y|)^{l+1} \Delta^{l-1}}\,
\end{aligned}
\end{equation*}
so we can restrict to $y \ll X = (T+\delta)^{1+\varepsilon}/\Delta$. Taking $l=1$,
\[
\Xi(iy;U_\delta) \ll \frac{\tilde{T}^{1/2}}{(1 + |y|)^{3/2}} + \frac{(T+\delta)^{1/2}}{(1 + |y|)^{2}}\ll \frac{\tilde{T}^{1/2}}{(1 + |y|)^{3/2}} + \frac{\delta^{1/2}}{(1 + |y|)^{2}} \,.
\]
When $\delta \gg T$ we can do slightly better for $y \ll \lambda \ll \sqrt{\delta/T}$. In this case take $l=0$, then
\begin{equation*}
    \Xi(iy;U_\delta) \ll \int_0^\infty  |U_\delta(\delta x)F_0(x)| dx 
 \ll T\delta^{-1/2}\,.
\end{equation*}
Taking the minimum between two regimes yields
\[
\Xi(iy;U_\delta) \ll \frac{\tilde{T}^{1/2}}{(|y| + \lambda)^{3/2}}\,, 
\]
and a similar bound applies to $\Theta(y;U_\delta)$ with $\Theta(y;U_\delta)  \ll  (1 + 1/|y|) |\Xi(iy;U_\delta)|$.

Now let us take $l=1$ in \eqref{eq:xiUdlder} and study $\Xi(iy;U_\delta)$ more carefully.
% \begin{equation*}
%     \Xi(iy;U_\delta) = -(1/2-iy)^{-1} \int_0^\infty U_\delta^{\prime}(x) F_1(x) dx\
% \end{equation*}
 We have
\begin{equation*}
    \begin{aligned}
        W_\delta^\prime(x) =  &\sqrt{2\pi } \delta^{3/2} \sum_{j\ge 0} \frac{(i\delta)^j }{ j! 2^j} W^{(2j+1)}(\delta x)(x(x+1))^j+ \\
        &\sqrt{2\pi } \delta^{3/2} (x+1/2)\sum_{j\ge 0} \frac{(i\delta)^j }{ j! 2^j} W^{(2j+2)}(\delta x)(x(x+1))^{j}\,,
    \end{aligned}
\end{equation*}
where the second term is smaller in size by the factor of $\frac{\tilde{T}^2}{\Delta^2 \delta} \cdot \frac{\Delta}{T} \ll \frac{\Delta}{T}$ than the first term. Moreover, $W^{(2j+1)}(\delta x)$ restricts $x$ to $|\delta x-T_m| \le \Delta$ ($m=1,2$), so from now on we can ignore derivatives of $x^j(x+1)^j$ at the same cost.

Integrating by parts $2j$ times,
\begin{equation*}
    \begin{aligned}
        \int_0^\infty W^{(2j+1)}(\delta x) F_1(x) dx = \delta^{-2j} \prod_{m=0}^{2j}(1/2-iy-m) \int W^{(2j+1)}(\delta x) F_{1-2j}(x) dx\,.
    \end{aligned}
\end{equation*}
Then taking $l=1$ in \eqref{eq:xiUdlder} and summing over $j\ge 0$ we immediately obtain
\begin{equation}\label{eq:xiu1_1}
    \begin{aligned}
        \Xi(iy;U_\delta)= -\frac{\sqrt{2\pi}\delta^{3/2}}{1/2-iy} \int_0^\infty W^\prime(\delta x) g_0(x) (1+ O(y\lambda/\delta))dx\,,
    \end{aligned}
\end{equation}
where the error $O(y\lambda/\delta)$ comes from the derivatives of $x^j(x+1)^j$ and
\begin{equation*}
    \begin{aligned}
g_0(x) = \int_0^1 (w(1-w))^{-1/2+iy} (x+w)^{1/2-iy} \exp\brackets{-\frac{iy^2 x(x+1)}{2\delta (x+w)^{2}} }dw\,.
    \end{aligned}
\end{equation*} 
Observe that 
\begin{equation*}
    \begin{aligned}
        \frac{x(x+1)}{(x+w)^{2}} - 1 =
        \brackets{1+ \frac{w-w_0}{(x(x+1))^{1/2}}}^{-2} \ll \begin{cases}
            \frac{|w-w_0|}{\sqrt{x(x+1)}}, & \text{if $w-w_0 \le cw_0$}\,,\\
            1/(x+1), & \text{otherwise}\,.
        \end{cases} 
    \end{aligned}
\end{equation*}
Then, similarly to \eqref{eq:2F1bound}, splitting the interval $[0,1]$ into $x \le cw_0$ and $x \ge cw_0$ and applying Lemma \ref{lem:oscI} to each of the intervals, we obtain
\begin{equation}\label{eq:xiu1_2}
    \int_0^1 (w(1-w))^{-1/2+iy} (x+w)^{1/2-iy} \brackets{ e^{-\frac{iy^2 }{2\delta}R_1(w,x)} -1}dw \ll \frac{|y|}{(\delta(T+\delta))^{1/2}}\,.
\end{equation}
Collecting \eqref{eq:xiu1_1} and \eqref{eq:xiu1_2}, we have
\begin{equation*}
    \begin{aligned}
        \Xi(iy;U_\delta)= -\frac{\sqrt{2\pi}\delta^{3/2}e^{-\frac{iy^2}{2\delta}}}{1/2-iy} \int_0^\infty W^\prime(\delta x) F_1(x)dx + R_2 = e^{\frac{\pi i}{4}-i\delta \log \frac{\delta}{2\pi e r} -\frac{iy^2}{2\delta}} \Xi(iy;W_1) + R_2\,,
        % \Xi(iy;U_1) (1+ O(\Delta/T+ y^{3/2}/\lambda^{1/2}\tilde{T}))\
    \end{aligned}
\end{equation*}
where $W_1(x) = \sqrt{2\pi \delta} W(\delta x)$ and
\[
    \begin{aligned}
R_2 \ll \lambda\delta^{1/2} \int_0^\infty |W^\prime(\delta x) F_1(x)|dx + \delta (T+\delta)^{-1/2} \int_0^\infty |W^\prime(\delta x)| dx \ll \\
\frac{\lambda\tilde{T}^{1/2}}{\delta(\lambda + |y|)^{1/2}} +(T+\delta)^{-1/2}\,.
    \end{aligned}
\]
Finally, we have $e^{\frac{\pi i}{4}-i\delta \log \frac{\delta}{2\pi e } -\frac{iy^2}{2\delta}} = \chi(1/2+i\delta\pm iy) (1 + O(y^3/\delta^2))$, hence
\[
\begin{aligned}
\brackets{e^{\frac{\pi i}{4}-i\delta \log \frac{\delta}{2\pi e r} -\frac{iy^2}{2\delta}} - \chi(1/2+i\delta\pm iy)}\Xi(iy;W_1) \ll  \frac{|y|^3}{\delta^2} |\Xi(iy;W_1)| \ll \\ 
\frac{|y|^{3} \tilde{T}^{1/2}}{\delta^2 (\lambda + y)^{3/2}} \ll 
\frac{\lambda\tilde{T}^{1/2}}{\delta(\lambda + |y|)^{1/2}} +(T+\delta)^{-1/2}
\end{aligned}
\]
and the statement of the lemma follows.
\end{proof}

\begin{lem} \label{lem:thetaxik} Let $X = (T+\delta)^{1+\varepsilon}/\Delta$ and $\lambda = \tilde{T}/T \asymp 1 + \sqrt{\delta/T}$. Then for any $k \in \NN$ and some constant $C>0$
% \[
%     \Xi(k-1/2;U_M) \ll \begin{cases}
%    \delta^{1/2} \brackets{\frac{\delta}{4T}}^{k-1}, & \text{if $\delta < T$}\,;\\
%  \frac{T}{\sqrt{\delta}}e^{-Ck/\lambda}, & \text{otherwise}\,.
%     \end{cases}
% \]
\[
    \Xi(k-1/2;U_\delta) \ll 
 \delta^{1/2} \lambda^{-2}e^{-Ck/\lambda}\,.
\]
Moreover, $\Xi(k-1/2;U_\delta) = (\delta/2\pi)^{k-1/2} \chi(k+i\delta) \Xi(k-1/2;W_1) (1+ O(X/\delta))$.
\end{lem}
\begin{proof}
    When $z = k-1/2$ and $k \in \NN$, similarly to Lemma \ref{lem:thetaxi} we have two cases depending on whether $\delta \ll T$ or $\delta \gg T$. When $\delta < T$, the series 
\begin{equation*}
    B(k,k) {}_2F_1(k,k,2k;-1/x) = \frac{(k-1)!^2}{(2k-1)!} \sum_{j=0}^{\infty} \frac{(k)_j^2}{(k)_{2j}j!} (-x)^{-j} \ll \frac{(k-1)!^2}{(2k-1)!} \ll 4^{-k}
\end{equation*}
immediately gives us 
\[
\Xi(k-1/2;U_\delta) = B(k,k) \int_0^\infty U_\delta(x) x^{-k} {}_2F_1(k,k,2k;-1/x)dx\ll \delta^{1/2}
     \brackets{\frac{\delta}{4T}}^{k-1} \ll \delta^{1/2}
    e^{-k}\,.
\]

If $\delta \ge T$, use the integral representation of the hypergeometric function instead. Integrating by parts $l<k$ times with respect to $x$ we have
\begin{equation}\label{eq:xikintegral}
    \Xi (k-1/2;U_\delta) =  \prod_{j=1}^{l} (k-j)^{-1} \int_0^\infty U_\delta^{(l)}(x) dx \int_0^1 (w(1-w))^{k-1} (x+w)^{l-k} dw \\,,
\end{equation}
where $U_\delta^{(l)}(x)$ restricts $x$ to $x \asymp T/\delta = 1/\lambda^2 \ll 1$. Observe that for $x \le 1$
\[
\frac{w(1-w)}{x+w} = \frac{1-w}{1+x/w} = \exp \brackets{\log \frac{1-w}{1+x/w}}\le \begin{cases}
e^{-C_1(w+ x/w)}, & \text{if $ w > x$}\,,\\
e^{-C_2 (w + w/x + \log w/x)}, & \text{if $ w \le x$}\,,
\end{cases}
\]
for some $C_1,C_2 > 0$. Split the internal integral in \eqref{eq:xikintegral} into intervals $[0,x]$ and $[x,1]$. Then
\begin{equation*}
    \begin{aligned}
        \int_0^x (w(1-w))^{k-1} (x+w)^{l-k} dw \ll x^{l-k} \int_0^x w^{k-1} e^{-C_2(k-1)w(1+1/x)}dw \ll x^lk^{-k-1}
    \end{aligned}
\end{equation*}
and
\begin{equation*}
    \begin{aligned}
        \int_x^1 (w(1-w))^{k-1} (x+w)^{l-k} dw \ll \int_x^1 w^{l-1} e^{-C_1(k-1)(w+ x/w)}dw \ll 
        % \int_{\sqrt{x}}^1 w^{l-1} e^{-C_1(k-1)w}dw\ll 
        k^{-l}e^{-C_1k \sqrt{x}}\,.
    \end{aligned}
\end{equation*}
Taking $l = 0,1$ and using the trivial bounds on the size of $U_\delta(x)$ and $U_\delta^\prime(x)$ yields
\begin{equation}
% \label{eq:Xibound}
    \Xi(k-1/2;U_\delta) \ll
    \min\{T \delta^{-1/2}, \delta^{1/2}k^{-2}\} e^{-C_3k\sqrt{T/\delta}} \ll 
    % T\delta^{-1/2} e^{-C_4k\sqrt{T/\delta}} \ll 
    \delta^{1/2} \lambda^{-2} e^{-C_4k/\lambda} \,
\end{equation}
for some positive constants $C_3$ and $C_4$. 

For all $\delta$ given the series representation of $U_\delta$, it similarly follows that the contribution of $j \ge 1$ is bounded by
\begin{equation*}
    \Xi(k-1/2;W_\delta) - \Xi(k-1/2;W_1) \ll \frac{\tilde{T}^2 |\Xi(k-1/2;W_1)|}{\delta \Delta T} \ll  \frac{(T+\delta)|\Xi(k-1/2;W_1)|}{\delta \Delta}\,.
\end{equation*}
Finally, since $k \ll 1+ (\delta/T)^{1/2}$, we can insert a factor of $e^{-ik^2/\delta}$ in a trivial manner as
\begin{equation*}
    e^{\frac{\pi i}{4}-i\delta \log \frac{\delta}{2\pi e } -\frac{z}{2\delta}} = (\delta/2\pi)^{\re z}\chi(1/2+i\delta + z) (1 + O(|z|^3/\delta^2))
\end{equation*}
for $z = k-1/2$.
\end{proof}

When changing the order of summation in $\sum_{r} e_j(r;U_\delta)I(\delta/2\pi r)/r$, the following series occur:
\begin{equation}\label{def:Lfns}
    \begin{aligned}
      & Z_{iy}(1/2+it) := \sum_{r=1}^\infty \frac{\sigma_{2iy}(r)I(|t|/2\pi r)}{r^{1/2+it+iy}} \,;\\
    & L_j(1/2+it) := \sum_{r=1}^\infty \frac{t_j(r)I(|t|/2\pi r)}{r^{1/2+it}}\,;\\
       & L_{j,k}(1/2+it): = \sum_{r=1}^\infty \frac{t_{j,k}(r)I(|t|/2\pi r)}{r^{1/2+it}} \,.
    \end{aligned}
\end{equation}
The weight function $I(|t|/2\pi r)$ essentially truncates the sum at $r \le |t|/2\pi$, hence functions $Z_{iy}, L_j, L_{j,k}$ defined above are approximations of $\zeta(1/2 + it+iy)\zeta(1/2+it-iy)$, $H_j(1/2+it)$ and $H_{j,k}(1/2+it)$ respectively. Replacing the approximations with the corresponding $L$-functions is not necessary to obtain the bounds on the error terms (and in fact does not affect them at all). The goal of this is only to obtain a nice closed form.

\begin{lem}\label{lem:Lfuncapprox} Let $t > 1$ and let $X_z(s) = \chi(1/2+it+z)\chi(1/2+it-z)$. Then, assuming that $y, \kappa_j, k \ll t^{1/2}$,
\begin{equation*}
    \begin{aligned}
    & \zeta(1/2+it+iy)\zeta(1/2+it-iy) = Z_{iy}(1/2+it) +  X_{iy}(1/2+it) Z_{iy}(1/2-it)
       + O \brackets{t^{-2/3+\varepsilon}}\,,\\
    & H_j(1/2+it) =  L_j(1/2+it)  + \chi_j(1/2+it) L_j(1/2-it)   + O \brackets{t^{-2/3+\varepsilon}}\,,\\
       & H_{j,k}(1/2+it) =  L_{j,k}(1/2+it)  + X_{k-1/2}(1/2+it) L_{j,k}(1/2-it) + O \brackets{t^{-2/3+\varepsilon}}\,.
    \end{aligned}
\end{equation*}
\end{lem}
\begin{proof}
Recall the definition of $I(w)$ given in \eqref{def:Iw}, then
\[
L_j(1/2+it) = \frac{1}{2\pi i} \int_{(1)} H_j(1/2+it+z) \brackets{\frac{t}{2\pi}}^{z} \frac{G(z/2)dz}{z}\,.
\]
Moving the line of integration to $\re z = -1$, we only pick up the pole at $z=0$, so
\begin{equation*}\label{eq:Ljcon}
    \begin{aligned}
         L_j(1/2+it) = 
        & H_j(1/2+it) + \frac{1}{2\pi i }\int_{(-1)} H_j(1/2+it+z) \brackets{\frac{t}{2\pi}}^{z} \frac{G(z/2)dz}{z} = \\
        &  H_j(1/2+it) - \frac{1}{2\pi i }\int_{(1)} H_j(1/2+it-z) \chi_j(1/2+it-z)\brackets{\frac{t}{2\pi}}^{-z} \frac{G(z/2)dz}{z}\,.
    \end{aligned}
\end{equation*}
By Lemma \ref{lem:stirling} for $t > 1$ and $z$ with bounded real part
\begin{equation}\label{eq:chijexp}
    \begin{aligned}
        \chi(1/2 +it -z -i\kappa_j)\chi(1/2 +it -z +i\kappa_j) =  \brackets{\frac{t}{2\pi}}^{2z}\chi_j(1/2+it)\brackets{1+R_j(z,t)}\,,
    \end{aligned}
\end{equation}
where $R_j(z,t) \ll (|z|^2+1)/t + z\kappa_j^2/t^2$ is an absolutely convergent series in powers of $z$ and $t$ coming from the asymptotic series of the $\Gamma$-function.

Applying the functional equation \eqref{eq:Hjfneq} together with \eqref{eq:chijexp}, we obtain
\begin{equation*}
\begin{aligned}
        H_j(1/2+it-z) \brackets{\frac{t}{2\pi}}^{-z} = H_j(1/2-it+z) \chi_j(1/2+it-z)\brackets{\frac{t}{2\pi}}^{-z} = \\
        H_j(1/2-it+z) \brackets{\frac{t}{2\pi}}^{z} \chi_j(1/2+it)(1+ R_j(z,t))\,.
\end{aligned}
\end{equation*}
Then
\begin{equation*}
\begin{aligned}
\frac{1}{2\pi i }\int_{(1)} H_j(1/2+it-z) \brackets{\frac{t}{2\pi}}^{-z} \frac{G(z/2)dz}{z} = 
 \chi_j(1/2+it) L_j(1/2-it) + E_j(t)\,,
\end{aligned}
\end{equation*}
where for $\re z = \sigma > 0$
\begin{equation*}
    \begin{aligned}
        E_j(t) \ll_\sigma &t^{\sigma} \int_{(\sigma)} |H_j(1/2-it+z) R_j(z,t) G(z/2)| \frac{|dz|}{|z|}\ll_\sigma \\
        &t^{\sigma} \brackets{\frac{1}{t} + \frac{\kappa_j^2}{t^2}} \max_{|y-t| \ll T^\varepsilon}|H_j(1/2+\sigma+iy)| \ll  t^{\sigma-1} \max_{|y-t| \ll T^\varepsilon}|H_j(1/2+\sigma+iy)|\,
    \end{aligned}
\end{equation*}
given $\kappa_j \ll t^{1/2}$. Taking $\sigma = \varepsilon$ and using the subconvexity bound \cite{jutila2005uniform} (though the convexity bound would be enough in this case) we have $E_j(t) \ll t^{-2/3+\varepsilon}$ and thus
\[
L_j(1/2+it)  + \chi_j(1/2+it) L_j(1/2-it)=  H_j(1/2+it)  + O \brackets{t^{-2/3+\varepsilon}}\,.
\]

The corresponding statement for $L_{j,k}$ follows similarly. For $Z_{iy}$ the computation is identical except for extra poles at $z = -1/2\pm it \pm iy$, but those have negligible contribution since
\begin{equation}\label{eq:Gbound1}
    G(1/2+it\pm i y) \ll t^{-a}
\end{equation}
for any $a> 0$ when $y \ll X \ll t$.
\end{proof}

\begin{prop}[Approximate spectral expansion II] \label{lem:bigdelta_correr}
Let $m(x,r)$ be as in \eqref{eq:mxr} and let
\[
\nu_{\delta}(y) = \brackets{1 + \frac{i}{\sinh{\pi y}}} \chi(1/2+i\delta-iy) \Xi(iy;W_1)\,,
\]
where
\[
\Xi(z;W_1) = B(1/2+z,1/2+z) \sqrt{2\pi \delta} \int_0^\infty W(\delta x) x^{-1/2-z} {}_2F_1(1/2+z,1/2+z,1+2z;-1/x)dx\,.
\]
Then for $U_\delta(x)$ defined in \eqref{def:UMd}
\[
 2 \re \sum_{n,r \ge 1}
   \frac{d(n)d(n+r)}{r} U_\delta(n/r) = 2 \re \sum_{r \ge 1} \frac{1}{r} \int_0^\infty m(x,r) U_\delta(x) dx + E_{sp} + O((T+\delta)^{1/2+\varepsilon})\,,
\]
where $E_{sp} = E_c+ E_d+E_r$, and
\begin{equation*}
    \begin{aligned}
   & E_c = \frac{1}{\pi} \re \int_{-\infty}^{\infty} \frac{\zeta(1/2+i\delta+iy)\zeta(1/2+i\delta-iy) |\zeta(1/2+iy)|^4}{|\zeta(1+2iy)|^2} \overline{\nu_{\delta}(y)} dy\,;\\
     &   E_d = \re \sum_{j=1}^{\infty} \alpha_j H_j(1/2+i\delta) H_j^2(1/2)\chi(1/2-i\delta+i\kappa_j) \overline{\nu_{\delta}(\kappa_j)}\,;\\
    &   E_r = \frac{1}{4} \re  \sum_{k=6}^{\infty} \sum_{j=1}^{\vartheta(k)} (-1)^k (\delta/2\pi)^{k-1/2} \alpha_{j,k} H_{j,k}(1/2+i\delta) H_{j,k}^2(1/2) \chi(k-i\delta)\Xi(k-1/2;W_1)\,.
    \end{aligned}
\end{equation*}
In particular, $E_{sp} \ll T^\varepsilon \brackets{\tilde{T}/\Delta^{1/2} + T^{1/2}\delta^{1/3} +  \delta^{1/2}}$.
\end{prop}
\begin{proof}
Let $\lambda = \sqrt{\delta/T} + 1$ and $X = (T+\delta)^{1+\varepsilon}/\Delta$. By Theorem \ref{thm:moto_exp} we have 
\begin{equation*}
    \sum_{n\ge 1} d(n) d(n+r) U_\delta(n/r) = \int_0^\infty m(x,r) U_\delta(x) dx + E(r;U_\delta)\,,
\end{equation*}
where $E(r;U_\delta) = e_1(r;U_\delta) + e_2(r;U_\delta) + e_3(r;U_\delta)$ and
\begin{equation}\label{eq:largedelta_error_exp}
\begin{aligned}
      & e_1(r;U_\delta) = \frac{r^{1/2}}{\pi} \int_{-\infty}^{\infty} \frac{r^{-iy}\sigma_{2iy}(r) |\zeta(1/2+iy)|^4}{|\zeta(1+2iy)|^2} \Theta(y;U_\delta) dy\,; \\
     &   e_2(r;U_\delta) = r^{1/2}\sum_{j=1}^{\infty} \alpha_j t_j(r) H_j^2(1/2) \Theta(\kappa_j;U_\delta)\,; \\
      &  e_3(r;U_\delta) = \frac{1}{4}r^{1/2} \sum_{k=6}^{\infty} \sum_{j=1}^{\vartheta(k)} (-1)^k \alpha_{j,k} t_{j,k}(r)H_{j,k}^2(1/2) \Xi(k-1/2;U_\delta)\,.
    \end{aligned}
\end{equation}
Then 
\[
\re \sum_{n,r \ge 1}
   \frac{d(n)d(n+r)}{r} U_\delta(n/r) = 2 \re \sum_{r \ge 1} \frac{1}{r} \int_0^\infty m(x,r) U_\delta(x) dx + E_1+E_2+E_3\,,
\]
where 
\[
E_j = \sum_{r\ge 1} \frac{e_j(r;U_\delta)I(\delta/2\pi r)}{r}\,.
\]
Let us start with $E_1$. Using the decomposition for $\Xi(iy;U_\delta)$ from Lemma \ref{lem:thetaxi} and the approximate functional equation from Lemma \ref{lem:Lfuncapprox} for $Z_{iy}$ together with the observation \eqref{eq:Ziyreal}, we can write
\begin{equation*}
    \begin{aligned}
        \sum_{r\ge 1} \frac{\sigma_{2iy}(r) I(\delta/2\pi r)}{r^{1/2+iy}} \Theta(y;U_\delta) = E_1^{(0)}(y) + E_1^{(1)}(y)\,,
    \end{aligned}
\end{equation*}
where
\begin{equation*}
    \begin{aligned}
       & E_1^{(0)}(y) = \re \zeta(1/2-i\delta+iy)\zeta(1/2-i\delta-iy) \nu_\delta(y)\,,\\
              &  E_1^{(1)}(y) \ll \brackets{\frac{\lambda\tilde{T}^{1/2}}{\delta(\lambda + |y|)^{1/2}} +(T+\delta)^{-1/2}}|Z_{iy}(1/2-i\delta)| + |\nu_\delta(y)|\delta^{-2/3+\varepsilon}\,,
    \end{aligned}
\end{equation*}
and
\[
\nu_\delta(y) = \brackets{1 + \frac{i}{\sinh{\pi y}}} \chi(1/2+i\delta-iy) \Xi(iy;W_1)\,.
\]
For $m=0,1$ set
\[
E_1^{(m)} = \re \frac{1}{\pi} \int_{-\infty}^{\infty} \frac{ |\zeta(1/2+iy)|^4  E_1^{(m)}(y)}{|\zeta(1+2iy)|^2} dy
\]
so that $E_1 = E_1^{(0)} + E_1^{(1)}$. By Lemma \ref{lem:thetaxi}, both sums can be restricted to $y \ll X$.

Using the bound from Lemma \ref{lem:thetaxi}, for $y \ll X$
\begin{equation}\label{eq:nudbound}
    \nu_\delta(y) \ll \frac{\tilde{T}^{1/2}(1+ 1/|y|)}{(\lambda+y)^{3/2}}\,.
\end{equation}
Together with $\zeta(1+it) \gg 1+ 1/|t|$ and $\zeta(1/2+it) \ll t^{b+\varepsilon}$, where $b = 13/84$ \cite{bourgain2017decoupling}, this yields
\begin{equation*}
\begin{aligned}
   E_1^{(0)} \ll \tilde{T}^{1/2} \delta^{2b}\int_{\lambda \ll y \ll X} \frac{|\zeta(1/2+iy)|^4}{|\zeta(1+2iy)|^2}  \frac{dy}{(\lambda+y)^{3/2}} + 
  T\delta^{2b-1/2} \int_{ y \ll \lambda}  \frac{ |\zeta(1/2+iy)|^4 }{|\zeta(1+2iy)|}  dy
   \,.
       \end{aligned}
\end{equation*}
Using $\int_0^N |\zeta(1/2+it)|^4dt \ll N^{1+\varepsilon}$,
\begin{equation*}
    \begin{aligned}
     \int_{\lambda \ll y \ll X} \frac{|\zeta(1/2+iy)|^4}{|\zeta(1+2iy)|^2} \frac{dy}{(\lambda+y)^{3/2}} \ll T^{\varepsilon} (\delta^{1/2}\lambda^{-1} + \tilde{T}^{1/2}\lambda^{-1/2}) \ll T^{1/2+\varepsilon}
    \end{aligned}
\end{equation*}
and
\begin{equation*}
    \begin{aligned}
        T\delta^{-1/2} \int_{ y \ll \lambda}  \frac{ |\zeta(1/2+iy)|^4 }{|\zeta(1+2iy)|}  dy \ll T^{1+\varepsilon}\delta^{-1/2}\lambda \ll T^{1/2+\varepsilon} \,.
    \end{aligned}
\end{equation*}
Thus
\begin{equation*}
  E_1^{(0)}\ll  T^{1/2+\varepsilon}\delta^{2b}\,.
\end{equation*}
Similarly, given $\Delta \gg T^\varepsilon(T+\delta)^{1/2}$, and $Z_{iy}(1/2-i\delta) \ll \delta^{2b+\varepsilon}$ we have 
\begin{equation*}
    \begin{aligned}
        E_1^{(1)} \ll \delta^{2b} \int_{y \ll X} \frac{ |\zeta(1/2+iy)|^4 }{|\zeta(1+2iy)|} \brackets{\frac{\lambda\tilde{T}^{1/2}}{\delta(\lambda + |y|)^{1/2}} +(T+\delta)^{-1/2} + |\nu_\delta(y)|\delta^{-1+\varepsilon}}dy \ll \\
        T^\varepsilon \delta^{2b} \brackets{\frac{\lambda\tilde{T}^{1/2}X^{1/2}}{\delta} +X(T+\delta)^{-1/2} + T^{1/2}\delta^{-2/3+\varepsilon}} \ll T^{\varepsilon}(\delta^{2b} + T^{1/2})\,.
    \end{aligned}
\end{equation*}
Therefore,
\begin{equation}\label{eq:E1total}
  E_1 =  E_1^{(0)} + T^{\varepsilon}(\delta^{2b} + T^{1/2}) \ll T^{1/2+\varepsilon}\delta^{2b}\,.
\end{equation}

Let us now move on to $E_2$. Using the decomposition from Lemma \ref{lem:thetaxi} and the approximate functional equation for $H_j$ from Lemma \ref{lem:Lfuncapprox},
\begin{equation*}
    \begin{aligned}
        \sum_{r\ge 1} \frac{t_j(r) I(\delta/2\pi r)}{r^{1/2}} \Theta(\kappa_j;U_\delta) = E_2^{(0)}(j) + E_2^{(1)}(j)
    \end{aligned}
\end{equation*}
where
\begin{equation*}
    \begin{aligned}
        & E_2^{(0)}(j) = \re   H_j(1/2-i\delta) \nu_\delta(\kappa_j)\,,\\
        & E_2^{(1)}(j)\ll \brackets{\frac{\lambda\tilde{T}^{1/2}}{\delta(\lambda + \kappa_j)^{1/2}} +(T+\delta)^{-1/2}}|L_j(1/2-i\delta)| + |\nu_\delta(\kappa_j)|\delta^{-2/3+\varepsilon}\,.
    \end{aligned}
\end{equation*}
For $m=0,1$ set
\[
E_2^{(m)} = \sum_{r \ge 1} \frac{1}{r^{1/2}} 
\sum_{j=1}^{\infty} \alpha_j t_j(r) H_j^2(1/2) E_2^{(m)}(j)
\]
so that $E_2 = E_1^{(0)} + E_1^{(1)}$.

Split $E_2^{(0)}$ into dyadic intervals. Applying the moment bounds \eqref{eq:Hjmoment} and \eqref{eq:hecke2mom} for each interval $\kappa_j \sim K$ ($K \le \kappa_j < 2K$) together with Cauchy's inequality, 
\begin{equation}\label{eq:mixedheckebound}
    \begin{aligned}
       \sum_{\kappa_j \sim K} \alpha_j |H_j(1/2+it)| H_j^2(1/2) \le \brackets{\sum_{\kappa_j \sim K}\alpha_j |H_j(1/2+it)|^2 \sum_{\kappa_j \sim K} \alpha_j H_j^4(1/2)}^{1/2} \ll \\
       K^{2+\varepsilon} + K^{1+\varepsilon}t^{1/3+\varepsilon}\,.
    \end{aligned}
\end{equation}
Using \eqref{eq:nudbound} and summing over the dyadic intervals, we obtain
\begin{equation*}
\begin{aligned}
    E_2^{(0)} \ll &\tilde{T}^{1/2}\sum_{\lambda \ll \kappa_j\ll X} \alpha_j |H_j(1/2+i\delta)| H_j^2(1/2) \kappa_j^{-3/2}  + \\
    & T \delta^{-1/2} \sum_{\kappa_j \ll \lambda}  \alpha_j |H_j(1/2+i\delta)| H_j^2(1/2)       \ll
    T^\varepsilon \brackets{\delta^{1/2} + T^{1/2}\delta^{1/3} + \tilde{T}^{1/2 + \varepsilon}X^{1/2}}\,.
    \end{aligned}
\end{equation*}
For $E_2^{(1)}$, note that \eqref{eq:mixedheckebound} still applies if $H_j(1/2+it)$ is replaced with $L_j(1/2+it)$. Then given $\Delta \gg T^\varepsilon (T+\delta)^{1/2}$ and $\delta \gg T^2/\Delta^2$ (see  \eqref{assum:largedelta} and \eqref{assum:largedeltaDQ})
\begin{equation*}
    \begin{aligned}
        E_2^{(1)} \ll \sum_{\kappa_j\ll X} \alpha_j |H_j(1/2+i\delta)| H_j^2(1/2) \brackets{\frac{\lambda\tilde{T}^{1/2}}{\delta(\lambda + \kappa_j)^{1/2}} +(T+\delta)^{-1/2}} \ll  (T+\delta)^{1/2+ \varepsilon}\,.
    \end{aligned}
\end{equation*}
Thus
\begin{equation}\label{eq:E2total}
    E_2 = E_2^{(0)}+  O((T+\delta)^{1/2+ \varepsilon}) \ll T^\varepsilon \brackets{\delta^{1/2} + T^{1/2}\delta^{1/3} + \tilde{T}\Delta^{-1/2}}\,.
\end{equation}
Finally, for 
\begin{equation*}
   E_3 =  \sum_{r \ge 1} \frac{I(\delta/2\pi r)e_3(r;U_\delta)}{r} = \frac{1}{4} \sum_{r \ge 1} \frac{1}{r^{1/2}} 
 \sum_{k=6}^{\infty} \sum_{j=1}^{\vartheta(k)} (-1)^k \alpha_{j,k} t_{j,k}(r)H_{j,k}^2(1/2) \Xi(k-1/2;U_\delta)
\end{equation*}
we use Lemma \ref{lem:thetaxik} and
\begin{equation}\label{eq:mixedheckebound2}
    \sum_{k \sim K} \sum_{j=1}^{\vartheta (k)} \alpha_{j,k} |H_{j,k}(1/2+it)|H_{j,k}(1/2)^2 \ll K^{2+\varepsilon} + K^{1+\varepsilon}t^{1/3+\varepsilon}\,,
\end{equation}
and the rest follows similarly. When $\delta < T$, we apply the the bound on $\Xi(k-1/2;U_\delta)$ from Lemma \ref{lem:thetaxik} with $\lambda \asymp 1$ together with \eqref{eq:mixedheckebound2}, then
\begin{equation*}
    \begin{aligned}
   E_3 \ll \delta^{1/2} \brackets{\frac{\delta}{T}}^5 \sum_{6 \le k \le X_2} \sum_{j=1}^{\vartheta(k)} \alpha_{j,k} |H_{j,k}(1/2+i\delta)| H_{j,k}^2(1/2) 2^{-k}\,\ll
%   \frac{\delta^{5+5/6+\varepsilon}}{(T+\alpha)^5} \ll 
   \delta^{5/6+\varepsilon} \,.
    \end{aligned}
\end{equation*}
If $\delta \ge T$, applying Lemma \ref{lem:thetaxik} with $\lambda \asymp 1+ \sqrt{\delta/T}$,
\begin{equation}\label{eq:E3total}
    \begin{aligned}
   E_3 \ll \frac{T}{\sqrt{\delta}}  \sum_{6 \le k \ll \lambda} \sum_{j=1}^{\vartheta(k)}  \alpha_{j,k} |L_{j,k}(1/2+i\delta)| H_{j,k}^2(1/2) \ll \frac{T^{1+\varepsilon}}{\sqrt{\delta}}\brackets{\frac{\delta}{T} + \sqrt{\frac{\delta}{T}}\delta^{1/3}} \ll \\
   T^\varepsilon(\delta^{1/2} + \sqrt{T} \delta^{1/3})\,.
    \end{aligned}
\end{equation}
Collecting \eqref{eq:E1total}, \eqref{eq:E2total} and \eqref{eq:E3total} the total error is
\[
E_1 + E_2 + E_3 \ll \tilde{T}^{1/2+\varepsilon}X^{-1/2} + T^{1/2+\varepsilon}\delta^{1/3} +  \delta^{1/2+\varepsilon} \ll \tilde{T}^{1+\varepsilon}/\Delta^{1/2} + T^{1/2+\varepsilon}\delta^{1/3} +  \delta^{1/2+\varepsilon}\,.
\]
In particular, this restricts us to $\delta \ll T^{3/2-\varepsilon}$. 
\end{proof}

\subsection{Explicit formula}
Finally, let us evaluate the main term given by
\begin{equation}\label{eq:bigdelta_main}
\begin{aligned}
         & \sum_{r=1}^{\infty} \frac{I(\delta/2\pi r)}{r} \int_0^{\infty} m(x,r) U_\delta( x) dx = \\
        % &  e^{\frac{\pi i}{4}-i\delta \log \frac{\delta}{2\pi e}} \sqrt{2\pi \delta}
 % \sum_{r=1}^{\infty} \frac{I(\delta/2\pi r)}{r^{1-i\delta}} \int_0^{\infty} m(x,r) \sum_{j=0}^{M-1} \frac{i^j W^{(2j)}(\delta x)}{j! 2^j} (\delta x (x+1))^jdx = \\
 & e^{\frac{\pi i}{4}-i\delta \log \frac{\delta}{2\pi e}} \sqrt{2\pi \delta}
\sum_{j\ge 0} \frac{(i\delta)^j}{j!2^j}  \int_0^{\infty} \sum_{r=1}^{\infty} \frac{I(\delta/2\pi r)m(x,r)}{r^{1-i\delta}}   W^{(2j)}(\delta x)  ( x (x+1))^jdx
 \,,
\end{aligned}
\end{equation}  
and confirm that it coincides with the Motohashi formula.

\begin{prop}\label{prop:largedeltamainod} Let
% \[
% h(z,s)=e^{2\gamma_0 s}\frac{\zeta(1+z+s)\zeta(1-z+s)}{\zeta(2+2s)}\,,
% \]
% and
\[
 H_{\delta}(t) = \frac{\partial^2}{\partial s_1 \partial s_2}\, \brackets{\frac{\zeta(1+i\delta+s_1+s_2)\zeta(1-i\delta+s_1+s_2)}{\zeta(2+2s_1+2s_2)} \brackets{\frac{te^{2\gamma_0}}{2\pi}}^{s_1}\brackets{\frac{(t+\delta)e^{2\gamma_0}}{2\pi}}^{s_2}}\bigg\vert_{s_j=0}\,.
 \]
 Then
\[
2 \re \sum_{r=1}^{\infty} \frac{1}{r} \int_0^{\infty} m(x,r) U_\delta(x) dx = \int_0^\infty W(t)  H_{\delta}(t) dt +  O\brackets{T^\varepsilon\Delta}\,.
\]
\end{prop}
\begin{proof}
Recall that
\[
\begin{aligned}
    m(x,r) = \frac{\sigma_1(r)}{\zeta(2)}(\log x \log (1+x) +\log x(1+x)(2\gamma_0-\log r)+ (2\gamma_0-\log r)^2) + \\ (\log x(1+x) +
    4\gamma_0- 2 \log r)\frac{d}{dh}\brackets{\frac{\sigma_{1+2h}(r)}{\zeta(2+2h)}}\bigg\vert_{h=0}
     + \frac{d^2}{dh^2}\brackets{\frac{\sigma_{1+2h}(r)}{\zeta(2+2h)}}\bigg\vert_{h=0}\,.
\end{aligned}
\]
Let
\[
L(x,s) := \sum_{r=1}^{\infty} \frac{m(x,r)}{r^s}, \quad L_0(x) := \sum_{r=1}^{\infty} \frac{m(x,r)}{r^{1-i\delta}}I(\delta/2\pi r)\,.
\]
Then $ \sum_{r=1}^{\infty} \frac{1}{r} \int_0^{\infty} m(x,r) U_\delta(x) dx = M_0 + M_1$, where
\begin{equation*}
    \begin{aligned}
   & M_0 =  e^{\frac{\pi i}{4}-i\delta \log \frac{\delta}{2\pi e}} \sqrt{2\pi \delta} \int_0^{\infty} L_0(x) W(\delta x) dx\,,\\
   & M_1 = e^{\frac{\pi i}{4}-i\delta \log \frac{\delta}{2\pi e}} \sqrt{2\pi \delta}
\sum_{j\ge 1} \frac{(i\delta)^j}{j!2^j}  \int_0^{\infty} L_0(x)   W^{(2j)}(\delta x)  ( x (x+1))^jdx\,.
    \end{aligned}
\end{equation*}
Using the integral definition of $I(w)$ \eqref{def:Iw}
\begin{equation}\label{eq:mxrseries}
    L_0(x) = \frac{1}{2\pi i} \int_{(2)} L(x, 1- i\delta + s) \brackets{\frac{\delta}{2\pi}}^s G(s/2)\frac{ds}{s}\,.
\end{equation}
For $\re z > 1$ we have $L(x, 1+z) = \lim_{h\to 0} L_h(x,1+z)$, where
\begin{eqnarray*}
        L_h(x,1+z) = \frac{\zeta^2(1+ h)}{\zeta(2+2h)} \zeta(1+z)\zeta(z-2h)x^h (1 + x)^h + \frac{\zeta^2(1- h)}{\zeta(2-2h)} \zeta(1+z-2h)\zeta(z) + \\ \frac{\zeta(1+h)\zeta(1-h)}{\zeta(2)}\zeta(1+z-h)\zeta(z-h)(x^h + (1+ x)^h) \,.
\end{eqnarray*}
As $L(x,1-i\delta + s)$ contains $\zeta(1-i\delta + s)$, $\zeta(s-i\delta)$ and their derivatives, its poles are at $s = i\delta$ and $s = 1+i\delta$. Moving the line of integration to $\re s = 0$ with semi-circles around $s = 0$ and $s=i\delta$ i.e. $\mathcal{L}= \{s:\ \re s = 0,\ |s-i\delta| > 1,\ |s| > 1\}$,
\[
\begin{aligned}
    \frac{1}{2\pi i} \int_{(2)} L(x, 1- i\delta + s) \brackets{\frac{\delta}{2\pi}}^s G(s/2)\frac{ds}{s} = 
    \frac{1}{2}\sum_{s \in \{0,i\delta\}} \res_s L(x, 1- i\delta + s) \brackets{\frac{\delta}{2\pi}}^s G(s/2) + \\
    \res_{s = 1+i\delta} L(x, 1- i\delta + s) \brackets{\frac{\delta}{2\pi}}^s G(s/2) + \frac{1}{2\pi i} \int_{\mathcal{L}} L(x, 1- i\delta + s) \brackets{\frac{\delta}{2\pi}}^s G(s/2)\frac{ds}{s} = \\
    \frac{1}{2} L(x, 1- i\delta) + \frac{1}{2\pi i} \int_{\mathcal{L}} L(x, 1- i\delta + s) \brackets{\frac{\delta}{2\pi}}^s G(s/2)\frac{ds}{s} + O(e^{-\delta^2/4Q})\,.
\end{aligned} 
\]
Here we used that $G^{(l)}(i\delta) \ll e^{-\delta^2/4Q}$, and thus
\[
\res_{s = i\delta,1+i\delta} L(x, 1- i\delta + s) \brackets{\frac{\delta}{2\pi}}^s G(s/2) \ll e^{-\delta^2/4Q}\,.
\]
For the $j = 0$ term in \eqref{eq:bigdelta_main}, let us bound
\[
\re \frac{1}{2\pi i}e^{\frac{\pi i}{4}-i\delta \log \frac{\delta}{2\pi e}}\int_{\mathcal{L}} L(x, 1- i\delta + iy) \brackets{\frac{\delta}{2\pi}}^{iy} G(iy/2)\frac{dy}{y}\,.
\]
For $y \ll T^\varepsilon Q^{1/2} \ll T^\varepsilon$
\[
e^{\frac{\pi i}{4}-i\delta \log \frac{\delta}{2\pi e}} \chi(i(y-\delta)) (\delta/2\pi)^{iy} = (\delta/2\pi)^{1/2} (1 + O(T^\varepsilon/|\delta|))
\]
is real up to an error of size $O(T^\varepsilon/\delta^{-1/2}))$. Then applying the functional equation of the Riemann zeta \eqref{eq:feqzeta} to every $\zeta$ in the definition of $L(x, 1 - i\delta +iy)$ for $y < 0$, we see that
\[
\re \frac{1}{2\pi i}e^{\frac{\pi i}{4}-i\delta \log \frac{\delta}{2\pi e}}\int_{\mathcal{L}} L(x, 1- i\delta + iy) \brackets{\frac{\delta}{2\pi}}^{iy} G(iy/2)\frac{dy}{y} \ll T^{\varepsilon} \log^2 x\,,
\]
and 
\begin{equation}\label{eq:reL0x}
   2 \re e^{\frac{\pi i}{4}-i\delta \log \frac{\delta}{2\pi e}} L_0(x) = \re L(x,1-i\delta) + O(T^\varepsilon\log^2x) = \sqrt{\delta}H_\delta(\delta x) + O(T^\varepsilon\log^2x)\,
\end{equation}
giving
\[
2\re M_0 = \int_0^\infty W(t)  H_{\delta}(t) dt + O(T^{1+\varepsilon}\delta^{-1/2})= \int_0^\infty W(t)  H_{\delta}(t) dt + O(T^{\varepsilon}\Delta)\,.
\]

To bound $M_1$ it suffices to use the trivial bound on $L_0(x)$. Using $|\chi(1/2+ \sigma + it)| \asymp |t|^{-\sigma}$ and $\zeta^{(l)}(1+i(1+t)) \ll |t|^{\varepsilon}$, we have
\begin{equation} \label{eq:L0xtrivial}
  L_0(x) =   \frac{1}{2\pi i} \int_{(2)} L(x, 1- i\delta + s) \brackets{\frac{\delta}{2\pi}}^s G(s/2)\frac{ds}{s} \ll \delta^{1/2}T^\varepsilon  \log^2 x\,. 
\end{equation}
Then we can bound the contribution of $j \ge 1$ by
\[
\begin{aligned}
     M_1 \ll \sqrt{\delta}
\sum_{j\ge1}\frac{(i\delta)^j}{j!2^j}  \int_0^{\infty} L_0(x) W^{(2j)}(\delta x)(x(x+1))^j dx \ll
T^\varepsilon
\sum_{j\ge 1} \frac{\tilde{T}^{2j}}{j!(2\delta)^{j}}  \int_0^{\infty}  |W^{(2j)}(t)| \,dt \ll \\
T^\varepsilon
\sum_{j\ge 1}\frac{\tilde{T}^{2j}}{j!(2\delta)^{j}\Delta^{2j}}  \int_{|t-T_j|\le\Delta}   1 dt
\ll \frac{\tilde{T}^{2+\varepsilon}}{\Delta \delta} \ll T^\varepsilon \Delta\,,
\end{aligned}
\]
and the statement of the proposition follows.
\end{proof}

Collecting error terms from Lemma \ref{lem:sweights}, Proposition \ref{thm:diag}, \ref{lem:bigdelta_statphase}, Proposition \ref{lem:bigdelta_correr} and Proposition \ref{prop:largedeltamainod} yields the following theorem. 

\begin{thm}\label{thm:explicit_large} 
Let $T \ge 10$, and let $0 \le \delta$ be such that $\delta \gg T^{\varepsilon}(T/\Delta)^2$. Let $Q$ and $\Delta$ be parameters satisfying
\[
Q = T^\varepsilon, \quad (T+\delta)^{1/2+\varepsilon} \ll \Delta \ll T^{1-\varepsilon}\,.
\]
Let $W(t)$ be as defined in \eqref{eq:dcomp}. Let
\[
\begin{aligned}
Q_2(t;\delta) = 2\re  \res_{s= 0,i\delta} \frac{\zeta^4(1+s)}{\zeta(2+2s)}  \frac{(t/2\pi)^s}{s-i\delta} + 
   \frac{\partial^2}{\partial s_1 \partial s_2}\, \brackets{h(i\delta,s_1+s_2) \brackets{ \frac{t}{2\pi} }^{s_1}\brackets{\frac{t+\delta}{2\pi} }^{s_2}} \bigg\vert_{s_j=0}\,,
   \end{aligned}
\]
where
\[
h(z,s) = \frac{e^{2\gamma_0 s}}{\zeta(2+2s)} \brackets{\zeta(1+z + s)\zeta(1-z+s) - \frac{2s\zeta(1+2s)}{s^2-z^2}}\,.
\]
Then
\begin{equation*}
    \begin{aligned}
       \int_0^\infty W(t) |\zeta(1/2+it)|^2 |\zeta(1/2+it+i\delta)|^2 = 
    \int_0^\infty W(t)Q_2(t;\delta) dt + E_c+ E_d+ E_r+ E_0\,,
    \end{aligned}
\end{equation*}
where $E_c, E_d, E_r$ are as defined in Proposition \ref{lem:bigdelta_correr} and $E_0 \ll T^\varepsilon\brackets{\Delta + \delta^{1/2} + T\delta^{-2} }$.
\end{thm}

Choosing $\Delta = T^{\varepsilon}( T^{2/3} + T\delta^{-1/2} + T^{1/2}\delta^{1/3})$ gives the optimal error term of size
\[
T^\varepsilon\brackets{T\delta^{-1/2} +\delta^{1/2} + T^{1/2}\delta^{1/3}}\,.
\]
This, in particular, implies Theorem \ref{thm:A} for $T^{2/3-\varepsilon}\ll \delta \ll T^{3/2-\varepsilon}$ and concludes the proof of the main theorem.

\section{Moments of moments}\label{sec:moms}

Let $c = c(T)\ll T^{1/2}$ be either fixed or depend on $T$. Let $g_0(x)$ be a normalised even rapidly decaying weight function concentrating on the interval $[-1,1]$ of total mass $1$. Set $g(x) = c^{-1}g_0(x/c)$, and consider
\begin{equation*}
\mM_{2,2}(T;g) = \int_0^T \brackets{\int_{-\infty}^{+\infty}  g(h) |\zeta(1/2+it+ih)|^{2}dh}^2 dt\,.    
\end{equation*}

\subsection{Computing $(2,2)$-moment of moment}
Let $x \in \RR$, and set
\begin{equation*}
    D(ix,t) = 2\re \sum_{s \in \{0,ix\}} \res_s \frac{\zeta^4(1+s)}{\zeta(2+2s)}  \frac{(t/2\pi)^s}{s-ix}\,;\ \ 
    H(ix,t) = \frac{d^2}{ds^2}\, \brackets{h(ix,s) \brackets{ \frac{t}{2\pi} }^{s}} \bigg\vert_{s=0}\,,
\end{equation*}
where
\begin{equation*}
    h(z,s) := \frac{e^{2\gamma_0 s}}{\zeta(2+2s)} \brackets{\zeta(1+z + s)\zeta(1-z+s) - \frac{2s\zeta(1+2s)}{s^2-z^2}}\,.
\end{equation*}
Let
\[
I_1(t,g) = \iint D(i(x-y),t) g(x) g(y)dx\, dy,\quad I_2(t;g) = \iint H(i(x-y),t) g(x) g(y)dx\, dy\,.
\]
Then by Theorem \ref{thm:A}
\[
\mM_{2,2}(T;g) =  \int_0^T (I_1(t,g) + I_2(t,g))dt + O(T^{2/3+\varepsilon})\,.
\]

\begin{lem}
Let $T \ge 10$, and set $N = \log T/2\pi$. Let $c = c(T) > 0$ be either fixed or depend on $T$, and let $g(x) = c^{-1} g_0(x/c)$, where $g_0(x)$ is a weight function rapidly decaying outside of $[-1,1]$. Let
\[
P_3 (x) := \res_{z=0} \frac{L(z)}{1+z}e^{zx}\,.
\]
Then
\[
\begin{aligned}
     \int_0^T I_1(t;g)dt = 
     2 \int_0^{N}\hat{g}^2 (w/2\pi) P_3(N - w)dw + O(T \max_{y \ge N/6} |\hat{g}(y/2\pi)|^2 + T^{2/3+\varepsilon})\,.
\end{aligned}
\]    
\end{lem}
\begin{proof}
Recall the classical result $\sum_{n \le x} d^2(n)/n = M(x) + E(x)$, where 
\[
M(x) = P_4(\log x) = \res_{s=0} \frac{\zeta^4(1+s)x^s}{s\zeta(2+2s)}
\]
and $E(x) \ll x^{-1/2+\varepsilon}$. Then integrating by parts,
    \begin{equation*}
        \begin{aligned}
            \sum_{n \le t/2\pi} \frac{d^2(n)}{n}  \hat{g}^2 \brackets{\frac{1}{2\pi} \log \frac{t}{2\pi n}} = \int_1^{t/2\pi} M^\prime(x) \hat{g}^2 \brackets{\frac{1}{2\pi} \log \frac{t}{2\pi x}} dx + R(t)\,,
        \end{aligned}
    \end{equation*}
where
\[
    R(t) = \int_1^{t/2\pi} E(x) \frac{d}{dx}\hat{g}^2 \brackets{\frac{1}{2\pi} \log \frac{t}{2\pi x}}  dx\,.
\]
Splitting the integral in $x$ according to $x \ll T^b$ and $x \gg T^b$ and using $E(x) \ll x^{-1/2+\varepsilon}$ we see that 
\begin{equation*}
    \int_0^T R(t) dt \ll T^{1-b/2+\varepsilon} + T  \max_{w \gg T^{1-b}}\big\vert\hat{g}((\log w)/2\pi) \big\vert^2\,.
\end{equation*}
Choosing $b = 2/3$, we then obtain the statement of the lemma.
\end{proof}

\begin{lem}Let $T \ge 10$, and set $N = \log T/2\pi$. Let $c = c(T) > 0$ be either fixed or depend on $T$. Let $g(x) = c^{-1} g_0(x/c)$, where $g_0(x)$ is a weight function rapidly decaying outside of $[-1,1]$. 
% Let
% \[
% h(z,s) = \frac{\zeta(1-z+s)\zeta(1+z +s) -\frac{2s\zeta(1+2s)}{s^2- z^2}}{\zeta(2+2s)} \,.
% \]
Then
\[
\begin{aligned}
\int_0^T I_2(t;g) dt = 
% \int_{-\infty}^{+\infty} \int_{-\infty}^{+\infty} \frac{d^j}{ds^j}\brackets{\frac{h(i(x-y),s)}{1+s} e^{(N+2\gamma_0)s}}\bigg\vert_{s=0} g(x) g(y) dx dy = \\
T(a_2(g) + 2 a_1(g) (N +2\gamma_0)+ a_0(g) (N +2\gamma_0)^2) + O(T^{\varepsilon}c)\,,
\end{aligned}
\]
where 
\[
a_j(g) =  \int_{-\infty}^{+\infty} \frac{d^j}{ds^j}\brackets{\frac{h(it,s)}{1+s}}\bigg\vert_{s=0} (g\ast g)(t) dt\ll 1\,.
\]
In particular, if $c = c(T) \to \infty$
\[
\begin{aligned}
       \int_0^T I_2(t;g) dt= \int_0^T \brackets{N + 2\gamma_0}^2 (1+ O(1/c)) dt\,.
\end{aligned}
\]
\end{lem}

\begin{proof}
Computing $I_2(t;g)$ boils down to computing integrals of the form
\begin{equation*}
\int_{-\infty}^{+\infty} \int_{-\infty}^{+\infty} h^{(j)}(i(x-y),0) g(x) g(y)dx dy = \int h^{(j)}(it,0) (g \ast g)(t)dt\,.
\end{equation*}
When $c \ll 1$, this is trivial, so assume $c(T) \to \infty$ as $T \to \infty$. Recall the approximate functional equation \cite[Theorem 4.11]{titchmarsh1986theory} for $\zeta(1+it)$ when $|t| \ll x$
\[
\zeta(1+it) = \sum_{n \le x} \frac{1}{n^{1+it}} - \frac{x^{-it}}{it}+ O(1/x)\,,
\]
and expand $h(i\delta,0)$ into the corresponding truncated Dirichlet series. Then
\[
    \begin{aligned}
            \int h(it,0) (g \ast g)(t)dt = &
             \frac{1}{\zeta(2)}\sum_{n,m \ll c} \frac{1}{nm}  \iint  \brackets{\frac{m}{n}}^{it}(g \ast g)(t)dt + O(c^{-1})= \\
             & \frac{1}{\zeta(2)}\sum_{n,m \ll c} \frac{1}{nm} \hat{g}^2 \brackets{\frac{1}{2\pi} \log \frac{m}{n}} + O(c^{-1})\,.
    \end{aligned}
\]
When $m = n$, $\sum_{n \ll c} \frac{1}{n^2} = \zeta(2) + O(1/c)$, and when $m \neq n$ 
\[
\begin{aligned}
\sum_{m \neq n \ll c}\frac{1}{mn} \hat{g}^2 \brackets{\frac{1}{2\pi} \log \frac{m}{n}} \ll 
\frac{1}{c^2} \sum_{m \neq n \ll c}\frac{1}{mn|\log m/n|^2} \ll \frac{1}{c}\,.
\end{aligned}
\]
Thus $a_0(g) = 1 + O(c^{-1})$ independent of $g_0$. The computation for the derivatives ($j = 1,2$) goes similarly producing $a_j(g) = j(-1)^j+O(c^{-1}\log ^j c)$, also independent of $g_0$. The statement of the lemma then follows.
\end{proof}

Combining the two lemmas above, we obtain the following general statement.

\begin{cor}\label{cor:mommain} Let $T \ge 10$, and set $N = \log \frac{T}{2\pi}$. Let $c>0$ be either constant, or depend on $T$ such that $c = c(T) \ll T^{1/2-\varepsilon}$, and let $g(x) = c^{-1} g_0(x/c)$, where $g_0(x)$ is a normalised even weight function rapidly decaying outside of $[-1,1]$. Let
\[
\begin{aligned}
    &P_3(x) = \res_{s=0} \frac{\zeta^4(1+s)e^{sx}}{\zeta(2+2s)(1+s)} = b_0 x^3+ b_1x^2+b_2x+b_3\,,\\
   &  h(z,s) = \frac{e^{2\gamma_0 s}}{\zeta(2+2s)} \brackets{\zeta(1+z + s)\zeta(1-z+s) - \frac{2s\zeta(1+2s)}{s^2-z^2}}\,,\\
     & a_j(g) =\int_{-\infty}^{+\infty} \frac{d^j}{ds^j}\brackets{\frac{h(it,s)}{1+s}}\bigg\vert_{s=0} (g\ast g)(t)dt\,.
\end{aligned}
\]
Then $\mM_{2,2}(T;g) = T(\overline{D}(T;g) + \overline{OD}(T;g)) + O(T\max_{w \le N/6} |\hat{g}(w/2\pi)|^2 + T^{2/3+\varepsilon})$, where
\[
\begin{aligned}
& \overline{D}(T;g)  = 2\int_{0}^{N} \hat{g}^2(y/2\pi) P_3(N  - y)dy\,,\\
& \overline{OD}(T;g) = a_0(g) (N +2\gamma_0)^2 +2 a_1(g) (N +2\gamma_0) + a_2(g) \,.
\end{aligned}
\]
In particular, 
\[
\mM_{2,2}(T;g) \sim \frac{1}{\pi c} \log^3 T + a_0(g) \log^2 T\,,
\]
which agrees with \eqref{eq:baileykeating}.
Further, if $c(T) \to \infty$, $\overline{OD}(T;g) = ((N+ 2\gamma_0-1)^2+1)(1+ O(1/c)) \sim \mathcal{M}_2(T)^2$ independent of $g$. 
\end{cor}

\begin{rmrk} When $c(T) \gg T^{1/2+\varepsilon}$ computing $\mM_{2,2}(T;g)$ with a power error term does not require any new results. In this case, the classic error term $O(T^{1/2+\varepsilon})$ for the second moment of the Riemann zeta \cite[Theorem 7.4]{titchmarsh1986theory} would be sufficient.
\end{rmrk}

\subsection{Structure of $\mM_{2,2}(T;g)$}

When $c$ grows sufficiently slowly, structure of $\mM_{2,2}(T;g)$ depends on the rate of decay of $\hat{g}$, which in turn depends on both the rate of growth of $c(T)$ and the rate of decay of $\hat{g}_0$. While the off-diagonal term does not cause any problems, the diagonal term here may looks a bit differently depending on the choice of $g$.

It follows from the proof of Proposition \ref{thm:diag} that we can write
\begin{equation*}
    D(i\delta,t) = 2 \re \sum_{n \le t/2\pi} \frac{d^2(n)}{n} \brackets{\frac{t}{2\pi n}}^{-i\delta}\,.
\end{equation*}
and, given that $g$ is even,
\begin{equation*}
   \int_0^T \int_{-\infty}^{\infty} D(ix,t) (g\ast g) (x) dx dt  = \int_{2\pi}^{T} \sum_{n \le t/2\pi} \frac{d^2(n)}{n}  \hat{g}^2 \brackets{\frac{1}{2\pi} \log \frac{t}{2\pi n}} dt \,.
\end{equation*}
% Further, given that $g$ is even, we have
% \begin{equation*}
%     \iint  \brackets{\frac{t}{2\pi n}}^{-i(x-y)} g(x) g(y) dx dy = \hat{g}^2 \brackets{\frac{1}{2\pi} \log \frac{t}{2\pi n}}\,.
% \end{equation*}
Now if $\hat{g}$ was decaying sufficiently fast, by Perron's formula the above equals
\begin{equation*}
   \frac{1}{2\pi i}\int_{(1)} L(1+z) \brackets{\frac{T}{2\pi}}^z \mathcal{L}(\hat{g}^2)(z) \frac{dz}{z+1}\,,
\end{equation*}
where $\mathcal{L}(\hat{g}^2)(z) = \int_0^\infty \hat{g}^2(y/2\pi) e^{-zy} dy$. Then moving the line of integration to $\re z = -1/2+\varepsilon$ and computing the residue at $z = 0$ yields the explicit expression for the averaged diagonal term. However, $\mathcal{L}(\hat{g}^2)(z)$ may not have an analytic continuation past $\re z = 0$, and, in fact, $g(x) = \mathbb\{|x| \le 1/2\}$ produces branch cuts as
\[
    \mathcal{L}\hat{g}^2(z) = 2z \log z -(z+i) \log \brackets{z+ i}- (z-i) \log \brackets{z- i}\,.
\]
So, at least naively, it is natural to expect the error term to depend on the rate of decay of $g$, or equivalently, on the domain of analyticity of $\mathcal{L}\hat{g}^2$.

\begin{cor} In the notation of Corollary \ref{cor:mommain} for $g(x) = \II\{|x| \le c\}/2c$
% Let $T \ge 10$, and set $N = \log \frac{T}{2\pi}$. Let $c>0$ be either fixed or depend on $T$ such that $c = c(T) \ll T^{1/2-\varepsilon}$. Let , and let $g(x) = \II\{|x| \le c\}/2c$. Let
% \[
% \begin{aligned}
%     &P_3(x) = \res_{s=0} \frac{\zeta^4(1+s)e^{sx}}{\zeta(2+2s)(1+s)} = b_0 x^3+ b_1x^2+b_2x+b_3\,,\\
%      & a_j(g) =\int_{-\infty}^{+\infty} \frac{d^j}{ds^j}\brackets{\frac{h(it,s)}{1+s}}\bigg\vert_{s=0} (g\ast g)(t)dt \ll 1\,.
% \end{aligned}
% \]
\[\mM_{2,2}(T;g) = T(\overline{D}(T;g) + \overline{OD}(T;g)) + O(T (c\log T)^{-2} + T^{2/3+\varepsilon})\,,
\]
where
\[
\begin{aligned}
  & \overline{D}(T;g)  = \frac{\pi}{c}P_3(N) -\frac{1}{c^2} P_3^\prime(N)(\log (2cN) +\gamma_0)+  \frac{1}{2c^2}(3b_0N^2-2b_2) - \frac{b_0}{4c^4} - \frac{b_3}{c^2N}\,,\\
& \overline{OD}(T;g) = a_0 (N +2\gamma_0)^2 +2 a_1 (N +2\gamma_0) + a_2\,.
\end{aligned}
\]
\end{cor}

\begin{proof}
If $g(x) = (2c)^{-1} \II\{|x| \le c)\}$, then
\[
 \hat{g}(x/2\pi) = \frac{\sin(cx)}{cx}\,.
\]
Substituting this $\hat{g}(x/2\pi)$ into $\overline{D}(T;g)$,
\[
\begin{aligned}
\int_0^{N} \hat{g}^2(y/2\pi) P_3(N  - y) dy = \frac{1}{c} \brackets{\frac{\pi}{2} - \frac{1}{2cN}}P_3(N) - \frac{1}{2c^2}P_3^\prime(N)(\log (2cN) + \gamma_0) +\\ \frac{N}{4c^2}P_3^{\prime\prime}(N) - \frac{N^2}{24c^2}P_3^{\prime\prime\prime}(N) - \frac{1}{48c^4}P_3^{\prime\prime\prime}(N) + O(1/c^6N^2)= \\
\frac{\pi}{2c}P_3(N) -\frac{\log (2cN)}{2c^2} P_3^\prime(N) + \frac{1}{c^2}Q_2(N) + \frac{B_1}{c^4} + \frac{B_2}{c^2N} + O(1/c^6N^2)\,,
\end{aligned}
\]
where $Q_2(x)$ is a non-trivial polynomial of degree $2$ independent of $T$ and $c$, and $B_1,B_2 \neq 0$ are absolute constants. 
This implies that
\[
\begin{aligned}
\mathcal{M}_{2,2}(T;g) = \frac{\pi}{c}P_3(N) -\frac{1}{2c^2} P_3^\prime(N)\log (2cN) +  \frac{1}{c^2}Q_2(N)  + \frac{B_1}{c^4} + \frac{B_2}{cN} +\\ 
   \sum_{j=0}^2 a_{2-j}(g) (N + 2\gamma_0)^j +  O\brackets{\frac{1}{N^2}} \sim \frac{1}{\pi c} \log^3 T + a_0(g) \log^2 T\,.
\end{aligned}
\]
Note that the $\log N$-factor comes from
\[
\int_0^N y \hat{g}^2(y/2\pi) dy = \frac{1}{c^2} \int_0^N \frac{\sin^2(cy)dy}{y}\,,
\]
and the error term
\[
\frac{B_2}{c^2N} + O\brackets{\frac{1}{c^2N^6}} + O(\max_{w \gg \log T} |\hat{g}(w)|^2) = \frac{B_2}{c^2N} + O\brackets{\frac{1}{c^2N^2}} 
\]
is of course not polynomial in $T$.
\end{proof}

 Both of the issues vanish if $g_0(x)$ is an appropriately chosen weight function (for example, analytic on a vertical strip).

\begin{cor}
In the notation of Corollary \ref{cor:mommain}, if $g(x) = c^{-1}g_0(x/c)$, where $g_0$ is a normalised smooth even function such that $\hat{g}_0(y/2\pi) \ll e^{-Ay}$ for some $A > 0$. Then
\[
\begin{aligned}
\overline{D}(T;g) = \sum_{j=0}^3\frac{(-1)^j}{c^{j+1} j!} P_3^{(j)}(N)\int_0^\infty y^j \hat{g}_0^2(y/2\pi) dy + O(N^3 e^{-c AN})\,,
\end{aligned}
\]
In particular, $\mathcal{M}_{2,2}(T;g) \sim \frac{\pi}{c} \|\hat{g}_0\|_2^2 N^3 = \frac{\pi}{c} \|g_0\|_2^2 N^3$, and, if $c$ is fixed, $\mathcal{M}_{2,2}(T;g)$ is a polynomial of degree $3$ with a power error term. 
\end{cor}
 
Choosing $c = \pi$ gives Corollary \ref{cor:mompiind} and \ref{cor:mompian}.

%  In particular, by the Paley-Wiener theorem this implies $\hat{g}_0(y) \ll e^{-Ay}$ for some $A > 0$, thus for $g(x) = c^{-1}g_0(x/c)$
% \[
% \begin{aligned}
% \int_0^{N} \hat{g}^2(y/2\pi) P_3(N  - y) dy = \sum_{j=0}^3\frac{(-1)^j}{c^{j+1} j!} P_3^{(j)}(N)\int_0^\infty y^j \hat{g}^2(y/2\pi) dy + O(N^3 e^{-4\pi c AN})\,.
% \frac{b_0(g_0)}{c} P_3(N) - \frac{b_1(g_0)}{c^2} P_3^\prime(N)  + \frac{b_2(g_0)}{2c^3} P_3^{\prime\prime}(N) - \frac{b_3(g_0)}{6c^4} P_3^{\prime\prime\prime}(N)
% \frac{b_0(g_0)}{c} P_3(N) \int_0^\infty \hat{g}^2_0(y/2\pi) dy - \frac{1}{c^2} P_3^\prime(N) \int_0^\infty y \hat{g}^2_0(y/2\pi) dy + 
% \end{aligned}
% \]

% As $c(T)$ grows, if $c(T) \gg \log T$, the off-diagonal term 
% \[
% \int_0^T I_2(t,g) dt \sim a_0(g) \log^2 \frac{T}{2\pi} \sim \log^2 \frac{T}{2\pi}
% \]
% becomes the main term. This is of course consistent with the probabilistic interpretation
% \begin{equation*}
% \begin{aligned}
% \EE [XY] =  \frac{1}{T} \int_0^T |\zeta(1/2+it+ih_1)|^{2}|\zeta(1/2+it+ih_2)|^{2} dt \approx \\ \brackets{\frac{1}{T} \int_0^T |\zeta(1/2+it+ih)|^{2} dt}^2 = \EE [X] \EE [Y]\,
% \end{aligned}
% \end{equation*}
% once random variables $X = |\zeta(1/2+it+ih_1)|^{2}$ and $Y = |\zeta(1/2+it+ih_2)|^{2}$ are uncorrelated.

\bibliographystyle{siam.bst}
\bibliography{main.bib}{}

\Addresses

\end{document}